\documentclass[10pt]{article}%
\usepackage{amsfonts}
\usepackage{amssymb,amsmath,amsthm, amsfonts}

\providecommand{\U}[1]{\protect \rule{.1in}{.1in}}
\newtheorem{theorem}{Theorem}[section]

\newtheorem{definition}{Definition}[section]

\newtheorem{lemma}{Lemma}[section]

\newtheorem{proposition}{Proposition}[section]
\newtheorem{remark}{Remark}[section]

\def\esssup{\hbox{\rm ess$\,$\rm sup$\,$}}
\def\esssup{\mathop{\rm esssup}}
\def\essinf{\mathop{\rm essinf}}
\def\sup{\mathop{\rm sup}}
\makeatletter
   
   \@addtoreset{equation}{section}
\makeatother

\begin{document}

\title{Stochastic differential games for fully coupled FBSDEs with jumps {\footnote {The work has been supported by the NSF of
P.R.China (No. 11071144, 11171187, 11222110), Shandong Province
(No. BS2011SF010, JQ201202), SRF for ROCS (SEM), supported by Program for New Century Excellent Talents in University (NCET, 2012), 111 Project (No.
B12023).}}}
\author{Juan Li\\{\small School of Mathematics and Statistics, Shandong University, Weihai,
Weihai 264209, P. R. China.}\\{\small \textit{E-mail: juanli@sdu.edu.cn}}\\Qingmeng Wei\\{\small School of Mathematics, Shandong University, Jinan 250100, P. R.
China.}\\{\small \textit{E-mail: qingmengwei@gmail.com}}}

\date{January 29, 2013}
\maketitle

\bigskip

\noindent \textbf{Abstract.} This paper is concerned with
stochastic differential games (SDGs) defined through fully coupled
forward-backward stochastic differential equations (FBSDEs) which
are governed by Brownian motion and Poisson random measure. For SDGs, the upper and the lower
value functions are defined by the controlled fully coupled FBSDEs
with jumps. Using a new transformation introduced in \cite{BLH}, we
prove that the upper and the lower value functions are
deterministic. Then, after establishing the dynamic programming
principle for the upper and the lower value functions of this SDGs,
we prove that the upper and the lower value functions are the
viscosity solutions to the associated upper and the lower
Hamilton-Jacobi-Bellman-Isaacs (HJBI) equations, respectively.
Furthermore, for a special case (when $\sigma,\ h$ do not depend on
$y,\ z,\ k$), under the Isaacs' condition, we get the existence of
the value of the game.

\bigskip

\noindent \textbf{Keyword.}  Fully coupled FBSDEs with
jumps, stochastic differential game, Hamilton-Jacobi-Bellman-Isaacs
equation, value function,  stochastic backward semigroup, dynamic
programming principle,  viscosity solution

\section{{\protect \large {Introduction}}}

General nonlinear backward stochastic differential
equations (BSDEs, for short) in the framework of Brownian motion
were first introduced by Pardoux, Peng in \cite{PaPe1}. They got the
uniqueness and the existence theorem for nonlinear BSDEs under
Lipschitz condition. Since then, the theory of BSDEs has been
studied widely, namely in stochastic control (see Peng \cite{Pe2}),
finance (see El Karoui, Peng and Quenez \cite{ELPeQu}), and the
theory of partial differential equations (PDEs, for short) (see
Pardoux, Peng \cite{PaPe2}, Peng \cite{Pe3}, etc). Related tightly
with the BSDE theory, the theory of fully coupled forward-backward
stochastic differential equations (FBSDEs, for short) has shown a
dynamic development. Fully coupled FBSDEs driven by Brownian motion
are encountered in the optimization problem when applying stochastic
maximum principle. Also, in finance, fully coupled FBSDEs are often
used when considering problems with the large investors, see
\cite{CM, MY}. On one hand, for the existence and uniqueness of
solutions of fully coupled FBSDEs driven by Brownian motion, the
reader can refer to Antonelli \cite{An}, Delarue \cite{D}, Hu, Peng
\cite{HP}, Ma, Protter and Yong \cite{MPY}, Ma, Wu, Zhang and Zhang \cite{Ma-Wu-Zhang-Zhang}, Ma, Yong \cite{MY},
Pardoux, Tang \cite{PaT}, Peng, Wu \cite{PW}, Yong \cite{Y, Y1},
Zhang \cite{Z}, etc. Pardoux, Tang \cite{PaT} associated fully
coupled FBSDEs driven by Brownian motion (without controls and
$\sigma$\ doesn't depend on $z$)\ with quasilinear parabolic PDEs,
and gave an existence result for viscosity solution. Wu, Yu
\cite{WY, WY1} proved the existence of a quasilinear PDEs with the
help of fully coupled FBSDEs driven by Brownian motion when $\sigma$
depends on $z$, but their stochastic systems are without controls.
Recently, Li, Wei \cite{LW} studied the stochastic optimal control
problems of fully coupled FBSDEs driven by Brownian motion in two
cases: (i) the diffusion coefficient $\sigma$ depends on $z$, i.e.,
depends on the second component of the solution $(Y,Z)$ of the BSDE
and does not depend on the control $u$; (ii) $\sigma$ depends on the
control $u$, but does not depend on $z$. They also proved some new
estimates for fully coupled FBSDEs on small time interval which were
used for the proof of the viscosity solution. Li \cite{L} studied
the general case, that is,  $\sigma$ depends on $z$ and the control
$u$ at the same time.

BSDEs with Poisson random measure were first discussed by Tang, Li
\cite{TL}. Later, Barles, Buckdahn and Pardoux \cite{BBP} proved
that the solutions of the BSDEs driven by a Brownian motion and a
Poisson random measure provide the viscosity solutions for the
associated system of parabolic integral-partial differential
equations. In \cite{LP}, using Peng's BSDE approach, Li, Peng
studied the stochastic control theory for BSDE with jumps.

On the other hand, as concerns stochastic differential games (SDGs,
for short), two-player zero-sum SDGs of the type of strategy against control, they were first studied by Fleming,
Souganidis~\cite{FS} in 1989. In their paper under the Isaacs' condition the lower and the upper
value functions of the game coincide, satisfy the dynamic programming
principle (DPP, for short), and they are
the unique viscosity solution of the associated Bellman-Isaacs
equation. Since then there are a lot of works about
this topic, such as, Buckdahn, Li \cite{BL}, they gave a more general but also a more direct approach than that in \cite{FS}, and also  Buckdahn, Cardaliaguet and Rainer \cite{BCR},
 Buckdahn, Li and Hu \cite{BLH}, Hou, Tang
\cite{HT} and so on.

BSDEs methods,  were introduced originally by Peng~\cite{Pe1,Pe3,
Pe4} for the stochastic control theory. Since then BSDE methods have been
extended to the theory of SDGs. Hamadene, Lepeltier~\cite{HL} and
Hamadene, Lepeltier and Peng~\cite{HLP}
 studied games with a dynamics whose diffusion coefficient is
strictly elliptic and does not depend on controls. Buckdahn,
Li~\cite{BL} studied two-player zero-sum SDGs with the help of
decoupled FBSDEs driven by Brownian motion. They introduced the
method of Girsanov transformation which turned out to be a
straightforward way to prove that the upper and the lower value
functions of the game are deterministic. However, this method can't
be applied to SDGs with jumps. Buckdahn, Li and Hu
\cite{BLH} introduced a new type of measure-preserving and
invertible transformation on the Wiener-Poisson space to prove that
the upper and the lower value functions for two-player zero-sum SDGs
with jumps are deterministic. And the proof that they are
deterministic does not depend on the BSDE methods so that the new method can be used for the standard stochastic control problems with
jumps. In \cite{LiWei-Lp}, Li, Wei  studied some useful estimates for fully coupled FBSDEs with jumps under
the monotonic condition. Moreover,  under Lipschitz condition and
linear growth condition, they established the existence and the
uniqueness of the solution and prove $L^p$-estimates on a small time interval, which play
an important role in the study of the existence of the viscosity
solution for the corresponding second order integral-partial
differential equation of Isaacs' type over an arbitrary time
interval.

Inspired by the control problems in Li \cite{L}, Li, Wei \cite{LW}, as well as
Buckdahn, Li and Hu \cite{BLH}, we will investigate SDGs defined
through fully coupled FBSDEs driven by Brownian motion and Poisson
random measure, where $\sigma,\ h$ depend on $z$ and the controls
$u,\ v$ at the same time. For the fully coupled FBSDEs with jumps,
under the monotonicity assumptions Wu \cite{W1999} obtained the
existence and the uniqueness of the solution. Later, Wu \cite{W2003}
proved the existence and the uniqueness of the solution as well as a
comparison theorem for fully coupled FBSDEs with jumps over a
stochastic interval. Similarly to \cite{L, LW}, the second order
integral-partial differential equations of Isaacs' type are also
combined with the algebraic equations. Therefore, we still need the
representation theorem for the related algebraic equation which is
got in \cite{LW}.

Precisely, in this paper, the cost functional (interpreted as a payoff for
Player I and as a cost for Player II) of our SDGs is
introduced by the following fully coupled FBSDE driven by Brownian motion and Poisson random measure:
\begin{equation}
\label{equ0.1}\left \{
\begin{array}
[c]{llll}%
dX_{s}^{t,x;u,v} & = & b(s,\Pi_{s}^{t,x;u,v},u_{s},v_{s})ds +
\sigma(s,\Pi_{s}^{t,x;u,v},u_{s},v_{s}) dB_{s}+\int
_{E}h(s,\Pi_{s-}^{t,x;u,v},u_{s},v_{s})\tilde{\mu}(ds,de), & \\
dY_{s}^{t,x;u,v} & = & -f(s,\Pi_{s}^{t,x;u,v},\int
_{E}K_{s}^{t,x;u,v}(e)l(e)\lambda(de),u_{s},v_{s})ds+
Z_{s}^{t,x;u,v}dB_{s} +\int_{E}K_{s}^{t,x;u,v}(e)\tilde{\mu}(ds,de), & \\
X_{t}^{t,x;u,v} & = & \zeta, & \\
Y_{T}^{t,x;u,v} & = & \Phi(X_{T}^{t,x;u,v}), &
\end{array}
\right.
\end{equation}
where $s\in[t,T],\  \Pi_{s}^{t,x;u,v}=(X_{s}^{t,x;u,v},
Y_{s}^{t,x;u,v}, Z_{s}^{t,x;u,v}),\
\Pi_{s-}^{t,x;u,v}=(X_{s-}^{t,x;u,v}, Y_{s-}^{t,x;u,v},
Z_{s}^{t,x;u,v}),$ $T>0$ is an arbitrarily fixed finite time
horizon, and the admissible controls $u=(u_{s})_{s\in[t,T]}\in \mathcal {U}_{t,T}$,
$v=(v_{s})_{s\in[t,T]}\in \mathcal {V}_{t,T}$ are predictable and take their values in a
compact metric space $U$ and $V$, respectively. Under our
assumptions (see Section 2), the equation (\ref{equ0.1}) has a
unique solution $(X_{s}^{t,x;u,v}, Y_{s}^{t,x;u,v}, Z_{s}^{t,x;u,v},
K_{s}^{t,x;u,v})_{s\in[t,T]}$ and the cost functional is defined by
\[
J(t,x;u,v):=Y_{t}^{t,x;u,v}.
\]

We define the lower value function and the upper value function of our SDG, respectively, as follows
\[
W(t,x):=\mathop{\rm essinf}_{\beta \in \mathcal{B}_{t,T}}%
\mathop{\rm esssup}_{u\in \mathcal{U}_{t,T}}J(t,x;u,\beta(u)),
\]
\[
U(t,x):=\mathop{\rm esssup}_{\alpha \in \mathcal{A}_{t,T}}%
\mathop{\rm essinf}_{v\in \mathcal{V}_{t,T}}J(t,x;\alpha(v),v),
\]
where $\mathcal{A}_{t,T},\ \mathcal{B}_{t,T}$\ are the sets of
nonanticipative strategies of Player I and Player II, respectively
(see Definition 3.2 in Section 3). The objective of our paper is to
investigate the lower and the upper value functions. The main
results of the paper state that $W$ and $U$ are deterministic
(Proposition \ref{pro1}), satisfy the DPP (Theorem 3.1), and are
continuous viscosity solutions of the associated Bellman-Isaacs
equations (Theorem 4.1).  In our approach, we will use in a crucial manner  the results for fully coupled FBSDEs with jumps on the small time interval obtained by Li, Wei \cite{LiWei-Lp}.

Our paper is organized as follows. In Section 2, we present some
preliminaries for BSDEs with jumps and fully coupled FBSDEs with
jumps, which will be used later. The setting of our SDGs is
introduced in Section 3. We also show that the lower and upper value
functions (\ref{ee2}), (\ref{ee3}) are deterministic functions,
Lipschitz in $x$ (Lemma \ref{l7}) and ${\frac{1}{2}}$-H\"{o}lder
continuous in $t$ (Theorem \ref{th3.2}). Moreover, they satisfy the
DPP (Theorem \ref{th3.1}). In Section 4, by using the DPP, we prove
that $W$ and $U$ are the viscosity solutions of the associated
integral-differential Bellman-Isaacs equation. Furthermore, Section
5 presents the uniqueness of viscosity solution for the case when
$\sigma,\ h$ does not depend on $y, \ z,\ k$. This shows that, under
Isaacs' condition this game has a value. Finally, in Appendix we
give the proof of the DPP.

\section{ {\protect \large Preliminaries}}

 Let $(\Omega, {\mathcal{F}}, P)$ be a probability space
which is the completed product of the Wiener space $(\Omega_{1},
{\mathcal{F}}_{1}, P_{1})$
and the Poisson space $(\Omega_{2}, {\mathcal{F}}_{2}, P_{2}).$ \newline%
$\bullet \  \ (\Omega_{1}, {\mathcal{F}}_{1}, P_{1})$ is a classical Wiener space, where
$\Omega_{1}=C_{0}(\mathbb{R};\mathbb{R}^{d})$ is the set of continuous functions
from $\mathbb{R}$ to $\mathbb{R}^{d}$ with value $0$ in time $0$, $\mathcal{F}_{1}$ is the
completed Borel $\sigma$-algebra over $\Omega_{1}$, and $P_{1}$ is the Wiener
measure such that $B_{s}(\omega)=\omega_{s},\ s\in \mathbb{R}_{+},\  \omega \in
\Omega_{1}$, and $B_{-s}(\omega)=\omega(-s),\ s\in \mathbb{R}_{+},\  \omega
\in \Omega_{1}$, are two independent $d$-dimensional Brownian motions. The
natural filtration $\{ \mathcal{F}_{s}^{B},s\geq0\}$ is generated by
$\{B_{s}\}_{s\geq0}$ and augmented by all $P_{1}$-null sets, i.e.,
$$
\mathcal{F}_{s}^{B}=\sigma \{B_{r},r\in(-\infty,s]\} \vee \mathcal{N}_{P_{1}%
},\ s\geq0.
$$
$\bullet \  \ (\Omega_{2}, {\mathcal{F}}_{2}, P_{2})$ is a Poisson
space.  $p:D_{p}\subset
\mathbb{R}\rightarrow E$ is a point function, where $D_{p}$ is a
countable subset of the real line $\mathbb{R}$,  $E=\mathbb{R}^{l}\backslash \{0\}$ is equipped
with its Borel $\sigma$-field $\mathcal{B}(E)$. We introduce the counting
measure $\mu(p,dtde)$ on $\mathbb{R}\times E$ as follows:
\[
\mu(p,(s,t]\times \Delta)=\sharp \{r\in D_{p}\cap(s,t]:p(r)\in
\Delta \},\  \Delta \in \mathcal{B}(E),\ s,t\in \mathbb{R},\ s<t,
\]
where $\sharp$ denotes the cardinal number of the set. We
identify the point function $p$ with $\mu(p,\cdot).$ Let
$\Omega_{2}$ be the set of all point functions $p$ on $E$, and
$\mathcal{F}_{2}$ be the smallest $\sigma$-field on $\Omega_{2}$.   The coordinate mappings $p\rightarrow \mu (p,(s,t]\times \Delta),\ s,t\in
\mathbb{R},\ s<t,\  \Delta \in \mathcal{B}(E)$ are measurable with
respect to $\mathcal {F}_2$. On the measurable space
$(\Omega_{2},\mathcal{F}_{2})$ we consider the probability measure
$P_{2}$ such that the canonical coordinate measure $\mu(p,dtde)$
becomes a Poisson random measure with the compensator
$\hat{\mu}(dtde)=dt\lambda(de)$ and the process $\{
\tilde{\mu}((s,t]\times A)=(\mu-\hat{\mu})((s,t]\times A)\}_{s\leq
t}$ is a martingale, for any $A\in \mathcal{B}(E)$ satisfying
$\lambda(A)<\infty.$ Here $\lambda$ is supposed to be a
$\sigma$-finite measure on $(E,\mathcal{B}(E))$ with $\int_{E}(1\wedge|e|^{2})\lambda(de)<\infty.$ The
filtration $\{\mathcal{F}_{t}^{\mu}\}_{t\geq0}$ generated by
the coordinate measure $\mu$ is introduced by setting:
\[
\dot{\mathcal{F}}_{t}^{\mu}=\sigma \{ \mu((s,r]\times \Delta):-\infty<s\leq
r\leq t,\Delta \in \mathcal{B}(E)\},\ t\geq0,
\]
and taking the right-limits $\mathcal{F}_{t}^{\mu}=(\bigcap\limits_{s>t}%
\dot{\mathcal{F}}_{s}^{\mu})\vee \mathcal{N}_{P_{2}},\ t\geq0$, augmented by
all the $P_{2}$-null sets. At last, we set $(\Omega,\mathcal{F},
P)=(\Omega_{1}\times \Omega_{2}, \mathcal{F}_{1}\otimes \mathcal{F}_{2}%
,P_{1}\otimes P_{2})$, where $\mathcal{F}$ is completed with respect to $P,$
and the filtration $\mathbb{F}=\{ \mathcal{F}_{t}\}_{t\geq0}$ is generated by
\[
\mathcal{F}_{t}:=\mathcal{F}_{t}^{B}\otimes \mathcal{F}_{t}^{\mu}%
,\ t\geq0,\  \mbox{augmented by all }P\mbox{-null sets}.
\]
For any $n\geq1,\ |z|$ denotes the Euclidean norm of $z\in \mathbb{R}^{n}.$ Fix
$T>0,$ we introduce the following spaces of processes which will be used later.
\begin{itemize}
\item $\mathcal{M}^{2}(t,T;\mathbb{R}^{d}):=\Big \{ \varphi \mid \varphi
:\Omega \times[t,T]\rightarrow \mathbb{R}^{d} \mbox{ is an }\mathbb{F}%
\mbox{-predictable process}:\ \parallel \varphi \parallel^{2}=E[\int^{T}_{t} |\varphi_{s}%
|^{2}ds]<+\infty \Big \}; $

\item ${\mathcal{S}}^{2}(t,T;\mathbb{R}):=\Big \{ \psi \mid \psi:\Omega
\times[t,T]\rightarrow \mathbb{R} \mbox{ is an } \mathbb{F}%
\mbox{-adapted c\`{a}dl\`{a}g process}:\ E[\mathop{\rm
sup}\limits_{t\leq s\leq T}| \psi_{s} |^{2}]< +\infty \Big \}; $

\item $\mathcal{K}_{\lambda}^{2}(t,T;\mathbb{R}^{n}):=\Big \{K\mid
K:\Omega \times[t,T]\times E\rightarrow \mathbb{R}^{n} \mbox{ is
}\mathcal{P}\otimes \mathcal{B}(E)-\mbox{measurable}:\newline \mbox{ }\hskip3cm
\parallel K\parallel^{2}=E[\int_{t}^{T}\int
_{E}|K_{s}(e)|^{2}\lambda(de)ds]<+\infty \Big
\},$\footnote{$\mathcal {P}$ denotes the $\sigma$-field of
$\mathbb{F}$-predictable subsets of $\Omega\times [0,T].$}
\end{itemize}
where $t\in[0,T].$

\subsection{ {\protect \large BSDEs with jumps}}

Let us consider the following BSDE with jumps:
\begin{equation}
\label{equ2.1}Y_{t}=\xi+\int_{t}^{T}g(s,Y_{s},Z_{s},K_{s})ds-\int_{t}^{T}%
Z_{s}dB_{s}-\int_{t}^{T}\int_{E}K_{s}(e)\tilde{\mu}(ds,de),\ t\in[0,T].
\end{equation}
where $T>0$\ is an arbitrary time horizon, and the coefficient $g:\Omega \times[0,T]\times \mathbb{R}\times
\mathbb{R}^{d}\times
L^{2}(E,\mathcal{B}(E),\lambda;\mathbb{R})\rightarrow \mathbb{R}$ is
$\mathcal {P}$-measurable for each $(y,z,k)\in \mathbb{R}\times
\mathbb{R}^{d}\times L^{2}(E,\mathcal{B}(E),\lambda;\mathbb{R})$ and
satisfies:

\begin{description}
\item[$( \mathbf{H2.1})$] (i) There exists a constant $C\geq0$ such that,
$P$-a.s., for all $t\in[0,T],\ y_{1},y_{2}\in \mathbb{R},\ z_{1},z_{2}%
\in \mathbb{R}^{d},\ k_{1},k_{2}\in L^{2}(E,\mathcal{B}(E),\lambda
;\mathbb{R}),$
\[
|g(t,y_{1},z_{1},k_{1})-g(t,y_{2},z_{2},k_{2})|\leq C(|y_{1}-y_{2}%
|+|z_{1}-z_{2}|+\parallel k_{1}-k_{2}\parallel);
\]
(ii) $E[(\int_0^T|g(s,0,0,0)|ds)^2]<+\infty.$
\end{description}

Let us recall some well-known results.

\begin{lemma}
\label{l1} Under the assumption $( \mathbf{H2.1})$, for any random
variable $\xi \in L^{2}(\Omega,\mathcal{F}_{T},P;\mathbb{R})$, the
BSDE with jumps (\ref{equ2.1}) has a unique adapted solution
\[
(Y_{t},Z_{t},K_{t})_{t\in[0,T]}\in{\mathcal{S}}^{2}(0,T;\mathbb{R}%
)\times{\mathcal{M}}^{2}(0,T;\mathbb{R}^{d})\times \mathcal{K}_{\lambda}%
^{2}(0,T;\mathbb{R}).
\]

\end{lemma}

\begin{lemma}
\label{l2} (Comparison Theorem) Let $a:\Omega \times[0,T]\times \mathbb{R}%
\times \mathbb{R}^{d}\times \mathbb{R}$ be $\mathcal{P}\otimes \mathcal{B}%
(\mathbb{R})\otimes \mathcal{B}(\mathbb{R}^{d})\otimes \mathcal{B}(\mathbb{R})$
measurable and satisfy

{\rm(i)} there exists a constant $C\geq0$ such that, \mbox{P-a.s.}, for all
$t\in[0,T],\ y_{1},y_{2}\in \mathbb{R},\ z_{1},z_{2}\in \mathbb{R}^{d},$
$k_{1},k_{2}\in \mathbb{R},$
\[
|a(t,y_{1},z_{1},k_{1})-a(t,y_{2},z_{2},k_{2})|\leq C(|y_{1}-y_{2}%
|+|z_{1}-z_{2}|+| k_{1}-k_{2}|).
\]

{\rm(ii)} $a(\cdot,0,0,0)\in \mathcal{M}^{2}(0,T;\mathbb{R}).$

{\rm(iii)} $k\rightarrow a(t,y,z,k)$ is non-decreasing, for all $(t,y,z)\in
[0,T]\times \mathbb{R}\times \mathbb{R}^{d}.$

Furthermore, let $l:\Omega \times[0,T]\times E\rightarrow \mathbb{R}$ be
$\mathcal{P}\otimes \mathcal{B}(E)$ measurable and satisfy
\[
0\leq l_{t}(e)\leq C(1\wedge|e|),\  \ e\in E.
\]
We set
\[
g(\omega,t,y,z,\varphi)=a(\omega,t,y,z,\int_{E}\varphi
(e)l_{t}(\omega,e)\lambda(de)),
\]
for $(\omega,t,y,z,\varphi)\in \Omega \times[0,T]\times \mathbb{R}\times
\mathbb{R}^{d}\times L^{2}(E,\mathcal{B}(E),\lambda;\mathbb{R}).$

Let $\xi,\ \xi^{\prime}\in L^2(\Omega,\mathcal{F}_{T},P;\mathbb{R})$
and $g^{\prime}$ satisfies $( \mathbf{H2.1})$.

We denote by $(Y,Z,K)$ (resp., $(Y^{\prime},Z^{\prime},K^{\prime})$) the
unique solution of equation (\ref{equ2.1}) with the data $(\xi,g)$ (resp.,
$(\xi^{\prime},g^{\prime})$). If

{\rm(iv)} $\xi \geq \xi^{\prime},\ a.s.;$

{\rm(v)} $g(t,y,z,k)\geq g^{\prime}(t,y,z,k),\ a.s.,\ a.e.,$ for any $(y,z,k)\in
\mathbb{R}\times \mathbb{R}^{d}\times L^{2}(E,\mathcal{B}(E),\lambda
;\mathbb{R}),$ \newline then, we have: $Y_{t}\geq Y^{\prime}_{t},\ a.s.,$ for
all $t\in[0,T]$. And if, in addition, we also assume that $P(\xi>\xi^{\prime
})>0,$ then $P(Y_{t}>Y^{\prime}_{t})>0,\ 0\leq t\leq T$, and in particular,
$Y_{0}>Y^{\prime}_{0}.$
\end{lemma}

Using the notation introduced in Lemma \ref{l2}, we suppose that,
for some
$g:\Omega \times[0,T]\times \mathbb{R}\times \mathbb{R}^{d}\times L^{2}%
(E,\mathcal{B}(E),\lambda;\mathbb{R})\rightarrow \mathbb{R}$ satisfying $(
\mathbf{H2.1})$, and for  $i\in \{1,2\}$, the drivers $g_{i}$ are of the
form
\[
g_{i}(s,y_{s}^{i},z_{s}^{i},k_{s}^{i})=g(s,y_{s}^{i},z_{s}^{i},k_{s}%
^{i})+\varphi_{i}(s),\ dsdP\mbox{-}a.e.
\]

\begin{lemma}
\label{l1.1} The difference of the solutions $(Y^{1},Z^{1},K^{1})$ and
$(Y^{2},Z^{2},K^{2})$ of BSDE (\ref{equ2.1}) with the data $(\xi_{1},g_{1})$
and $(\xi_{2},g_{2})$, respectively, satisfies the following estimate:
\[%
\begin{array}
[c]{llll}
&  & |Y_{t}^{1}-Y_{t}^{2}|^{2}+{\frac{1}{2}}E[\int_{t}%
^{T}e^{\beta(s-t)}(|Y_{s}^{1}-Y_{s}^{2}|^{2}+|Z_{s}^{1}-Z_{s}^{2}%
|^{2})ds|\mathcal{F}_{t}]\\
&&+{\frac{1}{2}}E[\int_{t}^{T}%
\int_{E}e^{\beta(s-t)}|K_{s}^{1}(e)-K_{s}^{2}(e)|^{2}\lambda(de)ds|\mathcal{F}%
_{t}] & \\
&  & \leq E[e^{\beta(T-t)}|\xi_{1}-\xi_{2}|^{2}|\mathcal{F}_{t}%
]+E[\int_{t}^{T}e^{\beta(s-t)}|\varphi_{1}(s)-\varphi
_{2}(s)|^{2}ds|\mathcal{F}_{t}],\  \mbox{P-a.s.},\  \mbox{for all }
0\leq t\leq T, &
\end{array}
\]
where $\beta \geq2+2C+4C^{2}.$
\end{lemma}
The reader may refer to Barles, Buckdahn and Pardoux \cite{BBP} for
the proof.

\subsection{{\protect \large Fully coupled FBSDEs with jumps}}

 Now we consider the following fully coupled FBSDE with
jumps associated with $(b, \sigma, h, f, \zeta, \Phi)$\ on the time interval $[t, T]$, where $t\in [0, T]$:
\begin{equation}
\left \{
\begin{array}
[c]{llll}%
dX_{s} & = & b(s,X_{s},Y_{s},Z_{s},K_{s})ds+\sigma(s,X_{s},Y_{s},Z_{s}%
,K_{s})dB_{s} +\int_{E}h(s,X_{s-},Y_{s-},Z_{s},K_{s}(e),e)\tilde{\mu
}(ds,de), & \\
dY_{s} & = & -f(s,X_{s},Y_{s},Z_{s},\int_{E}K_{s}(e)l(e)\lambda(de))ds+Z_{s}dB_{s}+\int_{E}K_{s}(e)\tilde{\mu
}(ds,de),\  \  \ s\in \lbrack t,T], & \\
X_{t} & = & \zeta, & \\
Y_{T} & = & \Phi(X_{T}), &
\end{array}
\right.  \label{equ2.2}%
\end{equation}
where $(X, Y, Z, K)$\ takes its values in $ \mathbb{R}^{n}\times \mathbb{R}^{m}\times
\mathbb{R}^{m\times d}\times \mathbb{R}^{m}$, and
 $$\begin{array}{llll}
&&b:\Omega \times \lbrack0,T]\times \mathbb{R}^{n}\times
\mathbb{R}^{m}\times \mathbb{R}^{m\times d}\times
L^{2}(E,\mathcal{B}(E),\lambda;\mathbb{R}^m)\longrightarrow \mathbb{R}^{n}, \\
&& \sigma
:\Omega \times \lbrack0,T]\times \mathbb{R}^{n}\times \mathbb{R}^{m}%
\times \mathbb{R}^{m\times d}\times L^{2}(E,\mathcal{B}(E),\lambda;\mathbb{R}^m)\longrightarrow \mathbb{R}%
^{n\times d},\\
&& h:\Omega \times \lbrack0,T]\times \mathbb{R}^{n}%
\times \mathbb{R}^{m}\times \mathbb{R}^{m\times d}\times
\mathbb{R}^{m}\times E\longrightarrow \mathbb{R}^{n},\\
&&f:\Omega
\times \lbrack0,T]\times \mathbb{R}^{n}\times \mathbb{R}^{m}\times
\mathbb{R}^{m\times d}\times \mathbb{R}^{m}\longrightarrow
\mathbb{R}^{m}, \end{array}$$
 $l:E\longrightarrow \mathbb{R}$ and $\Phi:\Omega \times \mathbb{R}^{n}%
\longrightarrow \mathbb{R}^{m}$ satisfy

\begin{description}
\item[$( \mathbf{H2.2})$] (i) $b,\  \sigma,\ f$ are uniformly Lipschitz with respect to $(x,y,z,k),$ and there exists $\rho:E\rightarrow
\mathbb{R}^{+}$ with $\int_{E}\rho^{2}(e)\lambda(de)<+\infty$
such that, for any $t\in[0,T],\ x,\bar{x}\in \mathbb{R}^{n},\ y,\bar{y}%
\in \mathbb{R}^{m},\ z,\bar{z}\in \mathbb{R}^{m\times d},\ k,\bar{k}%
\in \mathbb{R}^{m}$ and $e\in E$,
\[
|h(t,x,y,z,k,e)-h(t,\bar{x},\bar{y},\bar{z},\bar{k},e)|\leq \rho(e)(|x-\bar
{x}|+|y-\bar{y}|+|z-\bar{z}|+|k-\bar{k}|);
\]
(ii) $k\rightarrow f(t,x,y,z,k)$ is non-decreasing, for all $(t,x,y,z)\in
[0,T]\times \mathbb{R}^{n}\times \mathbb{R}^{m}\times \mathbb{R}^{m\times d}; $

(iii) there exists a constant $C>0$ such that
\[
0\leq l(e)\leq C(1\wedge|e|),\ x\in \mathbb{R}^{n},\ e\in E;
\]
(iv) $\Phi(x)$ is uniformly Lipschitz with respect to $x\in \mathbb{R}^{n};$

(v) for every $(x,y,z,k)\in \mathbb{R}^{n}\times
\mathbb{R}^{m}\times \mathbb{R}^{m\times d}\times \mathbb{R}^{m}, \
\Phi(x)\in L^{2}(\Omega ,\mathcal{F}_{T},P;\mathbb{R}^m)$, $b,\
\sigma,\ h,\ f$ are $\mathbb{F}$-progressively measurable and
\[%
\begin{array}
[c]{ll}%
E\int_{0}^{T}|b(s,0,0,0,0)|^{2}ds+E\int_{0}%
^{T}|f(s,0,0,0,0)|^{2}ds+E\int_{0}^{T}|\sigma(s,0,0,0,0)|^{2}%
ds & \\
+E\int_{0}^{T}\int_{E}|h(s,0,0,0,0,e)|^{2}%
\lambda(de)ds<\infty. &
\end{array}
\]
\end{description}

Let
\[
g(s,x,y,z,k):=f(s,x,y,z,\int_{E}k(e)l(e)\lambda(de)),
\]
$(s,x,y,z,k)\in[0,T]\times \mathbb{R}^{n}\times \mathbb{R}^{m}\times
\mathbb{R}^{m\times d}\times L^{2}(E,\mathcal{B}(E),\lambda;\mathbb{R}).$

In this paper we use the usual inner product and the Euclidean norm in
$\mathbb{R}^{n},\  \mathbb{R}^{m}$ and $\mathbb{R}^{m\times d},$ respectively.
Given an $m \times n$ full-rank matrix $G$, we define:
\[
\pi= \  \left(
\begin{array}
[c]{c}%
x\\
y\\
z
\end{array}
\right)  \ , \  \  \  \  \  \  \  \  \  \ A(t,\pi,k)= \  \left(
\begin{array}
[c]{c}%
-G^{T}g\\
Gb\\
G\sigma
\end{array}
\right)  (t,\pi,k),
\]
where $G^{T}$ is the transposed matrix of $G$.

We assume the following monotonicity conditions:
\begin{description}
\item[$( \mathbf{H2.3})$] {\rm (i)}
 $\begin{array}[c]{llll}&&\langle
A(t,\pi,k)-A(t,\bar{\pi},\bar{k}),\pi-\bar{\pi} \rangle
+\int_{E}\langle G\widehat{h}(e), \widehat{k}(e)\rangle
\lambda(de) \\
&& \leq-\beta_{1}|G\widehat{x}|^{2}-\beta_{2}(  |G^{T} \widehat{y}%
|^{2}+|G^{T} \widehat{z}|^{2})-\beta_{3}\int_{E}|G^{T} \widehat{k}%
(e)|^{2}\lambda(de),
\end{array}
$

(ii) $\langle \Phi(x)-\Phi(\bar{x}),G(x-\bar{x}) \rangle \geq \mu_{1}%
|G\widehat{x}|^{2},\  \forall \pi=(x,y,z),\  \bar{\pi}=(\bar{x},\bar{y},\bar
{z}),\  \widehat{x}=x-\bar{x},\  \widehat{y}=y-\bar{y},\  \widehat{z}=z-\bar
{z},\  \widehat{k}=k-\bar{k},\  \widehat{h}(e)=h(t,\pi,k,e)-h(t,\bar{\pi}%
,\bar{k},e)$, \newline where $\beta_{1},\ \beta_{2},\ \beta_{3},\
\mu_{1}$ are nonnegative constants with $\beta_{1} + \beta_{2}>0,\
\beta_{1} + \beta_{3}>0,\ \beta_{2} + \mu_{1}>0,\ \beta_{3} +
\mu_{1}>0$. Moreover, we have $\beta_{1}>0,\ \mu_{1}>0 \
(\mbox{resp., }\beta_{2}>0,\ \beta_{3}>0)$, when $m>n$ (resp.,
$m<n$).
\end{description}

\begin{remark}
\begin{description}\item[$( \mathbf{H2.3})$-(ii)']  A consequence of  $( \mathbf{H2.3})$ {\rm(ii)} is the weaker condition:
$\langle \Phi(x)-\Phi(\bar{x}),G(x-\bar{x})
\rangle \geq0$, for all $x,\ \bar{x}\in\mathbb{R}^n.$ \end{description}
\end{remark}

When $\Phi(x)=\xi\in L^{2}(\Omega,\mathcal{F}_{T},P;\mathbb{R}^{m})$, $( \mathbf{H2.3})$-(i) can be weaken as follows:

\begin{description}\item[$( \mathbf{H2.4})$] $\langle
A(t,\pi,k)-A(t,\bar{\pi},\bar{k}),\pi-\bar{\pi} \rangle
+\int_{E}\langle G\widehat{h}(e), \widehat{k}(e)\rangle
\lambda(de)  \leq-\beta_{1}|G\widehat{x}|^{2}-\beta_{2}|G^{T} \widehat{y}
|^{2},$

where $\beta_1,\ \beta_2$ are nonnegative constants with $\beta_1 +
\beta_2>0$. Moreover, we have $\beta_1>0$ (resp., $\beta_2>0$), when
$m>n$ (resp., $m<n$).
\end{description}

\begin{lemma}
\label{l3} Under the assumptions $( \mathbf{H2.2})$ and $(\mathbf{H2.3})$, for
any $\zeta \in L^{2}(\Omega,\mathcal{F}_{t},P;\mathbb{R}^{n})$, FBSDE
(\ref{equ2.2}) has a unique adapted solution $(X_{s},Y_{s},Z_{s},K_{s})_{s
\in[t,T]}\in{\mathcal{S}}^{2}(t, T; {\mathbb{R}^{n}})\times{\mathcal{S}}%
^{2}(t, T; {\mathbb{R}^{m}})\times{\mathcal{M}}^{2}(t, T; {\mathbb{R}^{m\times
d}})\times \mathcal{K}_{\lambda}^{2}(t, T; \mathbb{R}^{m}).$
\end{lemma}

\begin{lemma}
\label{le2.2} Under the assumptions $( \mathbf{H2.3})$-{\rm(ii)'} and $(\mathbf{H2.4})$, for
any $\zeta \in L^{2}(\Omega,\mathcal{F}_{t},P;\mathbb{R}^{n})$ and the terminal condition $\Phi(x)=\xi\in L^{2}(\Omega,\mathcal{F}_{T},P;\mathbb{R}^{m})$, FBSDE
(\ref{equ2.2}) has a unique adapted solution $(X_{s},Y_{s},Z_{s},K_{s})_{s
\in[t,T]}\in{\mathcal{S}}^{2}(t, T; {\mathbb{R}^{n}})\times{\mathcal{S}}%
^{2}(t, T; {\mathbb{R}^{m}})\times{\mathcal{M}}^{2}(t, T; {\mathbb{R}^{m\times
d}})\times \mathcal{K}_{\lambda}^{2}(t, T; \mathbb{R}^{m}).$
\end{lemma}

For the proof, the reader can refer to Wu \cite{W1999, W2003}.

Now we recall the comparison theorem for fully coupled FBSDEs with jumps.

\begin{lemma}
\label{l4} (Comparison Theorem) Let $m=1$ and assume that $(b, \sigma
, h, f^{i}, a^{i}, \Phi^{i})$, for $i=1,2,$ satisfy $(\mathbf{H2.2}), \ (
\mathbf{H2.3})$, where $b,\  \sigma,\ h$ do not depend on $k$, $a_i$ is the initial state. Let $(x_{s}%
^{i},y_{s}^{i},z_{s}^{i}, k_{s}^{i})_{t\leq s \leq T}$ be the solution of
FBSDE (\ref{equ2.2}) associated with $(b,\sigma,h,f^{i},a^{i},\Phi^{i})$\ on the time interval $[t, T]$. We
assume $a^{1}\geq a^{2},\  \Phi^{1}(x)\geq \Phi^{2}(x),\ f^{1}(t,x,y,z,k)\geq
f^{2}(t,x,y,z,k),\  \mbox{P-a.s.}$, for all $(x,y,z,k)\in \mathbb{R}^{n}%
\times \mathbb{R}\times \mathbb{R}^{d}\times \mathbb{R},$ then, $y_{t}^{1}\geq
y_{t}^{2},\  \mbox{P-a.s.}$
\end{lemma}

The above lemma can be found in \cite{W2003}.

\section{{\protect \large {A DPP for stochastic differential games for fully
coupled FBSDEs with jumps}}}

 In this section, we consider  stochastic differential
games for fully coupled FBSDEs with jumps. First we introduce the
background of stochastic differential games. Suppose that the
control state spaces $U,\ V$ are compact metric spaces. By
${\mathcal{U}}$ (resp., ${\mathcal{V}}$) we denote the admissible
control set of all $U$ (resp., $V$)-valued
$\mathcal{F}_{t}$-predictable processes for the first (resp.,
second) player. If $u\in \mathcal{U}$ (resp., $v\in \mathcal{V}$),
we call $u$ (resp., $v$) an admissible control.

Let us give the following deterministic measurable functions
$$
\begin{array}
[c]{l}%
b: [0,T] \times \mathbb{R} \times \mathbb{R} \times \mathbb{R}^{d}
\times
U\times V \longrightarrow \mathbb{R},\ \ \ \ \ \ \ \ \ \
\sigma: [0,T] \times \mathbb{R} \times \mathbb{R} \times
\mathbb{R}^{d}
\times U\times V \longrightarrow \mathbb{R}^{d},\\
h: [0,T] \times \mathbb{R} \times \mathbb{R} \times
\mathbb{R}^{d}\times
U\times V \times E \longrightarrow \mathbb{R},\  \  \
f: [0,T] \times \mathbb{R} \times \mathbb{R} \times \mathbb{R}^{d}
\times \mathbb{R}\times U\times V \longrightarrow \mathbb{R},\end{array}$$
and $l: E \longrightarrow \mathbb{R},\
\Phi: \mathbb{R} \longrightarrow \mathbb{R},$
which are continuous in $(t, u,v)$, and satisfy the assumptions $(
\mathbf{H2.2}), \ ( \mathbf{H2.3})$, uniformly in $u\in U,\ v\in V.$

For given admissible controls $u(\cdot)\in{\mathcal{U}}$, $v(\cdot
)\in{\mathcal{V}}$ and the initial data $(t,\zeta)\in[0,T]\times
L^{2} (\Omega,{\mathcal{F}}_{t}, P;{\mathbb{R}})$, we consider the
following fully coupled forward-backward stochastic system
\begin{equation}
\label{equ3.1}\left \{
\begin{array}
[c]{llll}%
dX_{s}^{t,\zeta;u,v} & = & b(s,\Pi_{s}^{t,\zeta;u,v},u_{s},v_{s})ds
+ \sigma(s,\Pi_{s}^{t,\zeta;u,v},u_{s},v_{s}) dB_{s}+ \int
_{E}h(s,\Pi_{s-}^{t,\zeta;u,v},u_{s},v_{s},e)\tilde{\mu}(ds,de),\
 & \\
dY_{s}^{t,\zeta;u,v} & = & -f(s,\Pi_{s}^{t,\zeta;u,v}, \int
_{E}K_{s}^{t,\zeta;u,v}(e)l(e)\lambda(de),u_{s},v_{s})ds+
Z_{s}^{t,\zeta;u,v}dB_{s}  +\int_{E}K_{s}^{t,\zeta;u,v}(e)\tilde{\mu}(ds,de), & \\
X_{t}^{t,\zeta;u,v} & = & \zeta, & \\
Y_{T}^{t,\zeta;u,v} & = & \Phi(X_{T}^{t,\zeta;u,v}), &
\end{array}
\right.
\end{equation}
where $s\in[t,T],\
\Pi_{s}^{t,\zeta;u,v}=(X_{s}^{t,\zeta;u,v},Y_{s}^{t,\zeta;u,v},Z_{s}^{t,\zeta;u,v})$
and $
\Pi_{s-}^{t,\zeta;u,v}=(X_{s-}^{t,\zeta;u,v},Y_{s-}^{t,\zeta;u,v},Z_{s}^{t,\zeta;u,v}).$

Therefore, for any $u(\cdot) \in \mathcal{U},\ v(\cdot) \in
\mathcal{V},$ from Lemma \ref{l3}, we have that FBSDE (\ref{equ3.1})
has a unique solution.

\begin{remark}
Due to the restrictions coming from  the comparison theorem (Lemma
\ref{l4}, Theorem 3.3 in \cite{LiWei-Lp}) which will be used in Section 4, we
emphasize that the coefficients $b,\ \sigma,\ h$ do not depend on
the variable $k$.
\end{remark}

\begin{remark}
Under our assumptions, it is obvious that $b,\sigma,h,f,\Phi$ have
linear growth in $(\pi,k)=(x,y,z,k)$, i.e.,
\[%
\begin{array}
[c]{llll}
&  & |b(t,\pi,u,v)|+|\sigma(t,\pi,u,v)|+|f(t,\pi,k,u,v)|+|\Phi(x)|  \leq C(1+|x|+|y|+|z|+|k|), & \\
&  & |h(t,\pi,u,v,e)|\leq \rho(e)(1+|x|+|y|+|z|), &
\end{array}
\]
for $(t,x,y,z,k,u,v)\in[0,T]\times \mathbb{R}\times \mathbb{R}%
\times \mathbb{R}^{d}\times \mathbb{R}\times U\times V.$
\end{remark}

From Proposition 3.1 in \cite{LiWei-Lp}, for our FBSDE with jumps
(\ref{equ3.1}), it is easy to check that, there exists $C \in
\mathbb{R}^{+}$ such that, for any $t \in[0,T]$, $\zeta,\
\zeta^{\prime}\in L^2(\Omega,\mathcal{F}_{t},P;\mathbb{R}),$
$u(\cdot) \in \mathcal{U},\ v(\cdot) \in \mathcal{V},$ we have,
\mbox{P-a.s.}:
\begin{equation}
\label{ee1}%
\begin{array}
[c]{llll}
&  & \mathrm{(i)}\  \  \ |Y_{t}^{t,\zeta;u,v}| \leq C(1+|\zeta|); & \\
&  & \mathrm{(ii)}\  \
|Y_{t}^{t,\zeta;u,v}-Y_{t}^{t,\zeta^{\prime};u,v}| \leq C|\zeta-
\zeta^{\prime}|. &
\end{array}
\end{equation}

Now, we introduce the subspaces of admissible controls and the
definition of admissible strategies, which are similar to \cite{BL,
BLH, FS}.

\begin{definition}
An admissible control process $u=(u_{r})_{r\in[t,s]}$ (resp., $v=(v_{r}%
)_{r\in[t,s]}$) for Player I (resp., II) on $[t,s]$ is an $\mathcal{F}_{r}%
$-predictable, $U$ (resp., $V$)-valued process. The set of all
admissible controls for Player I (resp., II) on $[t,s]$ is denoted by
$\mathcal{U}_{t,s}$ (resp., $\mathcal{V}_{t,s}).$ If $P\{u\equiv
\bar{u},\ a.e.,\  \mbox{in } [t,s]
\}=1$, we will identify both processes $u$ and $\bar{u}$ in $\mathcal{U}%
_{t,s}.$ Similarly we interpret $v\equiv \bar{v}$ on $[t,s]$ in $\mathcal{V}%
_{t,s}.$
\end{definition}

\begin{definition}
\ A nonanticipative strategy for Player I on $[t,s] \ (t<s\leq T)$
is a mapping $\alpha:\  \mathcal{V}_{t,s}\rightarrow
\mathcal{U}_{t,s}$ such that, for any $\mathcal{F}_{r}$-stopping
time $S:\Omega \rightarrow[t,s]$ and any $v_{1},v_{2}\in
\mathcal{V}_{t,s},$ with $v_{1}\equiv v_{2}$ on $[[t,S]]$, it holds
that $\alpha(v_{1})\equiv \alpha(v_{2})$ on $[[t,S]]$.
Nonanticipative
strategies for Player II on $[t,s]$, $\beta: \  \mathcal{U}_{t,s}%
\rightarrow \mathcal{V}_{t,s},$ are defined similarly. The set of
all nonanticipative strategies $\alpha:\
\mathcal{V}_{t,s}\rightarrow \mathcal{U}_{t,s}$ for Player I on
$[t,s]$ is denoted by $\mathcal{A}_{t,s}.$
The set of all nonanticipative strategies $\beta:\  \mathcal{U}_{t,s}%
\rightarrow \mathcal{V}_{t,s}$ for Player II on $[t,s]$ is denoted
by $\mathcal{B}_{t,s}.$ (Recall that $[[t,S]]
=\{(r,\omega)\in[0,T]\times \Omega,t\leq r \leq S(\omega)$.)
\end{definition}

For given processes $u(\cdot) \in \mathcal{U}_{t,T},\ v(\cdot) \in
\mathcal{V}_{t,T}$, the cost functional is defined as follows:
\begin{equation}
J(t,x;u,v):=Y_{t}^{t,x;u,v}, \  \ (t,x)\in[0,T] \times \mathbb{R},
\end{equation}
where the process $Y^{t,x;u,v}$ is defined by FBSDE (\ref{equ3.1}).

From Theorem 3.1 in \cite{LiWei-Lp} we have, for any $t \in[0,T]$ and $\zeta
\in L^{2}(\Omega,\mathcal{F}_{t},P;\mathbb{R}),$
\begin{equation}
J(t,\zeta;u,v):=Y_{t}^{t,\zeta;u,v}, \mbox{ \mbox{P-a.s.}}
\end{equation}
For $\zeta= x \in \mathbb{R},$ we define the lower value function of
our stochastic differential games
\begin{equation}
\label{ee2}W(t,x) :=\mathop{\rm essinf}_{\beta \in \mathcal{B}_{t,T}%
}\mathop{\rm esssup}_{u\in \mathcal{U}_{t,T}}J(t,x;u,\beta(u)),
\end{equation}
and its upper value function
\begin{equation}
\label{ee3}U(t,x) := \mathop{\rm esssup}_{\alpha \in \mathcal{A}_{t,T}%
}\mathop{\rm essinf}_{v\in \mathcal{V}_{t,T}}J(t,x;\alpha(v),v).
\end{equation}

\begin{remark}
Thanks to the assumptions $( \mathbf{H2.2}),\ ( \mathbf{H2.3})$, the
lower value function $W(t,x)$ and the upper value function $U(t,x)$
are well defined, and they are bounded
${\mathcal{F}}_{t}$-measurable random variables. But they even turn
out to be deterministic.
\end{remark}

Next we will prove that $W,\ U$ are deterministic. The method of
Girsanov transformation for fully coupled FBSDEs without jumps (see
\cite{BL, LW}) does not apply to the case with jumps now. So we use
a new transformation method introduced by Buckdahn, Li and Hu
\cite{BLH} to complete the proof that $W,\ U$ are deterministic.
Next we only give the proof for $W$, that for $U$ is  similar.

\begin{proposition}
\label{pro1} For any $(t,x) \in[0,T] \times \mathbb{R},$ $W(t,x)$ is
a deterministic function in the sense that $W(t,x) = E[W(t,x)],\
\mbox{P-a.s.}$
\end{proposition}

Combining the following both lemmas, we can complete the proof of
this proposition.

\begin{lemma}
\label{l5} Let $(t,x)\in[0,T]\times \mathbb{R}$ and $\tau:\Omega
\rightarrow \Omega$ be an invertible $\mathcal{F}-\mathcal{F}$
measurable transformation such that

{\rm{(i)}} $\tau^{-1}(\mathcal {F}_t)\subset \mathcal {F}_t$ and
$\tau(\mathcal {F}_t)\subset \mathcal {F}_t$;

{\rm{(ii)}} $(B_{s}-B_{t})\circ \tau=B_{s}-B_{t},\ \mu((t,s]\times A)\circ \tau=\mu((t,s]\times A),\ s\in
[t,T],\ A\in \mathcal{B}(E)$;

{\rm{(iii)}} the law $P\circ[\tau]^{-1}$ of $\tau$ is equivalent to
the underlying probability measure $P$. \newline Then, $W(t,x)\circ
\tau=W(t,x),\  \mbox{P-a.s.}$
\end{lemma}

\begin{remark} \label{4.4} The assumptions {\rm(i)} and {\rm(ii)} of Lemma \ref{l5} imply
that $$\tau^{-1}(\mathcal {F}_s)=\tau(\mathcal {F}_s)=\mathcal
{F}_s,\ s\in[t,T].
$$
\end{remark}

\begin{proof} We split the proof in the following steps:

(1). For any $u\in \mathcal {U}_{t,T},\ v\in \mathcal {V}_{t,T},\
J(t,x,u,v)\circ \tau=J(t,x,u(\tau),v(\tau)),\  \mbox{P-a.s.}$ \\
In fact, applying the transformation $\tau$ to FBSDE (\ref{equ3.1})
(with $\zeta=x$) and comparing the obtained equation with the FBSDE
obtained from (\ref{equ3.1}) by substituting the controlled
processes $u(\tau),\ v(\tau)$ for $u$ and $v$, we get from the
uniqueness of the solution of (\ref{equ3.1}) and the properties of
the transformation  $\tau$ that
$$ \begin{array}{llll}
X_s^{t,x;u,v}(\tau)=X_s^{t,x,u(\tau),v(\tau)},\  \mbox{for all }
s\in [t,T],\  \mbox{P-a.s.}\\
Y_s^{t,x;u,v}(\tau) =Y_s^{t,x,u(\tau),v(\tau)},\  \mbox{for all }
s\in [t,T],\  \mbox{P-a.s.}\\
Z_s^{t,x;u,v}(\tau) =Z_s^{t,x,u(\tau),v(\tau)},\ \mbox{dsdP-a.e. on } [t,T]\times \Omega,\\
K_s^{t,x;u,v}(\tau) =K_s^{t,x,u(\tau),v(\tau)},\
\mbox{ds}\lambda\mbox{(de)dP-a.e.  on } [t,T]\times E \times \Omega.\\
\end{array}$$
\\
Consequently, in particular, we have
$$J(t,x,u,v)\circ \tau=J(t,x,u(\tau),v(\tau)),\  \mbox{P-a.s.}$$
(2). For $\beta \in \mathcal {B}_{t,T},$ let
$\hat{\beta}(u):=\beta(u(\tau^{-1}))(\tau),\ u\in \mathcal
{U}_{t,T}$. Then, $\hat{\beta}\in \mathcal {B}_{t,T}.$

Obviously, $\hat{\beta}$ maps $\mathcal {U}_{t,T}$ into $\mathcal
{V}_{t,T}$. Moreover, $\hat{\beta}$ is nonanticipative. Indeed, let
$S:\Omega \rightarrow [t,T]$ be an $\mathbb{F}$-stopping time and
$u_1,u_2\in \mathcal {U}_{t,T}$ such that $u_1\equiv u_2$ on
$[[t,S]]$. Then, obviously, $u_1(\tau^{-1})\equiv u_2(\tau^{-1})$ on
$[[t,S(\tau^{-1})]]$ (notice that $S(\tau^{-1})$ is still an
$\mathbb{F}$-stopping time. For this we use that the assumptions (i) and
(ii) imply that $\tau(\mathcal {F}_s):=\{ \tau(A),A\in \mathcal
{F}_s\}=\mathcal {F}_s,\ s\in[t,T]$). Since $\beta \in \mathcal
{B}_{t,T}$, we have $\beta(u_1(\tau^{-1}))\equiv
\beta(u_2(\tau^{-1}))$ on $[[t,S(\tau^{-1})]]$. Therefore,
$$\hat{\beta}(u_1)=\beta(u_1(\tau^{-1}))(\tau)\equiv \beta(u_2(\tau^{-1}))(\tau)=\hat{\beta}(u_2),\  \mbox{on } [[t,S]].$$
(3). For all $\beta \in \mathcal {B}_{t,T}$ we have
$$(\esssup \limits_{u\in \mathcal {U}_{t,T}}J(t,x;u,\beta(u)))(\tau)=\esssup \limits_{u\in \mathcal {U}_{t,T}}(J(t,x;u, \beta(u))(\tau)),\ \mbox{P-a.s.}$$
Indeed, let us use the notation $I(t,x,\beta):=\esssup \limits_{u\in
\mathcal {U}_{t,T}}J(t,x;u,\beta(u)),\  \beta \in \mathcal
{B}_{t,T},\ \mbox{P-a.s.}$ Then, $I(t,x,\beta)(\tau)\geq J(t,x;u, \beta(u))(\tau),$ P-a.s., for all $u\in\mathcal {U}_{t,T}.$ From the definition of essential supremum
over a family of random variables, for any random variable $\zeta$
satisfying $\zeta \geq J(t,x;u, \beta(u))(\tau)$ and, hence,
$\zeta(\tau^{-1})\geq J(t,x;u, \beta(u)),$ $\mbox{P-a.s.},$ for all
$u\in \mathcal {U}_{t,T}$, we have $\zeta(\tau^{-1})\geq
I(t,x,\beta), $ P-a.s., i.e., $\zeta \geq I(t,x,\beta)(\tau)$.
Consequently,
$$I(t,x,\beta)(\tau)=\esssup \limits_{u\in \mathcal
{U}_{t,T}}(J(t,x;u, \beta(u))(\tau)),\  \mbox{P-a.s.}$$ (4).
Similarly to the above proof, we can prove
$$(\essinf_{\beta \in \mathcal {B}_{t,T}}I(t,x,\beta))(\tau)=\essinf_{\beta \in \mathcal {B}_{t,T}}(I(t,x,\beta)(\tau)),\  \mbox{P-a.s.}$$
(5). $W(t,x)$ is invariant with respect to the transformation
$\tau,$ i.e., $$W(t,x)(\tau)=W(t,x),\  \mbox{P-a.s.}$$ Indeed,
combing those steps above we have
$$\begin{array}{llll} W(t,x)(\tau)&=&\essinf \limits_{\beta \in
\mathcal {B}_{t,T}}\esssup \limits_{u\in \mathcal
{U}_{t,T}}(J(t,x;u,\beta(u))(\tau))=\essinf \limits_{\beta \in \mathcal {B}_{t,T}}\esssup
\limits_{u\in \mathcal
{U}_{t,T}}J(t,x;u(\tau),\hat{\beta}(u(\tau)))\\
&=&\essinf \limits_{\beta \in \mathcal {B}_{t,T}}\esssup
\limits_{u\in \mathcal
{U}_{t,T}}J(t,x;u,\hat{\beta}(u))=\essinf \limits_{\beta \in \mathcal {B}_{t,T}}\esssup
\limits_{u\in \mathcal
{U}_{t,T}}J(t,x;u,\beta(u))\\
&=& W(t,x),\  \mbox{P-a.s.}
\end{array}$$
where in the both latter equalities we have used
$\{u(\tau)|u(\cdot)\in \mathcal {U}_{t,T}\}=\mathcal {U}_{t,T},\ \{
\hat{\beta}|\beta \in \mathcal {B}_{t,T}\}=\mathcal {B}_{t,T}$.

\end{proof}

Now let $l\geq1.$ We define the transformation
$\tau^{\prime}_{l}:\Omega
_{1}\rightarrow \Omega_{1}$ such that, for all $\omega_{1}\in \Omega_{1}%
=C_{0}(\mathbb{R};\mathbb{R}^{d})$,
\begin{equation}%
\begin{array}
[c]{llll} &  &
(\tau^{\prime}_{l}\omega_{1})((t-l,r])=\omega_{1}((t-2l,r-l])(:=\omega
_{1}(r-l)-\omega_{1}(t-2l)); & \\
&  & (\tau^{\prime}_{l}\omega_{1})((t-2l,r-l])=\omega_{1}((t-l,r]),
\  \mbox{for
}r\in[t-l,t]; & \\
&  & (\tau^{\prime}_{l}\omega_{1})((s,r])=\omega_{1}((s,r]), \
(s,r]\cap
(t-2l,t]=\emptyset; & \\
&  & (\tau^{\prime}_{l}\omega_{1})(0)=0. &
\end{array}
\end{equation}
Moreover, for $p\in \Omega_{2},\ p=\sum \limits_{x\in
D_{p}}p(x)\delta_{x}$, we set
\[
\tau^{\prime \prime}_{l}:=\sum_{x\in D_{p}\cap(t-2l,t]^{c}}p(x)\delta_{x}%
+\sum_{x\in D_{p}\cap(t-l,t]}p(x)\delta_{x-l}+\sum_{x\in D_{p}\cap
(t-2l,t-l]}p(x)\delta_{x+l}.
\]
It is easy to check that, $\tau^{\prime \prime}_{l}:\Omega_{2}\rightarrow
\Omega_{2}$ is a bijection, $\tau^{\prime \prime-1}_{l}=\tau^{\prime
\prime}_{l}$, which preserves the measure $P_{2}\circ[\tau^{\prime
\prime}]^{-1}=P_{2}.$ Moreover,
\begin{equation}%
\begin{array}
[c]{llll} &  & \mu(\tau^{\prime \prime}_{l}p;(t-l,r]\times
\Delta)=\mu(p;(t-2l,r-l]\times
\Delta),\ r\in(t-l,t],\  \Delta \in \mathcal{B}(E); & \\
&  & \mu(\tau^{\prime \prime}_{l}p;(t-2l,r-l]\times
\Delta)=\mu(p;(t-l,r]\times
\Delta),\ r\in(t-l,t],\  \Delta \in \mathcal{B}(E); & \\
&  & \mu(\tau^{\prime \prime}_{l}p;(s,r]\times
\Delta)=\mu(p;(s,r]\times \Delta),
\ (s,r]\cap(t-2l,t]=\emptyset,\  \Delta \in \mathcal{B}(E). & \\
&  &  &
\end{array}
\end{equation}
Consequently, the transformation $\tau_{l}:\Omega \rightarrow
\Omega,\ \tau
_{l}\omega:=(\tau^{\prime}_{l}\omega_{1},\tau^{\prime \prime}_{l}%
p),\  \omega=(\omega_{1},p)\in \Omega=\Omega_{1}\times \Omega_{2},$
satisfies the assumptions (i), (ii), (iii) of Lemma \ref{l5}.
Therefore, $W(t,x)(\tau _{l})=W(t,x),\  \mbox{P-a.s.},\ l\geq1.$
Combined with the following auxiliary Lemma \ref{l6}, we can
complete the proof of Proposition \ref{pro1}.

\begin{lemma}
\label{l6} Let $\zeta \in L^{\infty}(\Omega,\mathcal{F}_{t},P)$ be
such that, for all $l\geq1$ natural number, $\zeta(\tau_{l})=\zeta,$
P-a.e. Then, there exists some real $C$ such that $\zeta=C,\
\mbox{P-a.s.}$
\end{lemma}

For the proof the reader is referred to  Lemma 3.2 in Buckdahn, Li
and Hu \cite{BLH}.

From (\ref{ee1}) and (\ref{ee2})-(the definition of the value
function $W(t,x)$), we get the following property:

\begin{lemma}
\label{l7} There exists a constant $C>0$ such that, for all $0 \leq
t \leq T,\ x,x^{\prime}\in \mathbb{R},$
\begin{equation}%
\begin{array}
[c]{llll}
&  & \mathrm{(i)} \  \ |W(t,x)-W(t,x^{\prime})| \leq C|x-x^{\prime}|; & \\
&  & \mathrm{(ii)}\  \ |W(t,x)| \leq C(1+|x|). &
\end{array}
\end{equation}
\end{lemma}

Moreover, for $W(t,x)$, we have the following monotonic property.

\begin{lemma}
\label{l8} Under the assumptions $( \mathbf{H2.2})$, $(
\mathbf{H2.3})$, the cost functional $J(t, x;u,v)$, for any $u\in
\mathcal{U}_{t,T},\ v\in \mathcal{V}_{t,T}$ , and the value function
$W(t, x)$ are monotonic in the following sense: for each $x, \bar{x}
\in \mathbb{R}, \ t \in[0, T],$
\[%
\begin{array}
[c]{llll} &  & \mathrm{(i)}\  \langle J(t, x; u,v)- J(t,\bar{x};
u,v), G(x -\bar
{x})\rangle \geq0,\ \mbox{P-a.s.}; & \\
&  & \mathrm{(ii)}\  \langle W(t, x)-W(t,\bar{x}),
G(x-\bar{x})\rangle \geq0. &
\end{array}
\]
\end{lemma}

\begin{proof} We define $(\hat{X}_s,\hat{Y}_s,\hat{Z}_s,\hat{K}_s): = (X^{t,x;u,v}_s-X^{t,\bar{x};u,v}_s,
Y^{t,x;u,v}_s-Y^{t,\bar{x};u,v}_s,
Z^{t,x;u,v}_s-Z^{t,\bar{x};u,v}_s,
K^{t,x;u,v}_s-K^{t,\bar{x};u,v}_s),$ and $\Delta l(s) =
l(s,\Pi^{t,x;u,v}_s,K^{t,x;u,v}_s,u_s,v_s)-l(s,\Pi^{t,\bar{x};u,v}_s,K^{t,\bar{x};u,v}_s,u_s,v_s),$
for $l = b,\sigma, f, A,$ respectively, and $\Delta h(s,e) =
h(s,\Pi^{t,x;u,v}_{s-},u_s,v_s,e)-h(s,\Pi^{t,\bar{x};u,v}_{s-},u_s,v_s,e)$.
Applying It\^{o}'s formula to $\langle \hat{Y}_s,
G\hat{X_s}\rangle$, we get immediately from $( \mathbf{H2.3})$,
$$\begin{array}{llll} &&\langle J(t, x; u,v) -
J(t,\bar{x}; u,v),G(x - \bar{x})\rangle =E[\langle Y^{t,x;u,v}_t - Y^{t,\bar{x};u,v}_t ,G(x -
\bar{x})\rangle| \mathcal {F}_t] \\
&=& E[\langle \Phi(X^{t,x;u,v}_T ) - \Phi(X^{t,\bar{x};u,v}_T
),G\hat{X}_T\rangle- \int_t^T\langle \Delta A(r),
(\hat{X}_r, \hat{Y}_r,\hat{Z}_r)\rangle dr- \int_t^T\int_E\langle G\Delta h(r,e),
\hat{K}_r(e)\rangle \lambda(de) dr |\mathcal {F}_t] \\
&\geq& E[\mu_1|G\hat{X}_T |^2 + \int_t^T(\beta_1|G\hat{X}_r|^2 +
\beta_2(|G^T\hat{Y}_r|^2 + |G^T\hat{Z}_r|^2)+\beta_3\int_E|G^T\hat{K}_r(e)|^2\lambda(de))dr |\mathcal {F}_t] \\
&\geq& 0, \  \mbox{for any }u\in \mathcal {U}_{t,T},\ v\in \mathcal
{V}_{t,T}.
\end{array}$$
From Remark \ref{re7.1}, $W(t,x)=\inf \limits_{\beta \in \mathcal
{B}_{t,T}}\sup \limits_{u\in \mathcal
{U}_{t,T}}E[J(t,x;u,\beta(u))].$ Setting $V(t,x,\beta)=\sup
\limits_{u\in \mathcal {U}_{t,T}}E[J(t,x;u,\beta(u))]$, we always
have $V(t,x,\beta)\geq E[J(t,x;u,\beta(u))],$ for any $u\in \mathcal
{U}_{t,T}$. On the other hand, for any $\varepsilon
>0,$ there exists $u^\varepsilon \in \mathcal {U}_{t,T}$, such that
$V(t,\bar{x},\beta) \leq
E[J(t,\bar{x};u^\varepsilon,\beta(u^\varepsilon))]+\varepsilon.$\\
If $G(x-\bar{x})\geq 0,$ then
$$\begin{array}{llll} \langle
V(t,x,\beta)-V(t,\bar{x},\beta),G(x-\bar{x})\rangle &\geq& \langle
E[J(t,x;u^\varepsilon,\beta(u^\varepsilon))-J(t,\bar{x};u^\varepsilon,\beta(u^\varepsilon))]-\varepsilon,G(x-\bar{x})\rangle \\
&\geq& -\varepsilon G(x-\overline{x}).
\end{array}$$
If $G(x-\bar{x})\leq 0,$ then similarly, there exists $u^\varepsilon\in\mathcal {U}_{t,T}$  such that $V(t,x,\beta)\leq E[J(t,x;u^\varepsilon,\beta(u^\varepsilon))]+\varepsilon,$ $$\begin{array}{llll} \langle
V(t,x,\beta)-V(t,\bar{x},\beta),G(x-\bar{x})\rangle
&=&  \langle V(t,\bar{x},\beta)-V(t,x,\beta),G(\bar{x}-x)\rangle  \\
&\geq& \langle
E[J(t,\bar{x};u^\varepsilon,\beta(u^\varepsilon))-J(t,x;u^\varepsilon,\beta(u^\varepsilon))]-\varepsilon,G(\bar{x}-{x})\rangle \\
&\geq&-\varepsilon G(\bar{x}-{x}).
\end{array}$$
Therefore, we always have $\langle
V(t,x,\beta)-V(t,\bar{x},\beta),G(x-\bar{x})\rangle \geq
-\varepsilon |G(x-\overline{x})|,$ for any $\beta \in \mathcal
{B}_{t,T},\ x,\bar{x}\in \mathbb{R}, \ t\in[0,T].$ Since
$W(t,x)=\inf \limits_{\beta \in \mathcal {B}_{t,T}}\sup
\limits_{u\in \mathcal {U}_{t,T}}E[J(t,x;u,\beta(u))]=\inf
\limits_{\beta \in \mathcal {B}_{t,T}}V(t,x,\beta),$ we have
$W(t,x)\leq V(t,x,\beta),$ for any $\beta \in \mathcal {B}_{t,T}.$
Moreover, for any $\varepsilon>0,$ there exists $\beta^\varepsilon
\in \mathcal {B}_{t,T}$ such that
$W(t,x)+\varepsilon \geq V(t,x,\beta^\varepsilon)$.\\
If $G(x-\bar{x})\geq 0,$ then
$$\begin{array}{llll} \langle W(t,x)-W(t,\bar{x}),G(x-\bar{x})\rangle
&\geq& \langle
V(t,x,\beta^\varepsilon)-V(t,\bar{x},\beta^\varepsilon)-\varepsilon,G(x-\bar{x})\rangle \geq-2\varepsilon G(x-\overline{x}).
\end{array}$$
If $G(x-\bar{x})\leq 0,$ then with $\beta^\varepsilon
\in \mathcal {B}_{t,T}$ such that $W(t,\bar{x})+\varepsilon\geq V(t,\bar{x},\beta^\varepsilon),$
$$\begin{array}{llll} \langle W(t,x)-W(t,\bar{x}),G(x-\bar{x})\rangle
&=&  \langle W(t,\bar{x})-W(t,{x}),G(\bar{x}-{x})\rangle \\
&\geq& \langle
V(t,\bar{x},\beta^\varepsilon)-V(t,{x},\beta^\varepsilon)-\varepsilon,G(\bar{x}-{x})\rangle \geq-2\varepsilon G(\bar{x}-{x}).
\end{array}$$
Therefore, we have $\langle W(t,x)-W(t,\bar{x}),G(x-\bar{x})\rangle
\geq-2\varepsilon |G(x-\overline{x})|,$ for any $x,\bar{x}\in
\mathbb{R}, \ t\in[0,T].$ Letting $\varepsilon \downarrow 0$,
$\langle W(t,x)-W(t,\bar{x}),G(x-\bar{x})\rangle \geq 0,$ for any
$x,\ \bar{x}\in \mathbb{R}, \ t\in[0,T].$\\
\end{proof}

\begin{remark} (1) From $(\mathbf{H2.3})$-{\rm(i)} we know that, if
$\sigma$ doesn't depend on $z$, then $\beta_2=0$; if $h$ doesn't
depend on $k$, then $\beta_3=0$. Furthermore, we assume:
\begin{description}
\item[$(\mathbf{H3.1})$] {\rm(i)} The Lipschitz constant $L_{\sigma}\geq0$ of $\sigma$
with respect to $z$ is sufficiently small, i.e., there exists some
$L_{\sigma }\geq0$ small enough such that, for all $t\in[0,T],\ u\in
U,\ v\in
V,\ x_{1},x_{2}\in \mathbb{R},\ y_{1},y_{2}\in \mathbb{R},\ z_{1},z_{2}%
\in \mathbb{R}^{d}$,
\[
|\sigma(t,x_{1},y_{1},z_{1},u,v)-\sigma(t,x_{2},y_{2},z_{2},u,v)|\leq
K(|x_{1}-x_{2}|+|y_{1}-y_{2}|)+L_{\sigma}|z_{1}-z_{2}|.
\]
And the Lipschitz coefficient $L_{h}(\cdot)$ of $h$ with respect to
$z$ is sufficiently small, i.e., there exists a
function
$L_{h}:E\rightarrow \mathbb{R}^{+}$ with $\tilde{C}_h:=\max(\sup\limits_{e\in E}L^2_h(e),\int_{E}L^{2}
_{h}(e)\lambda(de))<+\infty$ sufficiently small, and for all $t\in[0,T],\ u\in U,\ v\in
V,\ x_{1},x_{2}\in \mathbb{R},\ y_{1},y_{2}\in \mathbb{R},\ z_{1},z_{2}%
\in \mathbb{R}^{d}, \ e\in E$,
\[
|h(t,x_{1},y_{1},z_{1},u,v,e)-h(t,x_{2},y_{2},z_{2},u,v,e)|\leq \rho
(e)(|x_{1}-x_{2}|+|y_{1}-y_{2}|)+L_{h}(e)|z_{1}-z_{2}|.
\]
{\rm(ii)} For all
$t\in[0,T],\ u\in U,\ v\in
V$,
for any $(x,y,z)\in \mathbb{R}\times \mathbb{R}\times \mathbb{R}^{d},\  \mbox{P-a.s.},$ $|h(t,x,y,z,u,v,e)|\leq \rho(e)(1+|x|+|y|),$ where $\rho(e)=C(1\wedge|e|).$
\end{description}
(2) Notice that when $\sigma,\ h$ don't depend on $z$, it's clearly
that $(\mathbf{H3.1})$ always holds true.
\end{remark}

Now we adopt Peng's notion of stochastic backward semigroup to
discuss a generalized DPP for our stochastic differential game
(\ref{equ3.1}), (\ref{ee2}). The notation of stochastic backward
semigroup was first introduced by Peng~\cite{Pe4} to prove the DPP
for stochastic control problems. Similar to \cite{LW}, first we
define the family of (backward) semigroups associated with FBSDE
with jumps (\ref{equ3.1}).

For given initial data $(t,x)$, a number $0<\delta \leq T-t,$
admissible control processes $u(\cdot)\  \in \
\mathcal{U}_{t,t+\delta},\ v(\cdot )\  \in \
\mathcal{V}_{t,t+\delta}$ and a real-valued random function
$\Psi:\Omega \times \mathbb{R}\rightarrow \mathbb{R},\
\mathcal{F}_{t+\delta
}\otimes \mathcal{B}(\mathbb{R}) $-measurable such that $(\mathbf{H2.3}%
)$-(ii) holds, we put
\[
G_{s,t+\delta}^{t,x;u,v}[\Psi(t+\delta,\overline{X}_{t+\delta}^{t,x;u,v}%
)]:=\overline{Y}_{s}^{t,x;u,v}, \ s \in[t,t+\delta],
\]
where $(\overline{\Pi}_{s}^{t,x;u,v},\overline{K}_{s}^{t,x;u,v}%
):=(\overline{X}_{s}^{t,x;u,v},\overline{Y}_{s}^{t,x;u,v},\overline{Z}
_{s}^{t,x;u,v},\overline{K}_{s}^{t,x;u,v})_{t \leq s \leq
t+\delta}$ is the solution of the following FBSDE with the time
horizon $t+\delta$:
\begin{equation}
\label{equ3.2}\left \{
\begin{array}
[c]{llll}%
d\overline{X}_{s}^{t,x;u,v} & = & b(s,\overline{\Pi}_{s}^{t,x;u,v}%
,u_{s},v_{s})ds + \sigma(s,\overline{\Pi}_{s}^{t,x;u,v},u_{s},v_{s}%
)dB_{s}+\int_{E}h(s,\overline{\Pi}_{s-}^{t,x;u,v},u_{s}%
,v_{s},e)\tilde{\mu}(ds,de),& \\
d\overline{Y}_{s}^{t,x;u,v} & = & -f(s,\overline{\Pi}_{s}^{t,x;u,v}%
,\int_{E}\overline{K}_{s}^{t,x;u,v}(e)l(e)\lambda(de),u_{s},v_{s})ds + \overline{Z}_{s}^{t,x;u,v}%
dB_{s} +\int_{E}\overline{K}_{s}^{t,x;u,v}(e)\tilde{\mu}(ds,de),
  & \\
\overline{X}_{t}^{t,x;u,v} & = & x,  \qquad\qquad\qquad\qquad\qquad\qquad\qquad\qquad\qquad\qquad\qquad\qquad\qquad\qquad s\in[t,t+\delta], & \\
\overline{Y}_{t+\delta}^{t,x;u,v} & = & \Psi(t+\delta,\overline{X}
_{t+\delta}^{t,x;u,v}). &
\end{array}
\right.
\end{equation}
Here we write again
$\overline{\Pi}_{s-}^{t,x;u,v}=(\overline{X}_{s-}^{t,x;u,v},\overline{Y}_{s-}^{t,x;u,v},\overline{Z}_s^{t,x;u,v}).$

\begin{remark}
From Theorem 3.2 in \cite{LiWei-Lp}, we know FBSDE (\ref{equ3.2}) has a
unique solution $(\overline{X}^{t,x;u,v},\overline{Y}^{t,x;u,v},$
$\overline{Z}^{t,x;u,v},\overline{K}^{t,x;u,v})$ on the small
interval $[t,t+\delta]$, for any $0\leq \delta \leq \delta_{0}$,
where $\delta_{0}>0$ is independent of $(t,x)$ and the controls $u,\
v$.
\end{remark}

Then, for the solution
$(X^{t,x;u,v},Y^{t,x;u,v},Z^{t,x;u,v},K^{t,x;u,v})$ of FBSDE
(\ref{equ3.1}) we get
\[
G_{t,T}^{t,x;u,v}[\Phi(X_{T}^{t,x;u,v})] = G_{t,t+\delta}^{t,x;u,v}%
[Y_{t+\delta}^{t,x;u,v}].
\]
We also have
$$J(t,x;u,v)=Y_{t}^{t,x;u,v}=G_{t,T}^{t,x;u,v}[\Phi(X_{T}^{t,x;u,v}%
)]=G_{t,t+\delta}^{t,x;u,v}[Y_{t+\delta}^{t,x;u,v}]=G_{t,{t+\delta}}%
^{t,x;u,v}[J(t+\delta,X_{t+\delta}^{t,x;u,v};u,v)].$$

\begin{theorem}
\label{th3.1} Under the assumptions $(\mathbf{H2.2}),\
(\mathbf{H2.3} )$ and  $(\mathbf{H3.1})$, the lower value function
$W(t,x)$ satisfies the following DPP: there exists sufficiently
small $\delta_{0}>0$ such that, for any $0 \leq \delta\leq\delta_{0},\
t\in[0,T-\delta],\ x \in \mathbb{R},$
\[
W(t,x)=\mathop{\rm essinf}_{\beta \in \mathcal{B}_{t,t+\delta}}%
\mathop{\rm esssup}_{u\in \mathcal{U}_{t,t+\delta}}G_{t,{t+\delta}%
}^{t,x;u,\beta(u)}[W(t+\delta,\widetilde{X}_{t+\delta}^{t,x;u,\beta(u)})].
\]
\end{theorem}

The proof is given in the Appendix.

From the definition of our stochastic backward semigroup we know
here:
\[
G_{s,t+\delta}^{t,x;u,v}[W(t+\delta,\widetilde{X}_{t+\delta}^{t,x;u,v}%
)]:=\widetilde{Y}_{s}^{t,x;u,v},\ s\in[t,t+\delta],\ u(\cdot)\in
\mathcal{U}_{t,t+\delta},\ v(\cdot)\in \mathcal{V}_{t,t+\delta},
\]
where
$(\widetilde{\Pi}_{s}^{t,x;u,v},\widetilde{K}_{s}^{t,x;u,v})_{t \leq
s
\leq t+\delta}:=(\widetilde{X}_{s}^{t,x;u,v},\widetilde{Y}_{s}^{t,x;u,v}%
,\widetilde{Z}_{s}^{t,x;u,v},\widetilde{K}_{s}^{t,x;u,v})_{t \leq s
\leq t+\delta}$ is the solution of the following FBSDE with the time
horizon $t+\delta$:
\begin{equation}
\label{equ3.3}\left \{
\begin{array}
[c]{llll}%
d\widetilde{X}_{s}^{t,x;u,v} & = & b(s,\widetilde{\Pi}_{s}^{t,x;u,v}%
,u_{s},v_{s})ds + \sigma(s,\widetilde{\Pi}_{s}^{t,x;u,v},u_{s},v_{s}%
)dB_{s}+\int_{E}h(s,\widetilde{\Pi}_{s-}^{t,x;u,v},u_{s}%
,v_{s},e)\tilde{\mu}(ds,de), \  & \\
d\widetilde{Y}_{s}^{t,x;u,v} & = & -f(s,\widetilde{\Pi}_{s}^{t,x;u,v}%
,\int_{E}\widetilde{K}_{s}^{t,x;u,v}(e)l(
e)\lambda(de),u_{s},v_{s})ds + \widetilde{Z}_{s}^{t,x;u,v}%
dB_{s} +\int_{E}\widetilde{K}_{s}^{t,x;u,v}(e)\tilde{\mu}(ds,de),
& \\
\widetilde{X}_{t}^{t,x;u,v} & = & x, \qquad\qquad\qquad\qquad\qquad\qquad\qquad\qquad\qquad\qquad\qquad\qquad\qquad\qquad\ s\in[t,t+\delta], \\
\widetilde{Y}_{t+\delta}^{t,x;u,v} & = &
W(t+\delta,\widetilde{X}_{t+\delta }^{t,x;u,v}). &
\end{array}
\right.
\end{equation}
From Lemma \ref{l7} we get that the value function $W(t,x)$ is
Lipschitz
continuous in $x$, uniformly in $t$. Now with the help of DPP we can derive the ${\frac{1}{2}}%
$-H\"{o}lder continuity property of $W(t,x)$ in $t$.

\begin{theorem}
\label{th3.2} Under the assumptions $(\mathbf{H2.2}),\ (\mathbf{H2.3}%
),\ (\mathbf{H3.1})$, the lower value function $W(t,x)$ is ${\frac{1}{2}}%
$-H\"{o}lder continuous in $t$: there exists a constant $C$ such
that, for all $x\in \mathbb{R},\ t, t^{\prime}\in[0,T],$
\[
|W(t,x)-W(t^{\prime},x)| \leq C(1+|x|)|t-t^{\prime}|^{\frac{1}{2}}.
\]
\end{theorem}
\begin{proof} Let $(t,x)\in[0,T]\times \mathbb{R}$, and $0<\delta \leq (T-t)\wedge \delta_0$. Obviously, for the desired result, it is sufficient to prove the following inequality: for
some constant $C$,
$$-C(1+|x|)\delta^{{\frac{1}{2}}} \leq W(t,x)-W(t+\delta,x) \leq
C(1+|x|)\delta^{{\frac{1}{2}}}.$$ Next we only prove the second
inequality. From Remark \ref{re7.1}, we know that for every $\beta
\in \mathcal {B}_{t,t+\delta}$ there exists $u^\varepsilon \in
\mathcal {U}_{t,t+\delta}$, such that
$$  W(t,x)\leq G_{t,{t+\delta}}^{t,x;u^\varepsilon,\beta(u^\varepsilon)}[W(t+\delta,\widetilde{X}_{t+\delta}^{t,x;u^\varepsilon,\beta(u^\varepsilon)})]+\varepsilon.$$
Therefore, $W(t,x)-W(t+\delta,x) \leq
G_{t,{t+\delta}}^{t,x;u^\varepsilon,\beta(u^\varepsilon)}[W(t+\delta,\widetilde{X}_{t+\delta}^{t,x;u^\varepsilon,\beta(u^\varepsilon)})]+\varepsilon-W(t+\delta,x)
=I_\delta^1+I_\delta^2+\varepsilon,$\\
where$$\begin{array}{llll}I_\delta^1& =&
G_{t,{t+\delta}}^{t,x;u^\varepsilon,\beta(u^\varepsilon)}[W(t+\delta,\widetilde{X}_{t+\delta}^{t,x;u^\varepsilon,\beta(u^\varepsilon)})]-
G_{t,{t+\delta}}^{t,x;u^\varepsilon,\beta(u^\varepsilon)}[W(t+\delta,x)],\\
I_\delta^2 &=&
G_{t,{t+\delta}}^{t,x;u^\varepsilon,\beta(u^\varepsilon)}[W(t+\delta,x)]-
W(t+\delta,x).\end{array}$$ We know
$G_{s,t+\delta}^{t,x;u^\varepsilon,\beta(u^\varepsilon)}[W(t+\delta,x)]:=\hat{Y}_s^{t,x;u^\varepsilon,\beta(u^\varepsilon)},\
s\in[t,t+\delta],\  \beta \in \mathcal {B}_{t,t+\delta}$, where
$(\hat{\Pi}^{t,x;u^\varepsilon,\beta(u^\varepsilon)},
\hat{K}^{t,x;u^\varepsilon,\beta(u^\varepsilon)}):=(\hat{X}^{t,x;u^\varepsilon,\beta(u^\varepsilon)},$ $\hat{Y}^{t,x;u^\varepsilon,\beta(u^\varepsilon)},\hat{Z}^{t,x;u^\varepsilon,\beta(u^\varepsilon)},\hat{K}^{t,x;u^\varepsilon,\beta(u^\varepsilon)})$
is the solution of the following FBSDE with the time horizon
$t+\delta:$
\begin{equation}\label{equ000}
\left \{
\begin{array}{lll}
d\hat{X}_s^{t,x;u^\varepsilon,\beta(u^\varepsilon)}  &=&  b(s,\hat{\Pi}_s^{t,x;u^\varepsilon,\beta(u^\varepsilon)},u^\varepsilon_s,\beta_s(u^\varepsilon))ds + \sigma(s,\hat{\Pi}_s^{t,x;u^\varepsilon,\beta(u^\varepsilon)},u^\varepsilon_s,\beta_s(u^\varepsilon))dB_s\\
&&+\int_Eh(s,\hat{\Pi}_{s-}^{t,x;u^\varepsilon,\beta(u^\varepsilon)},u^\varepsilon_s,\beta_s(u^\varepsilon),e)\tilde{\mu}(ds,de), \\
d\hat{Y}_s^{t,x;u^\varepsilon,\beta(u^\varepsilon)} & =  & -f(s,\hat{\Pi}_s^{t,x;u^\varepsilon,\beta(u^\varepsilon)},\int_E\hat{K}_s^{t,x;u^\varepsilon,\beta(u^\varepsilon)}(e)l(e)\lambda(de),u^\varepsilon_s,\beta_s(u^\varepsilon))ds \\
&&+ \hat{Z}_s^{t,x;u^\varepsilon,\beta(u^\varepsilon)}dB_s+\int_E\hat{K}_s^{t,x;u^\varepsilon,\beta(u^\varepsilon)}(e)\tilde{\mu}(ds,de), \  \  \  \  \ s\in [t,t+\delta],\\
\hat{X}_t^{t,x;u^\varepsilon,\beta(u^\varepsilon)} & =&  x,\\
\hat{Y}_{t+\delta}^{t,x;u^\varepsilon,\beta(u^\varepsilon)}&  =&
W(t+\delta,x),
\end{array}
\right.
\end{equation}
where
$\hat{\Pi}_{s-}^{t,x;u^\varepsilon,\beta(u^\varepsilon)}=(\hat{X}_{s-}^{t,x;u^\varepsilon,\beta(u^\varepsilon)},\hat{Y}_{s-}^{t,x;u^\varepsilon,\beta(u^\varepsilon)},\hat{Z}_{s}^{t,x;u^\varepsilon,\beta(u^\varepsilon)}).$\\
By applying  It\^{o}'s formula to $e^{\gamma
s}|\widetilde{Y}_s^{t,x;u^\varepsilon,\beta(u^\varepsilon)}-\hat{Y}_s^{t,x;u^\varepsilon,\beta(u^\varepsilon)}|^2$,
 taking $\gamma$ large enough and using standard methods for BSDEs,
 we deduce
from Lemma \ref{l7}, Theorem 3.4-(ii) and Remark 3.7 in \cite{LiWei-Lp}
\begin{equation}\begin{array}{llll}
&&|\widetilde{Y}_t^{t,x;u^\varepsilon,\beta(u^\varepsilon)}-\hat{Y}_t^{t,x;u^\varepsilon,\beta(u^\varepsilon)}|^2\\
&\leq&CE[|W(t+\delta,\widetilde{X}_{t+\delta}^{t,x;u^\varepsilon,\beta(u^\varepsilon)})-W(t+\delta,x)|^2|\mathcal
{F}_t]
+CE[\int_t^{t+\delta}|\widetilde{X}_r^{t,x;u^\varepsilon,\beta(u^\varepsilon)}-\hat{X}_r^{t,x;u^\varepsilon,\beta(u^\varepsilon)}|^2dr|\mathcal
{F}_t]\\
&\leq&
CE[|\widetilde{X}_{t+\delta}^{t,x;u^\varepsilon,\beta(u^\varepsilon)}-x|^2|\mathcal
{F}_t]\\
&&+C\delta(E[\sup
\limits_{r\in[t,t+\delta]}|\widetilde{X}_r^{t,x;u^\varepsilon,\beta(u^\varepsilon)}-x|^2|\mathcal
{F}_t]+E[\sup
\limits_{r\in[t,t+\delta]}|\hat{X}_r^{t,x;u^\varepsilon,\beta(u^\varepsilon)}-x|^2|\mathcal
{F}_t] )\\
&\leq& C\delta(1+|x|^2).
\end{array}\end{equation}
Therefore, there exists some constant $C$ independent of the
controls such that
$$|I_\delta^1|=|\widetilde{Y}_s^{t,x;u^\varepsilon,\beta(u^\varepsilon)}-\hat{Y}_s^{t,x;u^\varepsilon,\beta(u^\varepsilon)}|\leq
C\delta^{\frac{1}{2}}(1+|x|).$$ Due to the definition of
$G_{t,t+\delta}^{t,x;u^\varepsilon,\beta(u^\varepsilon)}[\cdot]$, we
can rewrite the second term $I_\delta^2$ as follows
$$\begin{array}{llll}
|I_\delta^2|
&=&|E[\int_t^{t+\delta}f(s,\hat{\Pi}_s^{t,x;u^\varepsilon,\beta(u^\varepsilon)},
\int_E\hat{K}_s^{t,x;u,v}(e)l(e)\lambda(de),u_s^\varepsilon,\beta_s(u^\varepsilon))ds\mid
\mathcal
{F}_t]|\\
&\leq&
\delta^{\frac{1}{2}}E[\int_t^{t+\delta}|f(s,\hat{\Pi}_s^{t,x;u^\varepsilon,\beta(u^\varepsilon)},
\int_E\hat{K}_s^{t,x;u,v,\beta(u^\varepsilon)}(e)l(e)\lambda(de),u_s^\varepsilon,\beta_s(u^\varepsilon))|^2ds\mid
\mathcal
{F}_t]^{\frac{1}{2}}\\
&\leq&
C\delta^{\frac{1}{2}}E[\int_t^{t+\delta}(1+|\hat{X}_s^{t,x;u^\varepsilon,\beta(u^\varepsilon)}|+|\hat{Y}_s^{t,x;u^\varepsilon,\beta(u^\varepsilon)}|+|\hat{Z}_s^{t,x;u^\varepsilon,\beta(u^\varepsilon)}|\\
&&+\int_E|\hat{K}_s^{t,x;u,\beta(u^\varepsilon)}(e)|l(e)\lambda(de))^2ds\mid
\mathcal
{F}_t]^{\frac{1}{2}}\\
&\leq &C(1+|x|)\delta^{\frac{1}{2}}.\end{array}$$ For the latter
inequality, we have used  estimates (refer to Remark 3.7 in \cite{LiWei-Lp}) for
FBSDEs (\ref{equ000}) with jumps.\\
Therefore, $W(t,x)-W(t+\delta,x) \leq
C(1+|x|)\delta^{{\frac{1}{2}}}+\varepsilon. $ Letting $\varepsilon
\downarrow0$, we complete the proof. \\ \end{proof}

\section{{\protect \large Viscosity Solutions of Isaacs' equations with
integral-differential operators}}

 Now we consider the following fully coupled FBSDE with
jumps:
\begin{equation}
\label{equ4.1}\left \{
\begin{array}
[c]{llll}%
dX_{s}^{t,x;u,v} & = & b(s,\Pi_{s}^{t,x;u,v},u_{s},v_{s})ds
+\sigma(s,\Pi _{s}^{t,x;u,v},u_{s},v_{s}) dB_{s}+{\int_{E}}h(s,\Pi
_{s-}^{t,x;u,v},u_{s},v_{s},e)\tilde{\mu}(ds,de),   & \\
dY_{s}^{t,x;u,v} & = & -f(s,\Pi_{s}^{t,x;u,v},{\int_{E}}%
K_{s}^{t,x;u,v}(e)l(e)\lambda(de),u_{s},v_{s})ds +
Z_{s}^{t,x;u,v}dB_{s}+{\int_{E}}K_{s}^{t,x;u,v}(e)\tilde{\mu}(ds,de),
 & \\
X_{t}^{t,x;u,v} & = & x, \qquad\qquad\qquad\qquad\qquad\qquad\qquad\qquad\qquad\qquad\qquad\qquad\qquad\qquad\qquad s\in[t,T],& \\
Y_{T}^{t,x;u,v} & = & \Phi(X_{T}^{t,x;u,v}),&
\end{array}
\right.
\end{equation}
where
$\Pi_{s-}^{t,x;u,v}=(X_{s-}^{t,x;u,v},Y_{s-}^{t,x;u,v},Z_{s}^{t,x;u,v}),$
and the related second order integral-partial differential equations
of Isaacs' type which are the following PDEs combined with an
algebraic equation:
\begin{equation}
\label{equ4.2}\left \{
\begin{array}
[c]{ll} & \! \! \! \! \! \frac{\partial}{\partial t} W(t,x) +
H_{V}^{-}(t, x,
W(t,x))=0, \hskip 0.5cm (t,x)\in[0,T)\times\mathbb{R} ,\\
& \! \! \! \! \!V(t,x,u,v)=DW(t,x).\sigma(t,x,W(t,x),V(t,x,u,v),u,v),\\
& \! \! \! \! \! W(T,x) =\Phi(x),\  \  \  \ x\in \mathbb{R},
\end{array}
\right.
\end{equation}
and
\begin{equation}
\label{equ4.2.1}\left \{
\begin{array}
[c]{ll} & \! \! \! \! \! \frac{\partial}{\partial t} U(t,x) +
H_{V_{1}}^{+}(t, x,
U(t,x))=0, \hskip 0.5cm (t,x)\in[0,T)\times\mathbb{R} ,\\
& \! \! \! \! \!V_{1}(t,x,u,v)=DU(t,x).\sigma(t,x,U(t,x),V_{1}%
(t,x,u,v),u,v),\\
& \! \! \! \! \! U(T,x) =\Phi(x),\  \  \  \ x\in \mathbb{R},
\end{array}
\right.
\end{equation}
where
\[%
\begin{array}
[c]{lll}%
H_{V}^{-}(t, x, W(t,x)) & = & \mathop{\rm sup}\limits_{u \in U}\inf
\limits_{v
\in V} H(t, x, W(t,x), DW(t,x),D^2W(t,x), V(t,x,u,v),u,v),\\
H_{V_{1}}^{+}(t, x, U(t,x)) & = & \inf \limits_{v \in V}%
\mathop{\rm sup}\limits_{u \in U} H(t, x, U(t,x), DU(t,x),D^2U(t,x),
V_{1}(t,x,u,v),u,v),
\end{array}
\]
and $$
\begin{array}
[c]{lll}
&  & H(t, x, \phi,p,A,r,u,v)\\
&  & ={\frac{1}{2}}tr(\sigma \sigma^{T}(t, x, \phi(t,x), r,u,v)A)+p.b(t, x, \phi(t,x),r,u,v)\\
&  &\ \ \   +{\int_{E}}[\phi(t,x+h(t,x,\phi(t,x),r,u,v,e))-\phi(t,x)-p.h(t,x,\phi(t,x),r,u,v,e)]\lambda(de)\\
&  & \ \ \ +f(t, x, \phi(t,x), r, {\int_{E}}%
[\phi(t,x+h(t,x,\phi(t,x),r,u,v,e))-\phi(t,x)]l(e)\lambda
(de),u,v),
\end{array}
$$ where $t\in[0, T],\ x\in\mathbb{R},\ \phi\in C([0,T]\times\mathbb{R},{\mathbb{R}}),\ p\in\mathbb{R}^d,\ A\in \mathcal{S}^d,\ r\in\mathbb{R}^d,\ u\in U,\ v\in V,$ where $\mathcal{S}^d$ is the set of $d\times d$ symmetric matrices.

We will show that the value function $W(t, x)$ (resp., $U(t,x)$)
defined in (\ref{ee2}) (resp., (\ref{ee3})) is a viscosity solution
of the corresponding equation (\ref{equ4.2}) (resp.,
(\ref{equ4.2.1})). For this we use Peng's BSDE approach~\cite{Pe4}
developed originally for stochastic control problems of decoupled
FBSDEs. We first  give the definition of viscosity solution for this
kind of PDEs. For more information on viscosity solution, the reader
is referred to Crandall, Ishii  and Lions
\cite{CRANDALL-ISHII-LIONS}.

\begin{definition}
\label{def1-viscosity}\mbox{ } A real-valued continuous function
$W\in C([0,T]\times\mathbb{R} )$ is called \newline {\rm(i)} a
viscosity subsolution of equation (\ref{equ4.2}) if $W(T,x) \leq
\Phi (x),\  \mbox{for all}\ x \in\mathbb{R}$, and if for all
functions $\varphi \in C^{3}_{l, b}([0,T]\times\mathbb{R})$ and for
all $(t,x) \in[0,T) \times\mathbb{R}$ such that $W-\varphi$\ attains
a local maximum at $(t, x)$,
\begin{equation}
\label{viscosity1}%
\begin{array}
[c]{ll} & \! \! \! \! \! \frac{\partial \varphi}{\partial t} (t,x)
+\mathop{\rm sup}\limits_{u\in U}\inf \limits_{v\in
V}\{A^{u,v}\varphi
(t,x)+B^{\delta,u,v}(W,\varphi)(t,x)\\
& \! \! \! \! \!
+f(t,x,W(t,x),\psi(t,x,u,v),C^{\delta,u,v}(W,\varphi
)(t,x),u,v)\} \geq0,\  \mbox{for any }\delta>0,\ \mbox{and } \\
& \! \! \! \! \! \psi(t,x,u,v)=D\varphi(t,x).\sigma(t,x,W(t,x),\psi
(t,x,u,v),u,v),
\end{array}
\end{equation}
where
\[%
\begin{array}
[c]{llll} & \! \! \! \! \!
A^{u,v}\varphi(t,x)={\frac{1}{2}}tr(\sigma \sigma
^{T}(t,x,W(t,x),\psi(t,x,u,v),u,v)D^{2}\varphi(t,x))+D\varphi
(t,x).b(t,x,W(t,x),\psi(t,x,u,v),u,v), &  & \\
& \! \! \! \! \! B^{\delta,u,v}(W,\varphi)(t,x)=\int_{E_{\delta}}%
(\varphi(t,x+h(t,x,W(t,x),\psi(t,x,u,v),u,v,e))-\varphi(t,x) &  & \\
& \! \! \! \! \! \qquad \qquad \qquad \qquad \qquad \qquad \qquad
\qquad \qquad \qquad \qquad
-D\varphi(t,x).h(t,x,W(t,x),\psi(t,x,u,v),u,v,e))\lambda(de) &  &
\\
& \! \! \! \! \! \qquad \qquad \qquad \qquad\ \ +\int_{E^{c}_{\delta}}%
(W(t,x+h(t,x,W(t,x),\psi(t,x,u,v),u,v,e))-W(t,x) &  & \\
& \! \! \! \! \! \qquad \qquad \qquad \qquad \qquad \qquad \qquad
\qquad \qquad \qquad
\qquad-D\varphi(t,x).h(t,x,W(t,x),\psi(t,x,u,v),u,v,e))\lambda(de) &
&
\\
& \! \! \! \! \! C^{\delta,u,v}(W,\varphi)(t,x)=\int_{E_{\delta}}%
(\varphi(t,x+h(t,x,W(t,x),\psi(t,x,u,v),u,v,e))-\varphi(t,x))l(e)\lambda
(de) &  & \\
& \! \! \! \! \! \qquad \qquad \qquad \qquad\ \ +\int_{E^{c}_{\delta}}%
(W(t,x+h(t,x,W(t,x),\psi(t,x,u,v),u,v,e))-W(t,x))l(e)\lambda(de) &
&
\end{array}
\]
with $E_{\delta}=\{e\in E| |e|<\delta \}.$

\noindent {\rm(ii)} a viscosity supersolution of equation
(\ref{equ4.2}) if $W(T,x) \geq \Phi(x),\ \mbox{for all}\ x
\in\mathbb{R}$, and if for all functions $\varphi \in C^{3}_{l,
b}([0,T]\times\mathbb{R})$  and for all $(t,x) \in[0,T)
\times\mathbb{R}$\ such that $W-\varphi$\ attains a local minimum at
$(t, x)$,
\[%
\begin{array}
[c]{ll} & \! \! \! \! \! \frac{\partial \varphi}{\partial t} (t,x)
+\mathop{\rm sup}\limits_{u\in U}\inf \limits_{v\in
V}\{A^{u,v}\varphi
(t,x)+B^{\delta,u,v}(W,\varphi)(t,x)\\
& \! \! \! \! \!
+f(t,x,W(t,x),\psi(t,x,u,v),C^{\delta,u,v}(W,\varphi
)(t,x),u,v)\} \leq0,\  \mbox{for any }\delta>0,\ \mbox{and } \\
& \! \! \! \! \! \psi(t,x,u,v)=D\varphi(t,x).\sigma(t,x,W(t,x),\psi
(t,x,u,v),u,v),
\end{array}
\]
\noindent {\rm(iii)} a viscosity solution of equation (\ref{equ4.2})
if it is both a viscosity sub- and supersolution of equation
(\ref{equ4.2}).
\end{definition}

Similar to the results in \cite{BBP,BLH,LP}, we claim the following
result.

\begin{lemma}
In Definition \ref{def1-viscosity}, we can replace $B^{\delta,u,v}%
(W,\varphi)(t,x)$ and $C^{\delta,u,v}(W,\varphi)(t,x)$ by $B^{u,v}%
\varphi(t,x)$ and $C^{u,v}\varphi(t,x)$, respectively, where
\[%
\begin{array}
[c]{llll} &  &
B^{u,v}\varphi(t,x)=\int_{E}(\varphi(t,x+h(t,x,W(t,x),\psi
(t,x,u,v),u,v,e))-\varphi(t,x) & \\
&  &\ \qquad \qquad
\quad\ -D\varphi(t,x).h(t,x,W(t,x),\psi
(t,x,u,v),u,v,e))\lambda(de), & \\
&  & C^{u,v}\varphi(t,x)=\int_{E}(\varphi(t,x+h(t,x,W(t,x),\psi
(t,x,u,v),u,v,e))-\varphi(t,x))l(e)\lambda(de). &
\end{array}
\]
\end{lemma}

In what follows, we always assume $W(t,x)=\varphi(t,x)$, otherwise,
we can replace $\varphi$ by $\varphi-(W(t,x)-\varphi(t,x))$. From
now on, we shall use the following equivalent definition of
viscosity solution.

\begin{definition}
\label{def2-viscosity}\mbox{ } A real-valued continuous function
$W\in C([0,T]\times\mathbb{R} )$ is called \newline {\rm(i)} a
viscosity subsolution of equation (\ref{equ4.2}) if $W(T,x) \leq
\Phi (x),\  \mbox{for all}\ x \in\mathbb{R}$, and if for all
functions $\varphi \in C^{3}_{l, b}([0,T]\times\mathbb{R})$ and for
all $(t,x) \in[0,T) \times\mathbb{R}$ such that $W-\varphi$\ attains
a local maximum at $(t, x)$,
\[
\left \{
\begin{array}
[c]{ll} & \! \! \! \! \! \frac{\partial \varphi}{\partial t} (t,x)
+H_{\psi}^{-}(t, x,
\varphi(t,x)) \geq0,\\
& \! \! \! \! \! \mbox{where }\psi \mbox{ is the unique solution of
the following algebraic
equation:}\\
& \! \! \! \! \!
\psi(t,x,u,v)=D\varphi(t,x).\sigma(t,x,\varphi(t,x),\psi
(t,x,u,v),u,v).
\end{array}
\right.
\]
\noindent {\rm(ii)} a viscosity supersolution of equation
(\ref{equ4.2}) if $W(T,x) \geq \Phi(x),\ \mbox{for all}\ x
\in\mathbb{R}$, and if for all functions $\varphi \in C^{3}_{l,
b}([0,T]\times\mathbb{R})$ and for all $(t,x) \in[0,T)
\times\mathbb{R}$\ such that $W-\varphi$\ attains a local minimum at
$(t, x)$,
\[
\left \{
\begin{array}
[c]{ll} & \! \! \! \! \! \frac{\partial \varphi}{\partial t}
(t,x)+H_{\psi}^{-}(t, x,
\varphi(t,x)) \leq0,\\
& \! \! \! \! \! \mbox{where }\psi \mbox{ is the unique solution of
the following algebraic
equation:}\\
& \! \! \! \! \!
\psi(t,x,u,v)=D\varphi(t,x).\sigma(t,x,\varphi(t,x),\psi
(t,x,u,v),u,v).
\end{array}
\right.
\]
\noindent {\rm(iii)} a viscosity solution of equation (\ref{equ4.2})
if it is both a viscosity sub- and supersolution of equation
(\ref{equ4.2}).
\end{definition}

\begin{remark} When $\sigma$ depends on $z$, we need the test function $\varphi$ in Definitions \ref{def1-viscosity} and \ref{def2-viscosity}  satisfies the monotonicity condition  $(\mathbf{H2.3})$-{\rm(ii)}'
and   the following technical assumptions:
\begin{description}
\item[$(\mathbf{H4.1})$]
{\rm(i)} $\beta_{2}>0;$

{\rm(ii)} $G\sigma(s, x, y, z,u,v)$ is continuous in
$(s,u,v)$, uniformly with respect to $(x, y, z)\in \mathbb{R}%
\times{\mathbb{R}}\times{\mathbb{R}}^{d}$.
\end{description}
\end{remark}

\begin{theorem}
\label{th2} Under the assumptions $(\mathbf{H2.2}),\
(\mathbf{H2.3}),\ (\mathbf{H3.1}),\ (\mathbf{H4.1}) $, the lower
value function $W$ is a viscosity solution of (\ref{equ4.2}), the
upper value function $U$ is a viscosity solution of
(\ref{equ4.2.1}).
\end{theorem}

We only give the proof for $W$, similar to $U$. Before proving the
theorem,
 we first consider the following equation:
\begin{equation}
\label{equ4.3}\left \{
\begin{array}
[c]{llll}%
dX_{s}^{u,v} & = & b(s,\Pi_{s}^{u,v},u_{s},v_{s})ds + \sigma(s,\Pi_{s}%
^{u,v},u_{s},v_{s})dB_{s}+{\int_{E}}h(s,\Pi_{s-}^{u,v}%
,u_{s},v_{s},e)\tilde{\mu}(ds,de), \  \ s\in[t,t+\delta], & \\
dY_{s}^{u,v} & = & -f(s,\Pi_{s}^{u,v},\int_{E}K_{s}%
^{u,v}(e)l(e)\lambda(de),u_{s},v_{s})ds + Z_{s}^{u,v}%
dB_{s}+{\int_{E}}K_{s}^{u,v}(e)\tilde{\mu}(ds,de), & \\
X_{t}^{u,v} & = & x, & \\
Y_{t+\delta}^{u,v} & = & \varphi(t+\delta,X_{t+\delta}^{u,v}),\  \
0\leq \delta \leq T-t, &
\end{array}
\right.
\end{equation}
where $\Pi_{s-}^{u,v}=(X_{s-}^{u,v},Y_{s-}^{u,v},Z_{s}^{u,v}).$ From
Theorems 3.2 and 3.4 in \cite{LiWei-Lp}, we know there exists
$0<\bar{\delta}_{0}\leq T-t$ such that for any $0\leq \delta \leq
\bar{\delta}_{0},$  (\ref{equ4.3}) has a unique solution
$(\Pi_{s}^{u,v},K_{s}^{u,v})_{s\in[t,t+\delta]}:=(X_{s}^{u,v}%
,Y_{s}^{u,v},Z_{s}^{u,v},K_{s}^{u,v})_{s\in[t,t+\delta]}%
\in \mathcal{S}^{2}(t,t+\delta;\mathbb{R})\times \mathcal{S}%
^{2}(t,t+\delta;\mathbb{R})\times
\mathcal{H}^{2}(t,t+\delta;\mathbb{R}^{d})\times \mathcal{K}_{\lambda}^{2}(t,t+\delta;\mathbb{R})$, and for $p\geq2,$
\begin{equation}
\label{ee4.1}%
\begin{array}
[c]{llll}%
\mbox{(i)} &  & E[\mathop{\rm sup}\limits_{t\leq s\leq t+\delta
}|X_{s}^{u,v}|^{p}+\mathop{\rm sup}\limits_{t\leq s\leq t+\delta
}|Y_{s}^{u,v}|^{p}+(\int_{t}^{t+\delta}|Z_{s}%
^{u,v}|^{2}ds)^{\frac{p}{2}}\\
&& +(\int_{t}^{t+\delta}\int_{E}%
|K_{s}^{u,v}(e)|^{2}\lambda(de)ds)^{\frac{p}{2}}\mid
\mathcal{F}_{t}]\leq
C(1+|x|^{p}),\   \mbox{P-a.s.}; & \\
\mbox{(ii)} &  & E[\mathop{\rm sup}\limits_{t\leq s\leq t+\delta
}|X_{s}^{u,v}-x|^{p}\mid \mathcal{F}_{t}]\leq C\delta
(1+|x|^{p}),\  \  \mbox{P-a.s.}, & \\
\mbox{(iii)} &  & E[(\int_{t}^{t+\delta}|Z_{s}%
^{u,v}|^{2}ds)^{\frac{p}{2}}+(\int_{t}^{t+\delta
}\int_{E}|K_{s}^{u,v}(e)|^{2}\lambda(de)ds)^{\frac{p}{2}}%
\mid \mathcal{F}_{t}]\leq C\delta^{\frac{p}{2}}(1+|x|^{p}%
),\  \  \mbox{P-a.s.} &
\end{array}
\end{equation}
Define$$
\begin{array}
[c]{llll}
&  & L(s,x,y,z,k,u,v) & \\
& = & \frac{\partial}{\partial s}\varphi(s,
x)+{\frac{1}{2}}tr(\sigma
\sigma^{T}(s,x,y+\varphi(s,x),z,u,v)D^{2}\varphi(s, x)) +D\varphi
(s,x).b(s,x,y+\varphi(s,x),z,u,v) & \\
&  & +\int_{E}[\varphi(s,x+h(s,x,y+\varphi
(s,x),z,u,v,e))-\varphi(s,x)-D\varphi(s,x).h(s,x,y+\varphi
(s,x),z,u,v,e)]\lambda(de) & \\
&  & +f(s,x,y+\varphi(s,x),z,\int_{E}[k(e)+\varphi
(s,x+h(s,x,y+\varphi(s,x),z,u,v,e))-\varphi(s,x)]l(e)\lambda(de),u,v).
&
\end{array}
$$
 Note that, for all
$(s,x,y,z,k,u,v)\in[0,T]\times \mathbb{R}\times \mathbb{R}\times
\mathbb{R}^{d}\times \mathbb{R}\times U\times V,$
$$
L(s,x,y,z,k,u,v)\leq C(1+|x|^{2}+|y|^{2}+|z|^{2}+|k|),
$$
and  for all $(s,x,y,z,k,u,v), \
(s,\bar{x},\bar{y},\bar{z},\bar{k},u,v)\in[0,T]\times
\mathbb{R}\times \mathbb{R}\times \mathbb{R}^{d}\times
\mathbb{R}\times U\times V$,
\[
|L(s,x,y,z,k,u,v)-L(s,x,\bar{y},\bar{z},\bar{k},u,v)|\leq
C(1+|x|+|y|+|\bar
{y}|+|z|+|\bar{z}|)(|y-\bar{y}|+|z-\bar{z}|+|k-\bar{k}|).
\]
Now we set $Y^{1,u,v}_{s} = Y^{u,v}_{s}-\varphi(s,X^{u,v}_{s})$. By
applying It\^{o}'s formula to $\varphi(s,X_{s}^{u,v})$ and setting
\[%
\begin{array}
[c]{ll}%
Z_{s}^{1,u,v}=Z_{s}^{u,v}-D\varphi(s,X_{s}^{u,v}).\sigma(s,X_{s}^{u,v}%
,Y_{s}^{u,v},Z_{s}^{u,v},u_{s},v_{s}), & \\
K_{s}^{1,u,v}(e)=K_{s}^{u,v}(e)-\varphi(s,X_{s}^{u,v}+h(s,X_{s-}^{u,v},Y_{s-}%
^{u,v},Z_{s}^{u,v},u_{s},v_{s},e))+\varphi(s,X_{s}^{u,v}), &
\end{array}
\]
we obtain
\begin{equation}
\label{equ4.6}\left \{
\begin{array}
[c]{llll}%
Y_{s}^{1,u,v} & = & \int_{s}^{t+\delta}\Big[\frac{\partial
}{\partial r}\varphi(r, X_{r}^{u,v}) +{\frac{1}{2}}tr(\sigma \sigma^{T}%
(r,\Pi_{r}^{u,v},u_{r},v_{r})D^{2}\varphi(r, X_{r}^{u,v}))+D\varphi
(r,X_{r}^{u,v}).b(r,\Pi_{r}^{u,v},u_{r},v_{r}) & \\
&  & +f(r,\Pi_{r}^{u,v},\int_{E}K_{r}^{u,v}(e)l(e)\lambda(de),u_{r},v_{r})+\int_{E}[\varphi(r,X_{r}%
^{u,v}+h(r,\Pi_{r-}^{u,v},u_{r},v_{r},e))-\varphi(r,X_{r}^{u,v}) & \\
&  & -D\varphi(r,X_{r}^{u,v}).h(r,\Pi_{r-}^{u,v}%
,u_{r},v_{r},e)]l(e)\lambda(de)\Big]dr -\int
_{s}^{t+\delta}Z_{r}^{1,u,v}dB_{r}-\int_{s}^{t+\delta}\int_{E}K_{r}%
^{1,u,v}(e)\tilde{\mu}(dr,de) & \\
& = & \int_{s}^{t+\delta}L(r,X_{r}^{u,v},Y_{r}^{1,u,v}%
,Z_{r}^{u,v},K_{r}^{1,u,v},u_{r},v_{r})dr-\int_{s}^{t+\delta
}Z_{r}^{1,u,v}dB_{r}  -\int_{s}^{t+\delta}\int_{E}K_{r}%
^{1,u,v}(e)\tilde{\mu}(dr,de), & \\
Z_{s}^{1,u,v} & = & Z_{s}^{u,v}-D\varphi(s,X_{s}^{u,v}).\sigma(s,X_{s}%
^{u,v},Y_{s}^{1,u,v}+\varphi(s,X_{s}^{u,v}),Z_{s}^{u,v},u_{s},v_{s}%
),\  \ s\in[t,t+\delta]. &
\end{array}
\right.
\end{equation}
Obviously, (\ref{equ4.6}) has a solution $(Y_{s}^{1,u,v},Z_{s}^{1,u,v}%
,K_{s}^{1,u,v})\in \mathcal{S}^{2}(t,t+\delta;\mathbb{R}%
)\times \mathcal{H}^{2}(t,t+\delta;\mathbb{R}^{d})\times
\mathcal{K}_{\lambda}^{2}(t,t+\delta;\mathbb{R})$, because
(\ref{equ4.3}) has a unique solution $(X_{s}^{u,v},Y_{s}^{u,v},Z_{s}%
^{u,v},K_{s}^{u,v})_{s\in[ t, t+\delta]}$, and $Y_{s}^{1,u,v}=Y_{s}%
^{u,v}-\varphi(s,X_{s}^{u,v}),$ $Z_{s}^{1,u,v}=Z_{s}^{u,v}-D\varphi
(s,X_{s}^{u,v}).\sigma(s,X_{s}^{u,v},Y_{s}^{1,u,v}+\varphi(s,X_{s}^{u,v}%
),Z_{s}^{u,v},u_{s},v_{s}),$ $K_{s}^{1,u,v}(e)=K_{s}^{u,v}(e)-\varphi(s,X_{s}%
^{u,v}+h(s,X_{s-}^{u,v},$
$Y_{s-}^{u,v},Z_{s}^{u,v},u_{s},v_{s},e))+\varphi (s,X_{s}^{u,v}).$

Next we consider the following BSDE combined with an algebraic
equation:
\begin{equation}
\label{equ4.61}\left \{
\begin{array}
[c]{llll}%
dY_{s}^{2,u,v} & = &- L(s,x,0,\hat{Z}_{s}^{u,v},0,u_{s},v_{s})ds+
Z_{s}^{2,u,v}dB_{s}+ \int_{E}K_{s}^{2,u,v}(e)\tilde{\mu}(ds,de),
\  \  \ s\in[t,t+\delta], & \\
\hat{Z}_{s}^{u,v} & = & Z_{s}^{1,u,v} + D\varphi(s,x).\sigma(s,x,Y_{s}%
^{1,u,v}+\varphi(s,x),\hat{Z}_{s}^{u,v},u_{s},v_{s}), & \\
Y_{t+\delta}^{2,u,v} & = & 0, &
\end{array}
\right.
\end{equation}
where $u(\cdot) \in \mathcal{U}_{t,t+\delta},\ v(\cdot) \in \mathcal{V}%
_{t,t+\delta}.$

 We first recall the following condition (Remark \ref{re3.5}) and the
Representation Theorem for the algebraic equation  obtained in
\cite{L, LW}.

\begin{remark}
\label{re3.5} Without loss of generality, we may assume
$G=1\in \mathbb{R}$. When $\sigma$ depends on $z$, we
get the following results directly from the monotonicity condition
$( \mathbf{H2.3})$:
\begin{equation}
\label{ee4.3}%
\begin{array}
[c]{llll}%
\mathrm{(i)} &  & \langle
\sigma(t,x,y,z,u,v)-\sigma(t,x,y,\bar
{z},u,v),z-\bar{z}\rangle \leq-\beta_{2}|z-\bar{z}|^{2}; & \\
\mathrm{(ii)} &  & D\varphi(s,x) \geq0. &
\end{array}
\end{equation}
Indeed, {\rm(i)} follows from $( \mathbf{H2.3})$-{\rm(i)}, {\rm(ii)}
follows from $( \mathbf{H2.3})$-{\rm(ii)}' satisfied by $\varphi:
\langle \varphi(s,x)-\varphi (s,\bar{x}),G(x-\bar{x})\rangle \geq
0,\ \varphi \in C_{l,b}^{3}([0,T]\times \mathbb{R})$.
\end{remark}

We recall the following Representation Theorem given in \cite{L,
LW}.

\begin{theorem}
\label{pro4.1} Under the assumptions $(\mathbf{H2.3}),\ (\mathbf{H3.1}),\ (\mathbf{H4.1}),$  for any $s\in[0, T],\  \xi \in \mathbb{R}^{d},\ x\in
\mathbb{R},\ y\in \mathbb{R},\ u\in U,\ v\in V$, there exists a
unique $z$ such that
$z=\xi+D\varphi(s,x).\sigma(s,x,y+\varphi(s,x),z,u,v). $ That means,
the solution $z$ can be written as $z=q(s, x,y, \xi,u,v)$, where the
function $q$\ is Lipschitz with respect to $\ y, \  \xi,$ and $|q(s,
x,y, \xi,u,v)|\leq C(1+|x|+|y|+|\xi|).$ The constant $C$\ is
independent of $s,\ x,\ y, \  \xi,\ u,\ v$. Moreover, $z=q(s,
x,y,\xi,u,v)$ is continuous with respect to $(s,u,v)$.
\end{theorem}

\begin{remark}
 From the above Representation Theorem  we have the existence and the
uniqueness of the solutions of the algebraic equations in
(\ref{equ4.6}) and (\ref{equ4.61}).
\end{remark}

In order to complete the proof of Theorem \ref{th2}, we need the
following lemmas.

\begin{lemma}
\label{l4.1} For every $u\in \mathcal{U}_{t,t+\delta},\ v\in \mathcal{V}%
_{t,t+\delta},$ $0\leq \delta\leq\bar{\delta}_{0}$,
\[
E[\int_{t}^{t+\delta}(|Y_{s}^{1,u,v}|+|Z_{s}^{1,u,v}%
|+\int_{E}|K_{s}^{1,u,v}(e)|l(e)\lambda(e))ds\mid \mathcal{F}_{t}] \leq
C\delta^{\frac{5}{4}},\  \  \mbox{P-}a.s.,
\]
where the constant $C$ is independent of the controls $u,\ v$ and of
$\delta>0$.
\end{lemma}
\begin{proof} From
$Y_s^{1,u,v}=Y_s^{u,v}-\varphi(s,X_s^{u,v})$ we have
$$\begin{array}{llll}Y_s^{1,u,v}&=&\int_s^{t+\delta}f(r,X_r^{u,v},Y_r^{u,v},Z_r^{u,v},\int_EK_r^{u,v}(e)l(e)\lambda(de),u_r,v_r)dr-\int_s^{t+\delta}Z_r^{u,v}dB_r\\
&&-\int_s^{t+\delta}
\int_EK_s^{u,v}(e)\tilde{\mu}(ds,de)+\varphi(t+\delta,X_{t+\delta}^{u,v})-\varphi(s,X_s^{u,v}),\
s\in [t,t+\delta]. \end{array}$$ Then,
$$\begin{array}{llll}Y_s^{1,u,v}&=&E[\int_s^{t+\delta}f(r,X_r^{u,v},Y_r^{u,v},Z_r^{u,v},\int_EK_r^{u,v}(e)l(e)\lambda(de),u_r,v_r)dr\mid \mathcal
{F}_s]\\
&&+E[\varphi(t+\delta,X_{t+\delta}^{u,v})-\varphi(s,X_s^{u,v})\mid
\mathcal {F}_s],\  \ s\in [t,t+\delta].\end{array}$$ Therefore, from
(\ref{ee4.1})-(i)
\begin{equation}\label{ee4.4}\begin{array}{llll}
|Y_s^{1,u,v}|&\leq&
CE[\int_s^{t+\delta}(1+|X_r^{u,v}|+|Y_r^{u,v}|+|Z_r^{u,v}|+
\int_E|K_r^{u,v}(e)||l(e)|\lambda(de))dr\mid \mathcal
{F}_s]\\
&&+C\delta+CE[|X_{t+\delta}^u-X_s^u|\mid \mathcal {F}_s]\\
&\leq &
C\delta^{\frac{1}{2}}(E[\int_s^{t+\delta}(1+|X_r^{u,v}|^2+|Y_r^{u,v}|^2+|Z_r^{u,v}|^2+
\int_E|K_r^{u,v}(e)|^2\lambda(de))dr\mid \mathcal
{F}_s])^{\frac{1}{2}}\\
&&+C\delta+C\delta^{\frac{1}{2}}(1+|X_s^{u,v}|)\\
&\leq &C\delta^{\frac{1}{2}}(1+|X_s^{u,v}|),\  \  \mbox{P-a.s.},\
s\in [t,t+\delta],
\end{array}\end{equation}
and in particular, $$|Y_t^{1,u,v}|\leq C(1+|x|),\  \mbox{P-a.s.}$$
From
$$\begin{array}{llll}
&&Z_s^{1,u,v}=Z_s^{u,v}-D\varphi(s,X_s^{u,v}).\sigma(s,X_s^{u,v},Y_s^{u,v},Z_s^{u,v},u_s,v_s),\\
&&K_s^{1,u,v}(e)=K_s^{u,v}(e)-\varphi(s,X_s^{u,v}+h(s,X_{s-}^{u,v},Y_{s-}^{u,v},Z_s^{u,v},u_s,v_s,e)+\varphi(s,X_s^{u,v})\end{array}$$
we get
\begin{equation}\label{ee4.5}\begin{array}{llll}&&|Z_s^{1,u,v}|\leq
C(1+|X_s^{u,v}|+|Y_s^{u,v}|+|Z_s^{u,v}|),\\
&&|K_s^{1,u,v}(e)|\leq
C(1+|X_s^{u,v}|+|Y_s^{u,v}|+|Z_s^{u,v}|+|K_s^{u,v}(e)|),\
\mbox{P-a.s.},\ s\in[t,t+\delta].\end{array}\end{equation} Moreover,
from $Y_s^{1,u,v}=Y_s^{u,v}-\varphi(s,X_s^{u,v})$ combined with
(\ref{ee4.4}), we get
\begin{equation}\label{ee4.5.1}|Y_s^{u,v}|\leq
C(1+|X_s^{u,v}|),\  \mbox{P-a.s.},\ s\in[t,t+\delta].\end{equation}
On the other hand,  from Theorem \ref{pro4.1}, we know
$Z_s^{u,v}=q(s,X_s^{u,v},Y_s^{1,u,v},Z_s^{1,u,v},u_s,v_s),$  where
the function $q$ is Lipschitz in $y,z$, and of linear growth.
Putting $F(s,x,y,z,k,u,v)=L(s,x,y,q(s,x,y,z,u,v),k,u,v)$,
(\ref{equ4.6}) can be rewritten as follows: $s\in
[t,t+\delta]$,
$$\begin{array}{llll}Y_s^{1,u,v}=\int_s^{t+\delta}F(r,X_r^{u,v},Y_r^{1,u,v},Z_r^{1,u,v},K_r^{1,u,v},u_r,v_r)dr-\int_s^{t+\delta}Z_r^{1,u,v}dB_r-\int_s^{t+\delta} \int_EK_s^{1,u,v}(e)\tilde{\mu}(ds,de).
\end{array}$$ Therefore, from (\ref{ee4.4}), (\ref{ee4.5}) as well
as (\ref{ee4.1})-(i)-(iii)
\begin{equation}\label{ee4.6.1}\begin{array}{llll}
&&|Y_t^{1,u,v}|^2+E[ \int_t^{t+\delta}|Z_r^{1,u,v}|^2dr+
\int_t^{t+\delta} \int_E|K_r^{1,u,v}(e)|^2\lambda(de)dr\mid \mathcal
{F}_t]\\
&=&2E[\int_t^{t+\delta}Y_r^{1,u,v}F(r,X_r^{u,v},Y_r^{1,u,v},Z_r^{1,u,v},K_r^{1,u,v},u_r,v_r)dr\mid
\mathcal
{F}_t]\\
&\leq&CE[\int_t^{t+\delta}|Y_r^{1,u,v}|(1+|X_r^{u,v}|^2+|Y_r^{1,u,v}|^2+|Z_r^{1,u,v}|^2+
\int_E|K_r^{u,v}(e)||l(e)|\lambda(de))dr\mid \mathcal
{F}_t]\\
&\leq&CE[\int_t^{t+\delta}|Y_r^{1,u,v}|(1+|X_r^{u,v}|^2+|Y_r^{1,u,v}|^2+|Z_r^{u,v}|^2+
\int_E|K_r^{u,v}(e)||l(e)|\lambda(de))dr\mid \mathcal
{F}_t]\\
&\leq&C\delta^{\frac{1}{2}}E[\int_t^{t+\delta}(1+|X_r^{u,v}|^2+|X_r^{u,v}|^3)dr\mid
\mathcal {F}_t]+C\delta^{\frac{1}{2}}E[
\int_t^{t+\delta}(1+|X_r^{u,v}|)|Z_r^{u,v}|^2dr\mid \mathcal
{F}_t]\\
&&+C\delta^{\frac{1}{2}}E[\int_t^{t+\delta}
\int_E|K_r^{u,v}(e)|^2\lambda(de)dr\mid \mathcal
{F}_t]\\
&\leq&C\delta^\frac{3}{2}.
\end{array}\end{equation}
Therefore, $$\begin{array}{llll}
&&E[\int_t^{t+\delta}(|Y_r^{1,u,v}|+|Z_r^{1,u,v}|+
\int_E|K_r^{u,v}(e)|l(e)\lambda(de))ds\mid \mathcal
{F}_t]\\
&\leq&C\delta^{\frac{1}{2}}E[\int_t^{t+\delta}(1+|X_r^{u,v}|)dr\mid
\mathcal {F}_t]+C\delta^{\frac{1}{2}}(E[
\int_t^{t+\delta}|Z_r^{1,u,v}|^2dr\mid \mathcal
{F}_t])^{\frac{1}{2}}\\
&&+C\delta^{\frac{1}{2}}(E[\int_t^{t+\delta}
\int_E|K_r^{1,u,v}(e)|^2\lambda(de)dr\mid \mathcal
{F}_t])^{\frac{1}{2}}\\
&\leq&C\delta^{\frac{5}{4}},\  \  \mbox{P-a.s.},\ 0\leq \delta \leq
\delta_1.
\end{array}$$
\end{proof}

\begin{remark}
{\label{re4.2}} From (\ref{ee4.1})-{\rm(iii)} we know
\[
E[(\int_{t}^{t+\delta}|Z_{s}^{u,v}|^{2}ds)^{2}\mid
\mathcal{F}_{t}]\leq C\delta^{2},\  \  \mbox{P-a.s.}
\]
Then, from (\ref{ee4.1})-{\rm(i)}, (\ref{ee4.4}), (\ref{ee4.5}) and
(\ref{ee4.5.1}),
\begin{equation}\label{ee4.6.2}
E[(\int_{t}^{t+\delta}|Z_{s}^{1,u,v}|^{2}ds)^{2}\mid
\mathcal{F}_{t}]\leq C\delta^{2},\  \  \mbox{P-a.s.}
\end{equation}
\end{remark}

\begin{remark}
{\label{re4.3}} As $(Y_{s}^{1,u,v},Z_{s}^{1,u,v})\in \mathcal{S}%
^{2}(t,t+\delta; \mathbb{R})\times \mathcal{M}^{2}(t,t+\delta; \mathbb{R}^{d}%
)$, we have from the algebraic equation in (\ref{equ4.61}) that also  $\hat{Z}^{u,v}%
\in \mathcal{H}^{2}(t,t+\delta; \mathbb{R}^{d}).$

Moreover, from Theorem \ref{pro4.1}, we know $\hat{Z}_{s}^{u,v}%
=q(s,x,Y_{s}^{1,u,v},Z_{s}^{1,u,v},u_{s},v_{s})$. Therefore, similar
to Remark \ref{re4.2}, we have
\[%
\begin{array}
[c]{ll}%
\mathrm{{(i)}}\  \
E[\int_{t}^{t+\delta}|\hat{Z}_{s}^{u,v}|^{2}ds\mid
\mathcal{F}_{t}]\leq C\delta,\  \  \mbox{P-a.s.} & \\
\mathrm{{(ii)}}\  \ E[(\int_{t}^{t+\delta}|\hat{Z}_{s}^{u,v}|^{2}ds)^{2}%
\mid \mathcal{F}_{t}]\leq C\delta^{2},\  \  \mbox{P-a.s.} &
\end{array}
\]
\end{remark}

\begin{lemma}
\label{l4.2} For every $u \in \mathcal{U}_{t,t+\delta},\ v \in \mathcal{V}%
_{t,t+\delta}$, we have
\[
|Y_{t}^{1,u,v}-Y_{t}^{2,u,v}|\leq C\delta^{{\frac{5}{4}}}%
,\  \  \mbox{P-a.s.},\  \ 0\leq \delta\leq\bar{\delta}_{0},
\]
where $C$ is independent of the control processes $u,v $ and of
$\delta>0$.
\end{lemma}
\begin{proof} We set $g(s)=L(s, X_s^{u,v},0,Z_s^{u,v},K_s^{1,u,v},u_s,v_s)-L(s,
x,0,Z_s^{u,v},K_s^{1,u,v},u_s,v_s)$ and
$\rho_0(r)=(1+|x|^2+|Z_s^{u,v}|)(r+r^2),\ r\geq 0.$ Obviously,
$|g(s)|\leq C\rho_0(|X_s^{u,v}-x|),$ for $s\in [t,t+\delta],\
(t,x)\in [0,T)\times \mathbb{R},\ u\in \mathcal {U}_{t,t+\delta},\
v\in \mathcal {V}_{t,t+\delta}.$ Therefore, we have, from equations
(\ref{equ4.6}) and (\ref{equ4.61}),
\begin{equation}\begin{array}{llll}
&& |Y_t^{1,u,v}-Y_t^{2,u,v}|=|E[|Y_t^{1,u,v}-Y_t^{2,u,v}|\mid
\mathcal
{F}_t]| \\
&=&|E[\int_t^{t+\delta}(L(s,X_s^{u,v},Y_s^{1,u,v},Z_s^{u,v},K_s^{1,u,v},u_s,v_s)-L(s,x,0,\hat{Z}_s^{u,v},0,u_s,v_s))ds\mid
\mathcal
{F}_t]| \\
& \leq &CE[\int_t^{t+\delta}(\rho_0(|X_s^{u,v}-x|)+C(1+|X_s^{u,v}|+|Y_s^{1,u,v}|+|Z_s^{u,v}|)|Y_s^{1,u,v}|\\
&&+C(1+|x|+|Z_s^{u,v}|+|Z_s^{u,v}-\hat{Z}_s^{u,v}|)|Z_s^{u,v}-\hat{Z}_s^{u,v}|+\int_E|K_s^{1,u,v}(e)|l(e)\lambda(de))ds\mid
\mathcal
{F}_t]\\
& \leq
&C\delta^{\frac{5}{4}}+CE[\int_t^{t+\delta}|Z_s^{u,v}-\hat{Z}_s^{u,v}|ds\mid
\mathcal {F}_t]+CE[
\int_t^{t+\delta}|Z_s^{u,v}||Z_s^{u,v}-\hat{Z}_s^{u,v}|ds\mid
\mathcal
{F}_t]\\
&&+CE[\int_t^{t+\delta}|Z_s^{u,v}-\hat{Z}_s^{u,v}|^2ds\mid \mathcal
{F}_t]+CE[\int_t^{t+\delta} \int_E|K_s^{1,u,v}(e)|l(e)\lambda(de)ds\mid
\mathcal {F}_t].\end{array}\end{equation} Similar to Lemma 4.6 in
\cite{LW}, we have
$$|Z_s^{u,v}-\hat{Z}_s^{u,v}|\leq
C(1+|X_s^{u,v}|)|X_s^{u,v}-x|+C|X_s^{u,v}-x|(|Y_s^{1,u,v}|+|Z_s^{u,v}|),$$
and from (\ref{ee4.1}), we know
$$E[\int_t^{t+\delta}|Z_s^{u,v}-\hat{Z}_s^{u,v}|ds\mid \mathcal
{F}_t]\leq \delta^\frac{3}{2},\
E[\int_t^{t+\delta}|Z_s^{u,v}||Z_s^{u,v}-\hat{Z}_s^{u,v}|ds\mid
\mathcal {F}_t]\leq \delta^\frac{3}{2},\ E[\int_t^{t+\delta}|Z_s^{u,v}-\hat{Z}_s^{u,v}|^2ds\mid \mathcal
{F}_t]\leq \delta^\frac{3}{2}.$$
Then, Lemma \ref{l4.1} allows to complete the proof.\\
\end{proof}

We now consider the following equation
\begin{equation}
\label{equ4.7}\left \{
\begin{array}
[c]{llll}%
dY_{s}^{3,u,v} & = & -L(s,x,0,\psi(s,x,u_{s},v_{s}),0,u_{s},v_{s})ds+Z_{s}%
^{3,u,v}dB_{s}+\int_{E}K_{s}^{3,u,v}(e)\tilde{\mu}%
(ds,de),\ s\in[t,t+\delta], & \\
\psi(s,x,u,v) & = & D\varphi(s,x).\sigma(s,x,\varphi(s,x),\psi
(s,x,u,v),u,v),\ s\in[t,t+\delta], & \\
Y^{3,u,v}_{t+\delta} & = & 0. &
\end{array}
\right.
\end{equation}

\begin{lemma}
\label{l4.3} For every $u \in \mathcal{U}_{t,t+\delta},\ v \in \mathcal{V}%
_{t,t+\delta}$, we have
\[
|Y_{t}^{2,u,v}-Y_{t}^{3,u,v}|\leq C\delta^{{\frac{5}{4}}}, \
\mbox{P-a.s.},\  \ 0\leq \delta\leq\bar{\delta}_{0},
\]
where $C$ is independent of the control processes $u,v$ and of
$\delta>0$.
\end{lemma}
\begin{proof} $$\begin{array}{llll}
|Y_t^{2,u,v}-Y_t^{3,u,v}|&=&|E[\int_t^{t+\delta}(L(s,x,0,\hat{Z}_s^{u,v},0,u_s,v_s)-L(s,x,0,\psi(s,x,u_s,v_s),0,u_s,v_s))ds\mid
\mathcal
{F}_t]| \\
& \leq &CE[
\int_t^{t+\delta}(1+|x|+|\hat{Z}_s^{u,v}|)|\hat{Z}_s^{u,v}-\psi(s,x,u_s,v_s)|ds\mid
\mathcal {F}_t].\end{array}$$ From Remark \ref{re4.3} we have
$\hat{Z}_s^{u,v}=q(s, x, Y_s^{1,u,v}, Z_s^{1,u,v},u_s,v_s)$, and
from Theorem \ref{pro4.1}
$\psi(s,x,u_s,v_s)=q(s,x,0,0,u_s,v_s)$. Hence, we obtain
$|\hat{Z}_s^{u,v}-\psi(s,x,u_s,v_s)|\leq
C(|Y_s^{1,u,v}|+|Z_s^{1,u,v}|).$ Moreover, from (\ref{ee4.6.1}),
(\ref{ee4.6.2}) and Remark \ref{re4.3},
$$\begin{array}{llll} &&E[
\int_t^{t+\delta}|\hat{Z}_s^{u,v}||\hat{Z}_s^{u,v}-\psi(s,x,u_s,v_s)|ds\mid
\mathcal
{F}_t]\leq CE[\int_t^{t+\delta}|\hat{Z}_s^{u,v}|(|Y_s^{1,u,v}|+|Z_s^{1,u,v}|)ds\mid
\mathcal
{F}_t]\\
&\leq& C(E[\int_t^{t+\delta}|\hat{Z}_s^{u,v}|^2ds\mid \mathcal
{F}_t])^{\frac{1}{2}}(E[
\int_t^{t+\delta}(|Y_s^{1,u,v}|+|Z_s^{1,u,v}|)^2ds\mid \mathcal
{F}_t])^{\frac{1}{2}}\\
&\leq&C\delta^{\frac{5}{4}},\  \  \mbox{P-a.s.}
\end{array}$$
From Lemma \ref{l4.1}, we have $|Y_t^{2,u,v}-Y_t^{3,u,v}|\leq
C\delta^{{\frac{5}{4}}}.$\\
\end{proof}

\begin{lemma}
\label{l4.4} Let $Y_{0}(\cdot)$ be the solution of the following
ordinary differential equation combined with an algebraic equation:
\begin{equation}
\label{equ4.8}\left \{
\begin{array}
[c]{llll}%
-dY_{0}(s) & = & L_{0}(s,x,0)ds,\ s\in[t,t+\delta], & \\
\psi(s,x,u,v) & = & D\varphi(s,x).\sigma(s,x,\varphi(s,x),\psi
(s,x,u,v),u,v),\ s\in[t,t+\delta], & \\
Y_{0}(t+\delta) & = & 0, &
\end{array}
\right.
\end{equation}
where the function $L_{0}$ is defined by
\[
L_{0}(s,x,0)=\mathop{\rm sup}\limits_{u\in U}\inf \limits_{v\in V}%
L(s,x,0,\psi(s,x,u,v),0,u,v),\ (s,x)\in[t,t+\delta]\times
\mathbb{R}.
\]
Then, \mbox{P-a.s.}, $
Y_{0}(t)=\mathop{\rm esssup}\limits_{u\in \mathcal{U}_{t,t+\delta}}%
\mathop{\rm essinf}\limits_{v\in
\mathcal{V}_{t,t+\delta}}Y_{t}^{3,u,v}. $
\end{lemma}
 \begin{proof} First we introduce the function
$$L_1(s,x,0,u)=\inf \limits_{v\in V}L(s,x,0,\psi(s,x,u,v),0,u,v),\ (s,x,u)\in
[0,T]\times \mathbb{R}\times U$$ and consider the following
equation:
\begin{equation}\label{equ4.9}
\left \{
\begin{array}{llll}
-dY_s^{4,u} & = &
L_1(s,x,0,u_s)ds-Z_s^{4,u}dB_s-\int_EK_s^{4,u}(e)\tilde{\mu}(ds,de),\
\  \  s\in
[t,t+\delta],\\
\psi(s,x,u,v)&=&D\varphi(s,x).\sigma(s,x,\varphi(s,x),\psi(s,x,u,v),u,v),\\
Y^{4,u}_{t+\delta}& = & 0.
\end{array}
\right.\end{equation} From Lemma \ref{l1}, for every $u\in \mathcal
{U}_{t,t+\delta}$, there exists a unique solution
$(Y^{4,u},Z^{4,u},K^{4,u})$ to (\ref{equ4.9}). Moreover,
$$Y^{4,u}_t=\essinf_{v\in \mathcal
{V}_{t,t+\delta}}Y_t^{3,u,v}, \  \mbox{\mbox{P-a.s.}\ for every
}u\in \mathcal {U}_{t,t+\delta}.$$ In fact, from the definition of
$L_1$ and Lemma \ref{l2} (comparison theorem), we have
$$Y^{4,u}_t \leq Y_t^{3,u,v}, \  \mbox{\mbox{P-a.s.}\  \ for any }v\in \mathcal {V}_{t,t+\delta},\ \mbox{for every }u\in \mathcal
{U}_{t,t+\delta}.$$ On the other hand, there exists a measurable
function $v^4:[t,t+\delta]\times \mathbb{R}\times U\rightarrow V$
such that
$$L_1(s,x,0,u)=L(s,x,0,\psi(s,x,u,v^4(s,x,u)),0,u,v^4(s,x,u)),\  \mbox{for any
}s,x,u.$$ We then put $\widetilde{v}_s^4=v^4(s,x,u_s),\
s\in[t,t+\delta],$ and we observe that $\widetilde{v}^4\in \mathcal
{V}_{t,t+\delta}$ ($\widetilde{v}^4$ depends on $u\in\mathcal
{U}_{t,t+\delta}$) and
$$L_1(s,x,0,u_s)=L(s,x,0,\psi(s,x,u_s,\widetilde{v}_s^4),0,u_s,\widetilde{v}_s^4),\
s\in [t,t+\delta].$$ Consequently, from the uniqueness of the
solution of (\ref{equ4.9}) it follows that
$(Y^{4,u},Z^{4,u})=(Y^{3,u,\widetilde{v}^4},Z^{3,u,\widetilde{v}^4})$,
and in particular, $Y_t^{4,u}=Y_t^{3,u,\widetilde{v}^4},\
\mbox{P-a.s.}$\ This proves that $Y_t^{4,u}=\essinf \limits_{v\in
\mathcal {V}_{t,t+\delta}}Y_t^{3,u,v},\  \mbox{P-a.s.},$ for every
$u\in \mathcal {U}_{t,t+\delta}.$\\
Finally, since $L_0(s,x,0)=\sup \limits_{u\in {U}}L_1(s,x,0,u)$, by
a similar proof we show that
$$Y_0(t)=\esssup_{u\in \mathcal{U}_{t,t+\delta}}Y^{4,u}_t=\esssup_{u\in \mathcal{U}_{t,t+\delta}}\essinf_{v\in
\mathcal{V}_{t,t+\delta}}Y_t^{3,u,v}, \mbox{P-a.s.}$$
\end{proof}
We are now able to finish the proof of Theorem \ref{th2}.

\begin{proof} (1) First we will prove that $W$ is a viscosity supersolution. Obviously, $W(T,x)=\Phi(x),\ x\in \mathbb{R}.$  We suppose that $\varphi \in
C_{l,b}^3([0,T]\times \mathbb{R})$ satisfying the monotonicity
condition $(\mathbf{H2.3})$-{\rm(ii)}' if $\sigma$ depends on $z$  from $(\mathbf{H4.1})$ and $(t,x)\in[0,T)\times
\mathbb{R}$ are such that $W-\varphi$ attains its minimum at
$(t,x)$. Since $W$ is continuous and of at most linear growth, we
can replace the condition of a local minimum by that of a global one
in the definition of the viscosity supersolution. Moreover, without
loss of generality we may assume that $\varphi(t,x)=W(t,x).$ Due to
the DPP (Theorem \ref{th3.1}), we have $$
\varphi(t,x)=W(t,x)=\essinf_{\beta \in \mathcal
{B}_{t,t+\delta}}\esssup_{u\in \mathcal
{U}_{t,t+\delta}}G_{t,t+\delta}^{t,x;u,\beta(u)}[W(t+\delta,\widetilde{X}_{t+\delta}^{t,x;u,\beta(u)})],\
0\leq \delta \leq  \delta_0,$$ where
$\widetilde{X}^{t,x;u,\beta(u)}$ is defined by FBSDE (\ref{equ3.3}).
From $\varphi(s,y) \leq W(s,y), \ (s,y) \in [0,T) \times
\mathbb{R},$ and the monotonicity property of
$G_{t,t+\delta}^{t,x;u,\beta(u)}[\cdot]$ (see Theorem 3.3 in \cite{LiWei-Lp})
we obtain
\begin{equation} \label{ee4.2}\essinf_{\beta \in \mathcal
{B}_{t,t+\delta}}\esssup_{u\in \mathcal
{U}_{t,t+\delta}}{G_{t,t+\delta}^{t,x;u,\beta(u)}[\varphi(t+\delta,X_{t+\delta}^{u,\beta(u)})]-\varphi(t,x)}\leq
0.\end{equation} According to the definition of the backward
stochastic semigroup for fully coupled FBSDE with jumps, we have
$$G_{s,t+\delta}^{t,x;u,v}[\varphi(t+\delta,X_{t+\delta}^{u,v})]=Y_s^{u,v},\  \ s\in [t,t+\delta].$$
Moreover, $Y^{1,u,v}_s=Y^{u,v}_s-\varphi(s,X^{u,v}_s);$ therefore,
we have
$$\essinf_{\beta \in \mathcal{B}_{t,t+\delta}}\esssup_{u\in \mathcal{U}_{t,t+\delta}}Y_t^{1,u,\beta(u)}\leq0, \  \mbox{P-a.s.}$$
Thus, from the Lemmas \ref{l4.2} and \ref{l4.3} we get
$\essinf\limits_{\beta \in \mathcal{B}_{t,t+\delta}}\esssup\limits_{u\in \mathcal{U}_{t,t+\delta}}Y_t^{3,u,\beta(u)} \leq
C\delta^{\frac{5}{4}}, \  \mbox{P-a.s.}$ Consequently, since
$\essinf \limits_{v\in \mathcal {V}_{t,t+\delta}}Y_t^{3,u,v}\leq
Y_t^{3,u,\beta(u)},\  \beta \in \mathcal {B}_{t,t+\delta},$ we get
$$\esssup_{u\in \mathcal {U}_{t,t+\delta}}\essinf_{v\in \mathcal {V}_{t,t+\delta}}Y_t^{3,u,v}\leq
\essinf_{\beta \in \mathcal {B}_{t,t+\delta}}\esssup_{u\in \mathcal
{U}_{t,t+\delta}}Y_t^{3,u,\beta(u)}\leq C\delta^{\frac{5}{4}},\
\mbox{P-a.s.}$$ Thus, by Lemma \ref{l4.4}, $Y_0(t)\leq
C\delta^{\frac{5}{4}}, \ \mbox{P-a.s.},$ where $Y_0(\cdot)$ is the
unique solution of (\ref{equ4.8}). Consequently,
$$C\delta^\frac{1}{4}\geq \frac{1}{\delta}Y_0(t)=\frac{1}{\delta}\int_t^{t+\delta}L_0(s,x,0)ds,\  \delta>0.$$  Thanks to the continuity of $s\mapsto L_0(s,x,0)$
it follows that
$$\sup \limits_{u\in U}\inf \limits_{v\in V}L(t,x,0,\psi(t,x,u,v),0,u,v)=L_0(t,x,0)\leq0,$$
where
$\psi(t,x,u,v)=D\varphi(t,x).\sigma(t,x,\varphi(t,x),\psi(t,x,u,v),u,v)$.
From the definition of $L$ we see that $W$ is a viscosity
supersolution of (\ref{equ4.2}).

(2) Now we prove $W$ is a viscosity subsolution. For this we suppose
that $\varphi \in C_{l,b}^3([0,T]\times \mathbb{R})$ satisfying the
monotonicity condition $(\mathbf{H2.3})$-{\rm(ii)}' if $\sigma$ depends on $z$  from $(\mathbf{H4.1})$  and
$(t,x)\in[0,T)\times \mathbb{R}$ are such that $W-\varphi$ attains
its maximum at $(t,x)$. Without loss of generality we may also
suppose that $\varphi(t,x)=W(t,x).$ We must prove that $$\sup
\limits_{u\in \mathcal {U}}\inf \limits_{v\in
V}L(t,x,0,\psi(t,x,u,v),0,u,v)=L_0(t,x,0)\geq0,$$ where
$\psi(t,x,u,v)=D\varphi(t,x).\sigma(t,x,\varphi(t,x),\psi(t,x,u,v),u,v).$
Let us suppose that this is not true. Then there exists some
$\theta>0$ such that
\begin{equation}\label{ee5.1}
L_0(t,x,0)=\sup \limits_{u\in \mathcal {U}}\inf \limits_{v\in
V}L(t,x,0,\psi(t,x,u,v),0,u,v)\leq- \theta<0.\end{equation} and we
can find a measurable function $\gamma:U\rightarrow V$ such that
$$L(t,x,0,\psi(t,x,u,\gamma(u)),0,u,\gamma(u))\leq -\frac{3}{4}\theta,\  \mbox{for all }u\in
U.$$ Moreover, since
$L(\cdot,x,0,\psi(\cdot,x,\cdot,\cdot),0,\cdot,\cdot)$ is uniformly
continuous on $[0,T]\times U \times V,$ there exists some $T-t\geq
R>0$ such that \begin{equation}\label{ee5.2}
L(s,x,0,\psi(s,x,u,\gamma(u)),0,u,\gamma(u))\leq
-{\frac{1}{2}}\theta, \  \mbox{for all }u\in U \mbox{ and }|s-t|\leq
R.\end{equation} On the other hand, due to the DPP (see Theorem
\ref{th3.1}),
$$\varphi(t,x)=W(t,x)=\essinf_{\beta \in \mathcal {B}_{t,t+\delta}}\esssup_{u\in \mathcal
{U}_{t,t+\delta}}G_{t,{t+\delta}}^{t,x;u,\beta(u)}[W(t+\delta,\widetilde{X}_{t+\delta}^{t,x,u,\beta(u)})],\
0\leq \delta \leq \delta_0,$$  where
$\widetilde{X}^{t,x;u,\beta(u)}$ is defined by FBSDE (\ref{equ3.3}).
And from $W \leq \varphi$ and the monotonicity property of
$G_{t,{t+\delta}}^{t,x;u,\beta(u)}[\cdot]$ (see Theorem 3.3 in \cite{LiWei-Lp})
we obtain
$$\essinf\limits_{\beta \in \mathcal {B}_{t,t+\delta}}\esssup_{u\in \mathcal
{U}_{t,t+\delta}}G_{t,{t+\delta}}^{t,x;u,\beta(u)}[\varphi(t+\delta,{X}_{t+\delta}^{u,\beta(u)})]-\varphi(t,x)\geq0,
\  \mbox{P-a.s.},$$ Then, similar to (1), from the definition of
backward semigroup, we have
$$\essinf\limits_{\beta \in \mathcal {B}_{t,t+\delta}}\esssup\limits_{u\in
\mathcal{U}_{t,t+\delta}}Y_t^{1,u,\beta(u)}\geq 0, \
\mbox{P-a.s.},$$ and, in particular,
$\esssup\limits_{u\in \mathcal{U}_{t,t+\delta}}Y_t^{1,u,\gamma(u)}\geq 0,\  \mbox{P-a.s.}$
Here, by putting $\gamma_s(u)(\omega)=\gamma(u_s(\omega)),\
(s,\omega)\in [t,T]\times \Omega,$ we identify $\gamma$ as an
element of $\mathcal {B}_{t,t+\delta}.$ Given any $\varepsilon >0$
we can choose $u^\varepsilon \in \mathcal {U}_{t,t+\delta}$ such
that $Y_t^{1,u^\varepsilon,\gamma(u^\varepsilon)}\geq -\varepsilon
\delta.$ From the Lemmas \ref{l4.2} and \ref{l4.3} we have
\begin{equation}\label{ee5.4}
Y_t^{3,u^\varepsilon,\gamma(u^\varepsilon)}\geq
-C\delta^{\frac{5}{4}}-\varepsilon \delta,\
\mbox{P-a.s.}\end{equation} Moreover, from (\ref{equ4.7})
$$Y_t^{3,u^\varepsilon,\gamma(u^\varepsilon)}=E[\int_t^{t+\delta}L(s,x,0,\psi(s,x,u_s^\varepsilon,\gamma(u_s^\varepsilon)),0,u_s^\varepsilon,\gamma(u_s^\varepsilon))ds|\mathcal {F}_t],$$
and we get from (\ref{ee5.2})
\begin{equation}\label{ee5.3}
\begin{array}{llll}Y_t^{3,u^\varepsilon,\gamma(u^\varepsilon)}\leq E[
\int_t^{t+\delta}|L(s,x,0,\psi(s,x,u_s^\varepsilon,\gamma(u_s^\varepsilon)),0,u_s^\varepsilon,\gamma(u_s^\varepsilon))|ds |\mathcal {F}_t]\leq -{\frac{1}{2}}\theta \delta,\  \mbox{P-a.s.}
\end{array}
\end{equation}
From (\ref{ee5.4}) and (\ref{ee5.3}),
$-C\delta^\frac{1}{4}-\varepsilon \leq -{\frac{1}{2}}\theta,\
\mbox{P-a.s.}$ Letting $\delta \downarrow 0,$ and then $\varepsilon
\downarrow 0,$ we get that $\theta \leq 0, $ which yields a
contradiction. Therefore,
$$\sup \limits_{u\in U}\inf \limits_{v\in V}L(t,x,0,\psi(t,x,u,v),0,u,v)=L_0(t,x,0)\geq0,$$ and from
the definition of $L$ we see that $W$ is a viscosity supersolution
of (\ref{equ4.2}). Finally, from the above two steps, we derive that
$W$ is a viscosity solution of (\ref{equ4.2}).\\
\end{proof}

\section{Viscosity solution of Isaacs' equation: Uniqueness Theorem}

 In this section, we will state the uniqueness of the
viscosity solution of Isaacs' equation (\ref{equ4.2}), in which
$\sigma,\ h$ do not depend on $y,\ z,\ k$, i.e.
\begin{equation}
\label{equ5.1}\left \{
\begin{array}
[c]{ll} & \! \! \! \! \! \frac{\partial}{\partial t} W(t,x) +
H^{-}(t, x, W, DW,
D^{2}W)=0, \hskip 0.5cm (t,x)\in[0,T)\times{\mathbb{R}} ,\\
& \! \! \! \! \! W(T,x) =\Phi(x),\  \  \  \ x\in \mathbb{R},
\end{array}
\right.
\end{equation}
\begin{equation}
\left \{
\begin{array}
[c]{ll}
& \! \! \! \! \! \frac{\partial}{\partial t}U(t,x)+H^{+}(t,x,U,DU,D^{2}%
U)=0,\hskip0.5cm(t,x)\in \lbrack0,T)\times\mathbb{R},\\
& \! \! \! \! \!U(T,x)=\Phi(x),\  \  \  \ x\in \mathbb{R},
\end{array}
\right.  \label{equ5.11}%
\end{equation}
where
\[%
\begin{array}
[c]{lll}%
H^{-}(t,x,W,DW,D^{2}W) & = & \mathop{\rm sup}\limits_{u\in U}\inf
\limits_{v\in
V}H(t,x,W,DW,D^{2}W,u,v),\\
H^{+}(t,x,U,DU,D^{2}U) & = & \inf \limits_{v\in V}\mathop{\rm
sup}\limits_{u\in U}H(t,x,U,DU,D^{2}U,u,v)
\end{array}
\]
and
\[%
\begin{array}
[c]{lll}
&  & H(t,x,W,DW,D^{2}W)\\
&  & ={\frac{1}{2}}tr(\sigma
\sigma^{T}(t,x,u,v)D^{2}W(t,x))+DW(t,x).b(t,x,W
(t,x),DW(t,x).\sigma(t,x,u,v),u,v)\\
&  & \ \ \ \ \ +{ \int_{E}}[W(t,x+h(t,x,u,v,e))-W(t,x)-DW
(t,x).h(t,x,u,v,e)]\lambda(de)\\
&  & \ \ \ \ \ +f(t,x,W(t,x),DW(t,x).\sigma(t,x,u,v),{\int_{E}}%
[W(t,x+h(t,x,u,v,e))-W(t,x)]l(e)\lambda(de),u,v),
\end{array}
\]
where $t\in \lbrack0,T],\ x\in\mathbb{R}$. \\ Set
$$
\begin{array}
[c]{llll}%
&&\Theta=\{ \varphi \in C([0,T]\times \mathbb{R}):\exists \tilde{A}%
>0\mbox{ such that }\lim \limits_{|x|\rightarrow
\infty}\varphi(t,x)\exp
\{-\tilde{A}[\log((|x|^{2}+1)^{\frac{1}{2}})]^{2}%
\}=0,\\
&&\qquad\qquad\qquad\qquad\qquad\quad \mbox{ uniformly in }t\in \lbrack0,T]\}.
\end{array}
$$ We will prove the uniqueness for equation (\ref{equ5.1}) in
$\Theta$. The growth condition in $\Theta$ is weaker than the
polynomial growth but more restrictive than the exponential growth.
Barles, Buckdahn and Pardoux \cite{BBP}, Barles, Imbert
\cite{Barles-Imbert} introduced this growth condition (which is
optimal for the uniqueness and can not be weaken in general) to
prove the uniqueness of the viscosity solution of an
integral-partial differential equation associated with a decoupled
FBSDE with jumps but without controls. Next, by applying the method
developed in \cite{BBP} and \cite{Barles-Imbert}, we get the
uniqueness of the viscosity solution of (\ref{equ5.1}) in $\Theta$.
The proof for (\ref{equ5.11}) is similar.  On the other hand, since $\sigma$ does not depend on $z$, we don't need the assumption $(\mathbf{H4.1})$, that is, the test function $\varphi$ in Definitions \ref{def1-viscosity} and \ref{def2-viscosity} does not need to satisfy $(\mathbf{H2.3})$-(ii)' now.   First we present two
auxiliary lemmas.

\begin{lemma}
\label{le5.1}Let $w_{1}\in \Theta$ be a viscosity subsolution and $w_{2}%
\in \Theta$ be a viscosity supersolution of equation (\ref{equ5.1}).
Then the function $w:=w_{1}-w_{2}$ is a viscosity subsolution of the
equation
\begin{equation}
\label{equ5.2}\left \{
\begin{array}
[c]{llll} & \! \! \! \! \! \frac{\partial}{\partial t} w(t,x)
+\sup\limits_{u\in U,v\in V}\{{\frac{1}{2}}tr(\sigma
\sigma^{T}(t,x,u,v)D^{2}w)+Dw.b(t,x,w_1
(t,x),0,u,v)+B^{u,v}%
w(t,x)+\tilde{K}|w(t,x)|\\
& \! \! \! \! \!
+\tilde{K}|Dw(t,x).\sigma(t,x,u,v)|+\tilde{K}(C^{u,v}w(t,x))^{+}\}=0
&  &
\\
& \! \! \! \! \! w(T,x) =0,\  \  \  \ x\in \mathbb{R}, &  &
\end{array}
\right.
\end{equation}
where $\tilde{K}$ is a constant depending on the Lipschitz constants
of $b,\ \sigma,\ h,\ f$, which is uniformly in $(t,u,v)$.
\end{lemma}

\begin{proof} With the help of Lemma 7 in Nie \cite{NIE}, combined with Lemma
3.7 in \cite{BBP}, we can obtain the result.

Let $\varphi\in C_{l,b}^3([0,T]\times \mathbb{R})$, and let $(t_0,x_0)\in
(0,T)\times \mathbb{R}$ be a maximum point of $w-\varphi$ and
$w(t_0,x_0)=\varphi(t_0,x_0)$.  Without loss of generality assume
that $(t_0,x_0)$ is a strict global maximum point of $w-\varphi$,
otherwise, we can modify $\varphi$ outside a small neighborhood of $(t_0,x_0)$  if necessary. Also, the Lipschitz
property of $w_1$ and $w_2$ allows to assume that $D\varphi$ is
uniformly bounded: $|D\varphi|\leq K_{w_1,w_2}.$
 For a given $\varepsilon>0$, define $$\psi_\varepsilon(t,x,y)=w_1(t,x)-w_2(t,y)-\frac{|x-y|^2}{\varepsilon^2}-\varphi(t,x).$$

From Proposition 3.7 in \cite{CRANDALL-ISHII-LIONS}, we conclude
that there exists a sequence
$(t_\varepsilon,x_\varepsilon,y_\varepsilon)$ such that

(i) $(t_\varepsilon,x_\varepsilon,y_\varepsilon)$ is a global
maximum point of $\psi_\varepsilon$ in $([0,T]\times\bar{B}_R)^2$,
where $B_R$ is a ball with a large radius $R$;

(ii) $(t_\varepsilon,x_\varepsilon,y_\varepsilon)\rightarrow
(t_0,x_0,x_0)$, as $\varepsilon\rightarrow 0;$

(iii) $\frac{|x_\varepsilon-y_\varepsilon|^2}{\varepsilon^2}$ is
bounded and tends to $0$, when $\varepsilon\rightarrow0.$

Moreover, since $(t_0,x_0)$ is a strict global maximum point of
$w_1-w_2-\varphi$ and $\psi_\varepsilon(t_0,x_0,x_0)\leq
\psi_\varepsilon(t_\varepsilon,x_\varepsilon,y_\varepsilon),$ we
have
$$
\begin{array}{llll}
0&=&w_1(t_0,x_0)-w_2(t_0,x_0)-\varphi(t_0,x_0)=\psi_\varepsilon(t_0,x_0,x_0)\leq
\psi_\varepsilon(t_\varepsilon,x_\varepsilon,y_\varepsilon)\\
&=&w(t_\varepsilon,x_\varepsilon)-\varphi(t_\varepsilon,x_\varepsilon)+w_2(t_\varepsilon,x_\varepsilon)-w_2(t_\varepsilon,y_\varepsilon)-\frac{|x_\varepsilon-y_\varepsilon|^2}{\varepsilon^2}\\
&\leq&w_2(t_\varepsilon,x_\varepsilon)-w_2(t_\varepsilon,y_\varepsilon)-\frac{|x_\varepsilon-y_\varepsilon|^2}{\varepsilon^2},
\end{array}
$$
from which we know
$\frac{|x_\varepsilon-y_\varepsilon|}{\varepsilon^2}\leq|\frac{w_2(t_\varepsilon,x_\varepsilon)-w_2(t_\varepsilon,y_\varepsilon)}{x_\varepsilon-y_\varepsilon}|\leq
K_{w_2}.$

Furthermore, from Theorem 8.3 in \cite{CRANDALL-ISHII-LIONS}, for
any $\alpha>0$, there exists $(X^\alpha,Y^\alpha)\in\mathcal{S}^d\times \mathcal{S}^d,\ c^\alpha\in\mathbb{R}$ such that
$$
\begin{array}{llll}
&&(c^\alpha+\frac{\partial}{\partial t}\varphi(t_\varepsilon,x_\varepsilon),\frac{2(x_\varepsilon-y_\varepsilon)}{\varepsilon^2}+D\varphi(t_\varepsilon,x_\varepsilon),X^\alpha)\in\bar{\mathcal {P}}^{2,+}w_1(t_\varepsilon,x_\varepsilon),\\
&&(c^\alpha,\frac{2(x_\varepsilon-y_\varepsilon)}{\varepsilon^2},Y^\alpha)\in\bar{\mathcal {P}}^{2,-}w_2(t_\varepsilon,x_\varepsilon),\\
\end{array}
$$
and $$\left(\begin{array}{ccc} X^\alpha & 0 \\
            0 & -Y^\alpha  \end{array}\right) \leq A+\delta A^2,$$
where $A=\left(\begin{array}{ccc} D^2\varphi(t_\varepsilon,x_\varepsilon)+\frac{2}{\varepsilon^2} & -\frac{2}{\varepsilon^2} \\
            -\frac{2}{\varepsilon^2} & \frac{2}{\varepsilon^2}  \end{array}\right).$

Since $w_1$  and $w_2$ are sub- and supersolution of (\ref{equ5.1}), respectively,
from the definitions of the viscosity solution, we have, for the
sufficiently small $\delta$,
\begin{equation}\label{uniqueness-w1}
\begin{array}{llll}
&&c^\alpha+\frac{\partial\varphi}{\partial
t}(t_\varepsilon,x_\varepsilon)+\sup\limits_{u\in U}\inf\limits_{v\in V}\{\frac{1}{2}tr(\sigma\sigma^T(t_\varepsilon,x_\varepsilon,u,v)X^\alpha)\\
&&+\langle
b(t_\varepsilon,x_\varepsilon,w_1(t_\varepsilon,x_\varepsilon),(\frac{2(x_\varepsilon-y_\varepsilon)}{\varepsilon^2}+D\varphi(t_\varepsilon,x_\varepsilon))\sigma(t_\varepsilon,x_\varepsilon,u,v),u,v),\frac{2(x_\varepsilon-y_\varepsilon)}{\varepsilon^2}+D\varphi(t_\varepsilon,x_\varepsilon)
 \rangle\\
&&+\int_{E_\delta}\frac{|h(t_\varepsilon,x_\varepsilon,u,v,e)|^2}{\varepsilon^2}\lambda(de)+\int_{E_\delta}(\varphi(t_\varepsilon,x_\varepsilon+h(t_\varepsilon,x_\varepsilon,u,v,e))-\varphi(t_\varepsilon,x_\varepsilon)-D\varphi(t_\varepsilon,x_\varepsilon)h(t_\varepsilon,x_\varepsilon,u,v,e))\lambda(de)   \\
&&+\int_{E^c_\delta}(w_1(t_\varepsilon,x_\varepsilon+h(t_\varepsilon,x_\varepsilon,u,v,e))-w_1(t_\varepsilon,x_\varepsilon)-(\frac{2(x_\varepsilon-y_\varepsilon)}{\varepsilon^2}+D\varphi(t_\varepsilon,x_\varepsilon))h(t_\varepsilon,x_\varepsilon,u,v,e))\lambda(de)   \\
 &&+f(t_\varepsilon,x_\varepsilon,w_1(t_\varepsilon,x_\varepsilon),(\frac{2(x_\varepsilon-y_\varepsilon)}{\varepsilon^2}+D\varphi(t_\varepsilon,x_\varepsilon))\sigma(t_\varepsilon,x_\varepsilon,u,v),B_1^\delta,u,v)\}\geq0,
\end{array}
\end{equation}
where $$\begin{array}{llll}B_1^\delta&=&\int_{E_\delta}(\frac{2(x_\varepsilon-y_\varepsilon)}{\varepsilon^2}h(t_\varepsilon,x_\varepsilon,u,v,e)+\frac{|h(t_\varepsilon,x_\varepsilon,u,v,e)|^2}{\varepsilon^2})l(e)\lambda(de)\\
&&+\int_{E_\delta}(\varphi(t_\varepsilon,x_\varepsilon+h(t_\varepsilon,x_\varepsilon,u,v,e))-\varphi(t_\varepsilon,x_\varepsilon))l(e)\lambda(de)\\
 &&
 +\int_{E^c_\delta}(w_1(t_\varepsilon,x_\varepsilon+h(t_\varepsilon,x_\varepsilon,u,v,e))-w_1(t_\varepsilon,x_\varepsilon))l(e)\lambda(de),\end{array}$$
and
\begin{equation}\label{uniqueness-w2}
\begin{array}{llll}
&&c^\alpha+\sup\limits_{u\in U}\inf\limits_{v\in
V}\{\frac{1}{2}tr(\sigma\sigma^T(t_\varepsilon,y_\varepsilon,u,v)Y^\alpha)+\langle
b(t_\varepsilon,y_\varepsilon,w_2(t_\varepsilon,y_\varepsilon),\frac{2(x_\varepsilon-y_\varepsilon)}{\varepsilon^2}\sigma(t_\varepsilon,y_\varepsilon,u,v),u,v),\frac{2(x_\varepsilon-y_\varepsilon)}{\varepsilon^2}
 \rangle\\
 &&-\int_{E_\delta}\frac{|h(t_\varepsilon,y_\varepsilon,u,v,e)|^2}{\varepsilon^2}\lambda(de)+\int_{E^c_\delta}(w_2(t_\varepsilon,y_\varepsilon+h(t_\varepsilon,y_\varepsilon,u,v,e))-w_2(t_\varepsilon,y_\varepsilon)-\frac{2(x_\varepsilon-y_\varepsilon)}{\varepsilon^2}h(t_\varepsilon,y_\varepsilon,u,v,e))\lambda(de)   \\
 &&+f(t_\varepsilon,y_\varepsilon,w_2(t_\varepsilon,y_\varepsilon),\frac{2(x_\varepsilon-y_\varepsilon)}{\varepsilon^2}\sigma(t_\varepsilon,y_\varepsilon,u,v),B_2^\delta,u,v)\}\leq0,
\end{array}\end{equation}
where
$$\begin{array}{llll}B_2^\delta&=&\int_{E_\delta}(-\frac{2(x_\varepsilon-y_\varepsilon)}{\varepsilon^2}h(t_\varepsilon,y_\varepsilon,u,v,e)-\frac{|h(t_\varepsilon,y_\varepsilon,u,v,e)|^2}{\varepsilon^2})l(e)\lambda(de)\\
 &&
 +\int_{E^c_\delta}(w_2(t_\varepsilon,y_\varepsilon+h(t_\varepsilon,y_\varepsilon,u,v,e))-w_2(t_\varepsilon,y_\varepsilon))l(e)\lambda(de),\end{array}$$
Set
$$\begin{array}{llll}
&&I_{1,u,v}^{\varepsilon,\alpha}:=\frac{1}{2}tr(\sigma\sigma^T(t_\varepsilon,x_\varepsilon,u,v)X^\alpha)-\frac{1}{2}tr(\sigma\sigma^T(t_\varepsilon,y_\varepsilon,u,v)Y^\alpha),\\
&&I_{2,u,v}^{\varepsilon,\alpha}:=\langle
b(t_\varepsilon,x_\varepsilon,w_1(t_\varepsilon,x_\varepsilon),(\frac{2(x_\varepsilon-y_\varepsilon)}{\varepsilon^2}+D\varphi(t_\varepsilon,x_\varepsilon))\sigma(t_\varepsilon,x_\varepsilon,u,v),u,v),\frac{2(x_\varepsilon-y_\varepsilon)}{\varepsilon^2}
 +D\varphi(t_\varepsilon,x_\varepsilon)\rangle\\
 &&\qquad\quad-\langle
b(t_\varepsilon,y_\varepsilon,w_2(t_\varepsilon,y_\varepsilon),\frac{2(x_\varepsilon-y_\varepsilon)}{\varepsilon^2}\sigma(t_\varepsilon,y_\varepsilon,u,v),u,v),\frac{2(x_\varepsilon-y_\varepsilon)}{\varepsilon^2}\rangle,\\
&&I_{3,u,v}^{\varepsilon,\alpha,\delta}:=\\
&&\int_{E_\delta}\frac{|h(t_\varepsilon,x_\varepsilon,u,v,e)|^2}{\varepsilon^2}\lambda(de)+\int_{E_\delta}(\varphi(t_\varepsilon,x_\varepsilon+h(t_\varepsilon,x_\varepsilon,u,v,e))-\varphi(t_\varepsilon,x_\varepsilon)-D\varphi(t_\varepsilon,x_\varepsilon)h(t_\varepsilon,x_\varepsilon,u,v,e))\lambda(de)   \\
&&+\int_{E^c_\delta}(w_1(t_\varepsilon,x_\varepsilon+h(t_\varepsilon,x_\varepsilon,u,v,e))-w_1(t_\varepsilon,x_\varepsilon)-(\frac{2(x_\varepsilon-y_\varepsilon)}{\varepsilon^2}+D\varphi(t_\varepsilon,x_\varepsilon))h(t_\varepsilon,x_\varepsilon,u,v,e))\lambda(de)   \\
&&+\int_{E_\delta}\frac{|h(t_\varepsilon,y_\varepsilon,u,v,e)|^2}{\varepsilon^2}\lambda(de)-\int_{E^c_\delta}(w_2(t_\varepsilon,y_\varepsilon+h(t_\varepsilon,y_\varepsilon,u,v,e))-w_2(t_\varepsilon,y_\varepsilon)-\frac{2(x_\varepsilon-y_\varepsilon)}{\varepsilon^2}h(t_\varepsilon,y_\varepsilon,u,v,e))\lambda(de),\\
&&I_{4,u,v}^{\varepsilon,\alpha,\delta}:=f(t_\varepsilon,x_\varepsilon,w_1(t_\varepsilon,x_\varepsilon),(\frac{2(x_\varepsilon-y_\varepsilon)}{\varepsilon^2}+D\varphi(t_\varepsilon,x_\varepsilon))\sigma(t_\varepsilon,x_\varepsilon,u,v),B_1^\delta,u,v)\\
 &&\qquad\quad-f(t_\varepsilon,y_\varepsilon,w_2(t_\varepsilon,y_\varepsilon),\frac{2(x_\varepsilon-y_\varepsilon)}{\varepsilon^2}\sigma(t_\varepsilon,y_\varepsilon,u,v),B_2^\delta,u,v),
\end{array}
$$
Then, from (\ref{uniqueness-w1}) and (\ref{uniqueness-w2}), we know
\begin{equation}\label{uniquenessw1-w2}\frac{\partial}{\partial
t}\varphi(t_\varepsilon,x_\varepsilon)+\sup\limits_{u\in U,v\in
V}\{I_{1,u,v}^{\varepsilon,\alpha}+I_{2,u,v}^{\varepsilon,\alpha}+I_{3,u,v}^{\varepsilon,\alpha,\delta}+I_{4,u,v}^{\varepsilon,\alpha,\delta}\}\geq0.\end{equation}
 For any $u\in U,v\in
V$, we want to prove the following result:
$$\begin{array}{llll}
&&\lim\limits_{\varepsilon\rightarrow0}\lim\limits_{\alpha\rightarrow0}\lim\limits_{\delta\rightarrow0}\{I_{1,u,v}^{\varepsilon,\alpha}+I_{2,u,v}^{\varepsilon,\alpha}+I_{3,u,v}^{\varepsilon,\alpha,\delta}+I_{4,u,v}^{\varepsilon,\alpha,\delta}\}\\
&\leq&\frac{1}{2}tr(\sigma\sigma^T(t_0,x_0,u,v)D^2\varphi(t_0,x_0))+\langle
b(t_0,x_0,w_1(t_0,x_0),0,u,v),D\varphi(t_0,x_0)\rangle+B^{u,v}\varphi(t_0,x_0)\\
&&+\tilde{K}\{|\varphi(t_0,x_0)|+|D\varphi(t_0,x_0)|\cdot|\sigma(t_0,x_0,u,v)|+(C^{u,v}\varphi(t_0,x_0))^+\}.
\end{array}$$
Similar to the proof of Lemma 7 in  Nie \cite{NIE}, we obtain
$$\begin{array}{llll}
I_{1,u,v}^{\varepsilon,\alpha}=\frac{1}{2}tr(\sigma\sigma^T(t_\varepsilon,x_\varepsilon,u,v)X^\alpha)-\frac{1}{2}tr(\sigma\sigma^T(t_\varepsilon,y_\varepsilon,u,v)Y^\alpha)\leq
\frac{1}{2}tr(\sigma\sigma^T(t_\varepsilon,x_\varepsilon,u,v)D^2\varphi(t_\varepsilon,x_\varepsilon))+K_\sigma^2\frac{|x_\varepsilon-y_\varepsilon|^2}{\varepsilon^2}.
\end{array}$$
For $I_{2,u,v}^{\varepsilon,\alpha}$, we have
$$\begin{array}{llll}
I_{2,u,v}^{\varepsilon,\alpha}&=&\langle
b(t_\varepsilon,x_\varepsilon,w_1(t_\varepsilon,x_\varepsilon),(\frac{2(x_\varepsilon-y_\varepsilon)}{\varepsilon^2}+D\varphi(t_\varepsilon,x_\varepsilon))\sigma(t_\varepsilon,x_\varepsilon,u,v),u,v),\frac{2(x_\varepsilon-y_\varepsilon)}{\varepsilon^2}
 +D\varphi(t_\varepsilon,x_\varepsilon)\rangle\\
 &&-\langle
b(t_\varepsilon,y_\varepsilon,w_2(t_\varepsilon,y_\varepsilon),\frac{2(x_\varepsilon-y_\varepsilon)}{\varepsilon^2}\sigma(t_\varepsilon,y_\varepsilon,u,v),u,v),\frac{2(x_\varepsilon-y_\varepsilon)}{\varepsilon^2}\rangle\\
 &=&\langle
 b(t_\varepsilon,x_\varepsilon,w_1(t_\varepsilon,y_\varepsilon),0,u,v),D\varphi(t_\varepsilon,x_\varepsilon)\rangle+\langle \Delta b_1,
 D\varphi(t_\varepsilon,x_\varepsilon)\rangle+\langle \Delta b_2,
 \frac{2(x_\varepsilon-y_\varepsilon)}{\varepsilon^2}\rangle,
 \end{array}$$
where $$\begin{array}{llll}\Delta
b_1&=&b(t_\varepsilon,x_\varepsilon,w_1(t_\varepsilon,x_\varepsilon),(\frac{2(x_\varepsilon-y_\varepsilon)}{\varepsilon^2}+D\varphi(t_\varepsilon,x_\varepsilon))\sigma(t_\varepsilon,x_\varepsilon,u,v),u,v)-b(t_\varepsilon,x_\varepsilon,w_1(t_\varepsilon,x_\varepsilon),0,u,v)\\
&\leq&
K_b|\frac{2(x_\varepsilon-y_\varepsilon)}{\varepsilon^2}+D\varphi(t_\varepsilon,x_\varepsilon)|\cdot|\sigma(t_\varepsilon,x_\varepsilon,u,v)|,\\
\Delta
b_2&=&b(t_\varepsilon,x_\varepsilon,w_1(t_\varepsilon,x_\varepsilon),(\frac{2(x_\varepsilon-y_\varepsilon)}{\varepsilon^2}+D\varphi(t_\varepsilon,x_\varepsilon))\sigma(t_\varepsilon,x_\varepsilon,u,v),u,v)\\
&&-b(t_\varepsilon,y_\varepsilon,w_2(t_\varepsilon,y_\varepsilon),\frac{2(x_\varepsilon-y_\varepsilon)}{\varepsilon^2}\sigma(t_\varepsilon,y_\varepsilon,u,v),u,v)\\
&\leq&
K_b\{|x_\varepsilon-y_\varepsilon|+|w_1(t_\varepsilon,x_\varepsilon)-w_2(t_\varepsilon,y_\varepsilon)|+|D\varphi(t_\varepsilon,x_\varepsilon)|\cdot|\sigma(t_\varepsilon,x_\varepsilon,u,v)|\}+K_bK_\sigma\frac{2|x_\varepsilon-y_\varepsilon|^2}{\varepsilon^2},
 \end{array}$$
therefore,
$$\begin{array}{llll}I_{2,u,v}^{\varepsilon,\alpha}&\leq& \langle
 b(t_\varepsilon,x_\varepsilon,w_1(t_\varepsilon,x_\varepsilon),0,u,v),D\varphi(t_\varepsilon,x_\varepsilon)\rangle\\
 &&+K_b|\frac{2(x_\varepsilon-y_\varepsilon)}{\varepsilon^2}+D\varphi(t_\varepsilon,x_\varepsilon)|\cdot|\sigma(t_\varepsilon,x_\varepsilon,u,v)|
 \cdot|D\varphi(t_\varepsilon,x_\varepsilon)|+K_bK_\sigma\frac{4|x_\varepsilon-y_\varepsilon|^3}{\varepsilon^4}\\
 &&+ K_b\frac{2|x_\varepsilon-y_\varepsilon|}{\varepsilon^2}\{|x_\varepsilon-y_\varepsilon|+|w_1(t_\varepsilon,x_\varepsilon)-w_2(t_\varepsilon,y_\varepsilon)|+|D\varphi(t_\varepsilon,x_\varepsilon)|\cdot|\sigma(t_\varepsilon,x_\varepsilon,u,v)|\},
 \end{array}$$
Similar to the proof of Lemma 3.7 in \cite{BBP}, we estimate the
differences of the integral-differential terms. From the fact that
$(t_\varepsilon,x_\varepsilon,y_\varepsilon)$ is a global maximum
point of $\psi_\varepsilon$ in $\bar{B}_{\frac{R}{2}}$, we know
$$\psi_\varepsilon(t_\varepsilon,x_\varepsilon+h(t_\varepsilon,x_\varepsilon,u,v,e),y_\varepsilon+h(t_\varepsilon,y_\varepsilon,u,v,e))\leq
\psi_\varepsilon(t_\varepsilon,x_\varepsilon,y_\varepsilon),$$ hence,
$$\begin{array}{llll}
&&[w_1(t_\varepsilon,x_\varepsilon+h(t_\varepsilon,x_\varepsilon,u,v,e))-w_1(t_\varepsilon,x_\varepsilon)]-[w_2(t_\varepsilon,y_\varepsilon+h(t_\varepsilon,y_\varepsilon,u,v,e))-w_2(t_\varepsilon,y_\varepsilon)]\\
&&-\langle\frac{2(x_\varepsilon-y_\varepsilon)}{\varepsilon^2},h(t_\varepsilon,x_\varepsilon,u,v,e)-h(t_\varepsilon,y_\varepsilon,u,v,e)\rangle-\frac{1}{\varepsilon^2}|h(t_\varepsilon,x_\varepsilon,u,v,e)-h(t_\varepsilon,y_\varepsilon,u,v,e)|^2\\
&&\leq
\varphi(t_\varepsilon,x_\varepsilon+h(t_\varepsilon,x_\varepsilon,u,v,e))-\varphi(t_\varepsilon,x_\varepsilon),
 \end{array}$$
furthermore,
$$\begin{array}{llll}
&&\int_{E_\delta^c}(w_1(t_\varepsilon,x_\varepsilon+h(t_\varepsilon,x_\varepsilon,u,v,e))-w_1(t_\varepsilon,x_\varepsilon)-\langle
\frac{2(x_\varepsilon-y_\varepsilon)}{\varepsilon^2}+D\varphi(t_\varepsilon,x_\varepsilon),h(t_\varepsilon,x_\varepsilon,u,v,e)
\rangle)\lambda(de)\\
&&-\int_{E^c_\delta}(w_2(t_\varepsilon,y_\varepsilon+h(t_\varepsilon,y_\varepsilon,u,v,e))-w_2(t_\varepsilon,y_\varepsilon)-\langle
\frac{2(x_\varepsilon-y_\varepsilon)}{\varepsilon^2},h(t_\varepsilon,y_\varepsilon,u,v,e)
\rangle)\lambda(de)\\
&&\leq
\int_{E_\delta^c}(\varphi(t_\varepsilon,x_\varepsilon+h(t_\varepsilon,x_\varepsilon,u,v,e))-\varphi(t_\varepsilon,x_\varepsilon)-\langle
D\varphi(t_\varepsilon,x_\varepsilon),h(t_\varepsilon,x_\varepsilon,u,v,e)
\rangle)\lambda(de)\\
&&+
\int_{E_\delta^c}\frac{|h(t_\varepsilon,x_\varepsilon,u,v,e)-h(t_\varepsilon,y_\varepsilon,u,v,e)|^2}{\varepsilon^2}\lambda(de).
 \end{array}$$
From these estimates,
$$\begin{array}{llll}
I_{3,u,v}^{\varepsilon,\alpha,\delta}&\leq&\int_{E_\delta}\frac{|h(t_\varepsilon,x_\varepsilon,u,v,e)|^2}{\varepsilon^2}\lambda(de)+\int_{E_\delta}\frac{|h(t_\varepsilon,y_\varepsilon,u,v,e)|^2}{\varepsilon^2}\lambda(de)+\int_{E_\delta^c}\frac{|h(t_\varepsilon,x_\varepsilon,u,v,e)-h(t_\varepsilon,y_\varepsilon,u,v,e)|^2}{\varepsilon^2}\lambda(de)\\
&&+\int_{E}(\varphi(t_\varepsilon,x_\varepsilon+h(t_\varepsilon,x_\varepsilon,u,v,e))-\varphi(t_\varepsilon,x_\varepsilon)-D\varphi(t_\varepsilon,x_\varepsilon)h(t_\varepsilon,x_\varepsilon,u,v,e))\lambda(de),
 \end{array}$$
as well as
$$\begin{array}{llll}
B_1^\delta-B_2^\delta
&\leq& \int_{E_\delta}(\frac{2(x_\varepsilon-y_\varepsilon)}{\varepsilon^2}h(t_\varepsilon,x_\varepsilon,u,v,e)+\frac{|h(t_\varepsilon,x_\varepsilon,u,v,e)|^2}{\varepsilon^2})l(e)\lambda(de)\\
&&+\int_{E_\delta}(\frac{2(x_\varepsilon-y_\varepsilon)}{\varepsilon^2}h(t_\varepsilon,y_\varepsilon,u,v,e)+\frac{|h(t_\varepsilon,y_\varepsilon,u,v,e)|^2}{\varepsilon^2})l(e)\lambda(de)\\
&&+\int_{E}(\varphi(t_\varepsilon,x_\varepsilon+h(t_\varepsilon,x_\varepsilon,u,v,e))-\varphi(t_\varepsilon,x_\varepsilon))l(e)\lambda(de)\\
&&+\int_{E^c_\delta}\langle
\frac{2(x_\varepsilon-y_\varepsilon)}{\varepsilon^2},h(t_\varepsilon,x_\varepsilon,u,v,e)-h(t_\varepsilon,y_\varepsilon,u,v,e)\rangle l(e)\lambda(de)\\
&&+\int_{E^c_\delta}\frac{|h(t_\varepsilon,x_\varepsilon,u,v,e)-h(t_\varepsilon,y_\varepsilon,u,v,e)|^2}{\varepsilon^2}l(e)\lambda(de).
\end{array}$$
Therefore,
$$\begin{array}{llll}
I_{4,u,v}^{\varepsilon,\alpha,\delta}&\leq&|f(t_\varepsilon,x_\varepsilon,w_1(t_\varepsilon,x_\varepsilon),(\frac{2(x_\varepsilon-y_\varepsilon)}{\varepsilon^2}+D\varphi(t_\varepsilon,x_\varepsilon))\sigma(t_\varepsilon,x_\varepsilon,u,v),B_1^\delta,u,v)\\
 &&-f(t_\varepsilon,y_\varepsilon,w_2(t_\varepsilon,y_\varepsilon),\frac{2(x_\varepsilon-y_\varepsilon)}{\varepsilon^2}\sigma(t_\varepsilon,y_\varepsilon,u,v),B_2^\delta,u,v)|\\
&\leq&
K_f|x_\varepsilon-y_\varepsilon|+K_f|w_1(t_\varepsilon,x_\varepsilon)-w_2(t_\varepsilon,y_\varepsilon)|+K_f|D\varphi(t_\varepsilon,x_\varepsilon))|\cdot|\sigma(t_\varepsilon,x_\varepsilon,u,v)|+K_fK_\sigma\frac{2|x_\varepsilon-y_\varepsilon|^2}{\varepsilon^2}\\
&&+K_f(B_1^\delta-B_2^\delta)^+.
 \end{array}$$
Since $\frac{2(x_\varepsilon-y_\varepsilon)}{\varepsilon^2}\leq
K_{w_2}$, and $D\varphi(t_\varepsilon,x_\varepsilon)\leq
K_{w_1,w_2}$,  let $\delta\rightarrow0$ with keeping
$\varepsilon,\alpha$ fixed, we  get
$$\begin{array}{llll}
&&I_{1,u,v}^{\varepsilon,\alpha}+I_{2,u,v}^{\varepsilon,\alpha}+I_{3,u,v}^{\varepsilon,\alpha,\delta}+I_{4,u,v}^{\varepsilon,\alpha,\delta}\\
&\leq&
\frac{1}{2}tr(\sigma\sigma^T(t_\varepsilon,x_\varepsilon,u,v)D^2\varphi(t_\varepsilon,x_\varepsilon))+K_\sigma^2\frac{|x_\varepsilon-y_\varepsilon|^2}{\varepsilon^2}+
\langle
 b(t_\varepsilon,x_\varepsilon,w_1(t_\varepsilon,x_\varepsilon),0,u,v),D\varphi(t_\varepsilon,x_\varepsilon)\rangle\\
 &&+\tilde{K}\{|D\varphi(t_\varepsilon,x_\varepsilon))|\cdot|\sigma(t_\varepsilon,x_\varepsilon,u,v)|+|w_1(t_\varepsilon,x_\varepsilon)-w_2(t_\varepsilon,y_\varepsilon)|\}+\tilde{K}\{|x_\varepsilon-y_\varepsilon|+|x_\varepsilon-y_\varepsilon|^2\}\\
 &&+\int_{E}\frac{|h(t_\varepsilon,x_\varepsilon,u,v,e)-h(t_\varepsilon,y_\varepsilon,u,v,e)|^2}{\varepsilon^2}\lambda(de)\\
 &&+\int_{E}(\varphi(t_\varepsilon,x_\varepsilon+h(t_\varepsilon,x_\varepsilon,u,v,e))-\varphi(t_\varepsilon,x_\varepsilon)-D\varphi(t_\varepsilon,x_\varepsilon)h(t_\varepsilon,x_\varepsilon,u,v,e))\lambda(de)\\
 &&+K_f(\int_{E}(\varphi(t_\varepsilon,x_\varepsilon+h(t_\varepsilon,x_\varepsilon,u,v,e))-\varphi(t_\varepsilon,x_\varepsilon))l(e)\lambda(de)\\
&&+\int_{E}\langle
\frac{2(x_\varepsilon-y_\varepsilon)}{\varepsilon^2},h(t_\varepsilon,x_\varepsilon,u,v,e)-h(t_\varepsilon,y_\varepsilon,u,v,e)\rangle l(e)\lambda(de)\\
&&+\int_{E}\frac{|h(t_\varepsilon,x_\varepsilon,u,v,e)-h(t_\varepsilon,y_\varepsilon,u,v,e)|^2}{\varepsilon^2}l(e)\lambda(de))^+.
 \end{array}$$
Finally, we let $\alpha\rightarrow0,\ \varepsilon\rightarrow0$,
from (ii), (iii), we get
$$\begin{array}{llll}0&\leq&\lim\limits_{\varepsilon\rightarrow0}\lim\limits_{\alpha\rightarrow0}\lim\limits_{\delta\rightarrow0}\{\frac{\partial}{\partial
t}\varphi(t_\varepsilon,x_\varepsilon)+\sup\limits_{u\in U,v\in
V}\{I_{1,u,v}^{\varepsilon,\alpha}+I_{2,u,v}^{\varepsilon,\alpha}+I_{3,u,v}^{\varepsilon,\alpha,\delta}+I_{4,u,v}^{\varepsilon,\alpha,\delta}\}\}\\
&\leq&\frac{\partial}{\partial t}\varphi(t_0,x_0)+\sup\limits_{u\in
U,v\in
V}\{\frac{1}{2}tr(\sigma\sigma^T(t_0,x_0,u,v)D^2\varphi(t_0,x_0))+\langle
b(t_0,x_0,w_1(t_0,x_0),0,u,v),D\varphi(t_0,x_0)\rangle\\
&&+\int_{E}(\varphi(t_0,x_0+h(t_0,x_0,u,v,e))-\varphi(t_0,x_0)-D\varphi(t_0,x_0)h(t_0,x_0,u,v,e))\lambda(de)\\
&&+\tilde{K}(|w_1(t_0,x_0)-w_2(t_0,x_0)|+|D\varphi(t_0,x_0)|\cdot|\sigma(t_0,x_0,u,v)|\\
&&+(\int_{E}(\varphi(t_0,x_0+h(t_0,x_0,u,v,e))-\varphi(t_0,x_0))l(e)\lambda(de))^+)\}\\
&=&\frac{\partial}{\partial t}\varphi(t_0,x_0)+\sup\limits_{u\in
U,v\in
V}\{\frac{1}{2}tr(\sigma\sigma^T(t_0,x_0,u,v)D^2\varphi(t_0,x_0))+\langle
b(t_0,x_0,w_1(t_0,x_0),0,u,v),D\varphi(t_0,x_0)\rangle\\
&&+B^{u,v}\varphi(t_0,x_0)+\tilde{K}|\varphi(t_0,x_0)|+\tilde{K}|D\varphi(t_0,x_0)|\cdot|\sigma(t_0,x_0,u,v)|+\tilde{K}(C^{u,v}\varphi(t_0,x_0))^+\}.\\
\end{array}$$
Therefore, $w$ is a viscosity subsolution of (\ref{equ5.2}).

\end{proof}

Following \cite{BBP, BLH}, we have

\begin{lemma}
\label{lem6.2} For any $\tilde{A}>0$, there exists $C_{1}>0$ such
that the function $\chi(t,x)=\exp[(C_{1}(T-t)+\tilde{A})\psi(x)]$,
with $\psi (x)=[\log((|x|^{2}+1)^{\frac{1}{2}})+1]^{2},\ x\in
\mathbb{R},$ satisfies
\begin{equation}
\label{equ5.3}
\begin{array}
[c]{llll} & \! \! \! \! \! \frac{\partial}{\partial t} \chi(t,x)
+\sup\limits_{u\in U,v\in V}\{{\frac{1}{2}}tr(\sigma
\sigma^{T}(t,x,u,v)D^{2}\chi(t,x))+D\chi(t,x).b(t,x,w_1
(t,x),0,u,v)+B^{u,v}\chi(t,x)\\
& \! \! \! \! \!
+\tilde{K}|\chi(t,x)|+\tilde{K}|D\chi(t,x).\sigma(t,x,u,v)|+\tilde{K}(C^{u,v}\chi(t,x))^{+}\}<0,\
\mbox{in } [t_{1},T]\times \mathbb{R}, &  &
\end{array}
\end{equation}
where $t_{1}=T-\frac{\tilde{A}}{C_{1}}$.
\end{lemma}
\begin{proof}
By direct calculus we deduce
$$|D\psi(x)|\leq\frac{2[\psi(x)]^\frac{1}{2}}{(|x|^2+1)^\frac{1}{2}},\ \ |D^2\psi(x)|\leq\frac{C(1+[\psi(x)]^\frac{1}{2})}{|x|^2+1},\ x\in\mathbb{R}.$$
Therefore, if $t\in[t_1,T],$
$$
\begin{array}{llll}
|D\chi(t,x)|\leq (C_1(T-t)+\tilde{A})\chi(t,x)|D\psi(x)| \leq
C\chi(t,x)\frac{[\psi(x)]^\frac{1}{2}}{(|x|^2+1)^\frac{1}{2}},
\end{array}$$
and $$|D^2\chi(t,x)|\leq C\chi(t,x)\frac{\psi(x)}{|x|^2+1}.$$ Notice
that the above estimates do not depend on $C_1$ because of the
definition of $t_1$. Then, from $\gamma$ is bounded and  $\psi$ is
Lipschitz continuous in $\mathbb{R}$, by a long but straight-forward
calculus, we get
$$\chi(t,x+h(t,x,u,v,e))-\chi(t,x)-D\chi(t,x).h(t,x,u,v,e)\leq C\chi(t,x)\frac{\psi(x)}{|x|^2+1}|h(t,x,u,v,e)|^2,$$
and $$\chi(t,x+h(t,x,u,v,e))-\chi(t,x)\leq
C\chi(t,x)\frac{[\psi(x)]^\frac{1}{2}}{(|x|^2+1)^\frac{1}{2}}|h(t,x,u,v,e)|.$$
Therefore, we have
$$
\begin{array}{llll}
&&\frac{\partial}{\partial t} \chi(t,x) +\sup\limits_{u\in U,v\in
V}\{{\frac{1}{2}}tr(\sigma
\sigma^{T}(t,x,u,v)D^{2}\chi(t,x))+D\chi.b(t,x,w_1
(t,x),0,u,v)+B^{u,v}\chi(t,x)\\
&& +\tilde{K}|\chi(t,x)|+\tilde{K}|D\chi(t,x).\sigma(t,x,u,v)|+\tilde{K}(C^{u,v}\chi(t,x))^{+}\}\\
&&\leq-\chi(t,x)\{C_1\psi(x)-C\psi(x)-C[\psi(x)]^\frac{1}{2}-C\frac{\psi(x)}{|x|^2+1}-\tilde{K}-C\tilde{K}[\psi(x)]^\frac{1}{2}-C\tilde{K}\frac{[\psi(x)]^\frac{1}{2}}{(|x|^2+1)^\frac{1}{2}}\}\\
&&\leq -\chi(t,x)\{C_1-[C+\tilde{K}]\}\psi(x)<0,\ \mbox{if }
C_1>C+\tilde{K}\ \mbox{large enough}.
\end{array}$$
\end{proof}

\begin{theorem}
Let $w_{1}$ (resp., $w_{2}$) $\in \Theta$ be a viscosity subsolution
(resp. supersolution) of equation (\ref{equ5.1}). Then, if $w_1$
(resp., $w_{2}$) is Lipschitz in $x$, uniformly in $t$, we have
\begin{equation}
\label{ee:comparsion}w_{1}(t,x)\leq w_{2}(t,x),\  \mbox{for all }
(t,x)\in[0,T]\times \mathbb{R}.
\end{equation}
\end{theorem}
\begin{proof} First we consider the case when $w_1$ and $w_2$ are bounded. Set $u:=w_1-w_2$. Theorem 4.1 in \cite{Barles-Imbert} proves a comparison
principle for bounded sub- and supersolutions of
Hamilton-Jacobi-Bellman equations with nonlocal term of type
(\ref{equ5.2}).  From Lemma \ref{le5.1}, we know  that $u$ is a
viscosity subsolution of equation (\ref{equ5.2}). On the other hand,
clearly, $\tilde{u}=0$ is a viscosity solution, hence it is also a
viscosity supersolution of equation (\ref{equ5.2}). Thus, Theorem
4.1 in \cite{Barles-Imbert} implies that $w_1-w_2=w\leq
\tilde{w}=0$, i.e., $w_1\leq w_2$ on $[0,T]\times \mathbb{R}$.
Finally, if $w_1,w_2$ are viscosity solutions of (\ref{equ5.2}),
they are both viscosity sub- and supersolution;  from the just
proved comparison result we get $w_1=w_2$.

However, under our standard assumptions, the lower value function $W$ defined by (\ref{ee2}) is not necessarily
bounded, so we still need to prove the case $w_1,w_2\in \Theta$. Set
$w:=w_1-w_2$. Then, for some $\tilde{A}>0$,
$$\lim_{|x|\rightarrow \infty}w(t,x)\exp \{-\tilde{A}[\log((|x|^2+1)^{\frac{1}{2}})]^2\}=0,$$
uniformly with respect to $t\in [0,T]$. Accordingly, for any
$\alpha>0$, $w(t,x)-\alpha \chi(t,x)$ is bounded from above in
$[t_1,T]\times \mathbb{R}$, and that
$$M:=\max \limits_{[t_1,T]\times \mathbb{R}}(w-\alpha \chi)(t,x)e^{-\tilde{K}(T-t)}$$
is achieved at some point $(t_0,x_0)\in [t_1,T]\times \mathbb{R}$
(depending on $\alpha$).

Now we consider the following two cases.

(i) We assume that: $w(t_0,x_0)\leq 0$, for any $\alpha>0$. Then,
$M\leq 0$ and $w_1(t,x)-w_2(t,x)\leq \alpha \chi(t,x)$ in
$[t_1,T]\times \mathbb{R}$. Consequently, letting $\alpha
\rightarrow 0$, we get
$$w_1(t,x)\leq w_2(t,x),\  \mbox{for all } (t,x)\in[t_1,T]\times
\mathbb{R}.$$

(ii) Suppose that there exists some $\alpha>0$ such that
$w(t_0,x_0)>0$. Notice that $w(t,x)-\alpha \chi(t,x)\leq
(w(t_0,x_0)-\alpha \chi(t_0,x_0))e^{-\tilde{K}(t-t_0)}$ in
$[t_1,T]\times \mathbb{R}$. Then, setting
$$\varphi(t,x)=\alpha \chi(t,x)+(w-\alpha \chi)(t_0,x_0)e^{-\tilde{K}(t-t_0)}$$
we get $w-\varphi \leq 0=(w-\varphi)(t_0,x_0)$ in $[t_1,T]\times
\mathbb{R}$. Due to Lemma \ref{le5.1}, $w$ is a viscosity
subsolution of (\ref{equ5.2}), we have
$$
\begin{array}{llll}
&&\frac{\partial}{\partial t}\varphi(t_0,x_0)+\sup\limits_{u\in
U,v\in
V}\{\frac{1}{2}tr(\sigma\sigma^T(t_0,x_0,u,v)D^2\varphi(t_0,x_0))+\langle
b(t_0,x_0,w_1(t_0,x_0),0,u,v),D\varphi(t_0,x_0)\rangle\\
&&+B^{u,v}\varphi(t_0,x_0)+\tilde{K}|\varphi(t_0,x_0)|+\tilde{K}|D\varphi(t_0,x_0)|\cdot|\sigma(t_0,x_0,u,v)|+\tilde{K}(C^{u,v}\varphi(t_0,x_0))^+\}
\geq 0.
\end{array}
$$
Moreover, due to our assumption that $w(t_0,x_0)>0$ and since
$w(t_0,x_0)=\varphi(t_0,x_0)$ we can replace
$\tilde{K}|\varphi(t_0,x_0)|$ by $\tilde{K}\varphi(t_0,x_0)$ in the
above formula. Then, from the definition of $\varphi$ and Lemma
\ref{lem6.2}, $$
\begin{array}{llll}
&&0\leq \alpha \{ \frac{\partial}{\partial
t}\chi(t_0,x_0)+\sup\limits_{u\in U,v\in
V}\{\frac{1}{2}tr(\sigma\sigma^T(t_0,x_0,u,v)D^2\chi(t_0,x_0))+\langle
b(t_0,x_0,w_1(t_0,x_0),0,u,v),D\chi(t_0,x_0)\rangle\\
&&+B^{u,v}\chi(t_0,x_0)+\tilde{K}|\chi(t_0,x_0)|+\tilde{K}|D\chi(t_0,x_0).\sigma(t_0,x_0,u,v)|+\tilde{K}(C^{u,v}\chi(t_0,x_0))^+
\}< 0.
\end{array}
$$
which causes a contradiction. Finally, by applying successively the
same argument on the interval $[t_2,t_1]$ with
$t_2=(t_1-\frac{\tilde{A}}{C_1})^+$, and then, if $t_2>0$ on
$[t_3,t_2]$ with $t_3=(t_2-\frac{\tilde{A}}{C_1})^+$, etc. We get
$$w_1(t,x)\leq w_2(t,x),\ (t,x)\in [0,T]\times \mathbb{R}.$$ Then,
the proof is complete.
\end{proof}

\begin{remark}
We have shown that the lower value function $W(t,x)$ is of at most
linear growth which belongs to $\Theta$, and so $W(t,x)$ is the
unique viscosity solution in $\Theta$ of equation(\ref{equ5.1}).
Similarly, we know the upper value function $U(t,x)$ is the unique
viscosity solution in $\Theta$ of the corresponding Isaacs' equation
(\ref{equ5.11}).
\end{remark}

\begin{remark}
Under the Isaacs' condition, that is, for all $(t,x)\in[0,T]\times
\mathbb{R}$,
\[
H^{-}(t,x,W(t,x),DW(t,x),D^{2}W(t,x))=H^{+}(t,x,W(t,x),DW
(t,x),D^{2}W(t,x)),
\]
the equation (\ref{equ5.1}) and (\ref{equ5.11}) coincide. From the
uniqueness in $\Theta$ of viscosity solution, the lower value
function $W(t,x)$ equals to the upper value function $U(t,x)$ which
means the associated stochastic differential game has a value.
\end{remark}

\section{Appendix: Proof of Theorem \ref{th3.1} (DPP)}

\begin{proof} For convenience, we set $$W_\delta(t,x)=\essinf_{\beta \in \mathcal {B}_{t,t+\delta}}\esssup_{u\in
\mathcal
{U}_{t,t+\delta}}G_{t,{t+\delta}}^{t,x;u,\beta(u)}[W(t+\delta,\widetilde{X}_{t+\delta}^{t,x;u,\beta(u)})].$$
We want to prove $W_\delta(t,x)$ and $W(t,x)$ coincide. For this we
only need to prove the following three lemmas. \end{proof}

\begin{lemma}
$W_{\delta}(t,x)$ is deterministic.
\end{lemma}
\noindent The proof of this lemma is similar to the proof of
Proposition \ref{pro1}, so we omit it here.

\begin{lemma}
$W_{\delta}(t,x) \leq W(t,x).$
\end{lemma}
\begin{proof}  Let $\beta \in \mathcal {B}_{t,T}$ be
arbitrarily fixed. Then, given any $u_2(\cdot)\in \mathcal
{U}_{t+\delta,T}$, we define as follows the restriction $\beta_1$ of
$\beta$ to $\mathcal {U}_{t,t+\delta}:$
$$\beta_1(u_1):=\beta(u_1\oplus u_2)|_{[t,t+\delta]},\  \  \
u_1(\cdot)\in \mathcal {U}_{t,t+\delta}, $$ where $u_1\oplus
u_2:=u_1\textbf{1}_{[t,t+\delta]}+u_2\textbf{1}_{(t+\delta,T]},$
extends $u_1(\cdot)$ to an element of $\mathcal {U}_{t,T}$. It is
easy to check that $\beta_1\in \mathcal {B}_{t,t+\delta}.$ Moreover,
from the nonanticipativity property of $\beta$ we deduce that
$\beta_1$ is independent of the special choice of $u_2(\cdot)\in
\mathcal {U}_{t+\delta,T}$. Consequently, from the definition of
$W_\delta(t,x),$
$$W_\delta(t,x)\leq \esssup_{u_1\in \mathcal
{U}_{t,t+\delta}}G_{t,{t+\delta}}^{t,x;u_1,\beta_1(u_1)}[W(t+\delta,\widetilde{X}_{t+\delta}^{t,x;u_1,\beta_1(u_1)})],\
\mbox{P-a.s.}$$ We use the notation
$I_\delta(t,x;u,v):=G_{t,{t+\delta}}^{t,x;u,v}[W(t+\delta,\widetilde{X}_{t+\delta}^{t,x;u,v})]$
and notice that there exists a sequence $\{u_i^1,i\geq 1\} \subset
\mathcal {U}_{t,t+\delta},$ such that
$$I_\delta(t,x,\beta_1):=\esssup_{u_1\in \mathcal
{U}_{t,t+\delta}}I_\delta(t,x;u_1,\beta_1(u_1))=\sup \limits _{i\geq
1}I_\delta(t,x;u_i^1,\beta_1(u_i^1)), \  \mbox{P-a.s.}$$
For any $\varepsilon>0,$ we put
$\widetilde{\Gamma}_i:=\{I_\delta(t,x,\beta_1)\leq
I_\delta(t,x;u_i^1,\beta_1(u_i^1))+\varepsilon \} \in \mathcal
{F}_t, \ i\geq 1.$ Then $\Gamma_1:=\widetilde{\Gamma}_1,\
\Gamma_i:=\widetilde{\Gamma}_i\setminus(\bigcup\limits_{l=1}^{i-1}\widetilde{\Gamma}_l)
\in \mathcal {F}_t,\ i\geq 2,$ form an $(\Omega,\mathcal
{F}_t)$-partition, and $u_1^\varepsilon:=\sum \limits_{i\geq
1}\textbf{1}_{\Gamma_i}u_i^1$ belongs obviously to $\mathcal
{U}_{t,t+\delta}.$ Moreover, from the nonanticipativity of $\beta_1$
we have $\beta_1(u_1^\varepsilon)=\sum \limits_{i\geq
1}\textbf{1}_{\Gamma_i}\beta_1(u_i^1)$, and from the uniqueness of
the solution of the fully coupled FBSDE with jumps, we deduce that
$I_\delta(t,x;u_1^\varepsilon,\beta_1(u_1^\varepsilon))=\sum\limits_{i\geq
1}\textbf{1}_{\Gamma_i}I_\delta(t,x;u_i^1,\beta_1(u_i^1)), \
\mbox{P-a.s.}$ Hence,
\begin{equation}\label{ee7.1}\begin{array}{lll} W_\delta(t,x)\leq I_\delta(t,x;\beta_1)\leq \sum \limits_{i\geq
1}\textbf{1}_{\Gamma_i}I_\delta(t,x;u_i^1,\beta_1(u_i^1))+\varepsilon=I_\delta(t,x;u_1^\varepsilon,\beta_1(u_1^\varepsilon))+\varepsilon \\
=G_{t,{t+\delta}}^{t,x;u_1^\varepsilon,\beta_1(u_1^\varepsilon)}[W(t+\delta,\widetilde{X}_{t+\delta}^{t,x;u_1^\varepsilon,\beta_1(u_1^\varepsilon)})]+\varepsilon.\end{array}\end{equation}
On the other hand, using the fact that $\beta_1(\cdot):=\beta(\cdot
\oplus u_2)\in \mathcal {B}_{t,t+\delta}$ does not depend on
$u_2(\cdot)\in \mathcal {U}_{t+\delta,T},$ we can define
$\beta_2(u_2):=\beta(u_1^\varepsilon \oplus u_2)|_{[t+\delta,T]},$
for all $u_2(\cdot)\in \mathcal {U}_{t+\delta,T}.$ Therefore, from
the definition of $W(t+\delta,y)$ we have, for any $y\in
\mathbb{R},$
$$W(t+\delta,y)\leq \esssup_{u_2\in \mathcal
{U}_{t+\delta,T}}J(t+\delta,y;u_2,\beta_2(u_2)),\  \mbox{P-a.s.}$$
Finally, because there exists a constant $C\in \mathbb{R}$ such that
\begin{equation}\label{ee7.2}\begin{array}{llll}{\rm(i)}&& |W(t+\delta,y)-W(t+\delta,y')|\leq C|y-y'| \  \mbox{for any }
y,y'\in \mathbb{R};\\
{\rm(ii)}&&
|J(t+\delta,y;u_2,\beta_2(u_2))-J(t+\delta,y';u_2,\beta_2(u_2))|\leq
C|y-y'|, \  \mbox{P-a.s.} \mbox{ for any } u_2\in  \mathcal
{U}_{t+\delta,T},\end{array}\end{equation} we can show by
approximating
$\widetilde{X}_{t+\delta}^{t,x;u_1^\varepsilon,\beta_1(u_1^\varepsilon)}$
that
$$W(t+\delta,\widetilde{X}_{t+\delta}^{t,x;u_1^\varepsilon,\beta_1(u_1^\varepsilon)})\leq \esssup_{u_2\in \mathcal
{U}_{t+\delta,T}}J(t+\delta,\widetilde{X}_{t+\delta}^{t,x;u_1^\varepsilon,\beta_1(u_1^\varepsilon)};u_2,\beta_2(u_2)),\
\mbox{P-a.s.}$$ To estimate the right side of the latter inequality
we note that there exists some sequence $\{u_j^2,j\geq1\} \subset
\mathcal {U}_{t+\delta,T}$ such that
$$\esssup_{u_2\in \mathcal
{U}_{t+\delta,T}}J(t+\delta,\widetilde{X}_{t+\delta}^{t,x;u_1^\varepsilon,\beta_1(u_1^\varepsilon)};u_2,\beta_2(u_2))=\sup
\limits
_{j\geq1}J(t+\delta,\widetilde{X}_{t+\delta}^{t,x;u_1^\varepsilon,\beta_1(u_1^\varepsilon)};u_j^2,\beta_2(u_j^2)),
\  \mbox{P-a.s.}$$ Then, putting $\widetilde{\Delta}_j:=\{ \esssup
\limits_{u_2\in \mathcal
{U}_{t+\delta,T}}J(t+\delta,\widetilde{X}_{t+\delta}^{t,x;u_1^\varepsilon,\beta_1(u_1^\varepsilon)};u_2,\beta_2(u_2))\leq
J(t+\delta,\widetilde{X}_{t+\delta}^{t,x;u_1^\varepsilon,\beta_1(u_1^\varepsilon)};u_j^2,\beta_2(u_j^2))+\varepsilon
\} \in \mathcal {F}_{t+\delta},\ j\geq1;$ we have with
$\Delta_1:=\widetilde{\Delta}_1,\
\Delta_j:=\widetilde{\Delta}_j\setminus(\bigcup\limits_{l=1}^{j-1}\widetilde{\Delta}_l)\in
\mathcal {F}_{t+\delta},\ j\geq2,$ an $(\Omega,\mathcal
{F}_{t+\delta})$-partition and $u_2^\varepsilon:=\sum
\limits_{j\geq1}\textbf{1}_{\Delta_j}u_j^2\in \mathcal
{U}_{t+\delta,T}.$ From the nonanticipativity of $\beta_2$ we have
$\beta_2(u_2^\varepsilon)=\sum
\limits_{j\geq1}\textbf{1}_{\Delta_j}\beta_2(u_j^2),$ and from the
definition of $\beta_1, \  \beta_2 $, we know that
$\beta(u_1^\varepsilon \oplus
u_2^\varepsilon)=\beta_1(u_1^\varepsilon)\oplus
\beta_2(u_2^\varepsilon).$ Thus, from the uniqueness of the solution
of fully coupled FBSDE with jumps, we get
$$\begin{array}{llll}
J(t+\delta,\widetilde{X}_{t+\delta}^{t,x;u_1^\varepsilon,\beta_1(u_1^\varepsilon)};u_2^\varepsilon,\beta_2(u_2^\varepsilon))&=&Y_{t+\delta}^{t+\delta,\widetilde{X}_{t+\delta}^{{t,x;u_1^\varepsilon,\beta_1(u_1^\varepsilon)}};u_2^\varepsilon,\beta_2(u_2^\varepsilon)}=\sum \limits_{j\geq1}\textbf{1}_{\Delta_j}Y_{t+\delta}^{t+\delta,\widetilde{X}_{t+\delta}^{t,x;u_1^\varepsilon,\beta_1(u_1^\varepsilon)};u_j^2,\beta_2(u_j^2)}\\
&=&\sum
\limits_{j\geq1}\textbf{1}_{\Delta_j}J(t+\delta,\widetilde{X}_{t+\delta}^{t,x;u_1^\varepsilon,\beta_1(u_1^\varepsilon)};u_j^2,\beta_2(u_j^2)),\
\  \mbox{P-a.s.}
\end{array}$$
Consequently, \begin{equation}\label{ee7.3}
\begin{array}{llll}
W(t+\delta,\widetilde{X}_{t+\delta}^{t,x;u_1^\varepsilon,\beta_1(u_1^\varepsilon)})&
\leq& \esssup \limits_{u_2\in \mathcal
{U}_{t+\delta,T}}J(t+\delta,\widetilde{X}_{t+\delta}^{t,x;u_1^\varepsilon,\beta_1(u_1^\varepsilon)};u_2,\beta_2(u_2)) \\
&\leq
&\sum \limits_{j\geq1}\textbf{1}_{\Delta_j}J(t+\delta,\widetilde{X}_{t+\delta}^{t,x;u_1^\varepsilon,\beta_1(u_1^\varepsilon)};u_1^\varepsilon \oplus u_j^2,\beta(u_1^\varepsilon \oplus u_j^2))+\varepsilon  \\
&=&J(t+\delta,\widetilde{X}_{t+\delta}^{t,x;u_1^\varepsilon,\beta_1(u_1^\varepsilon)};u_1^\varepsilon \oplus u_2^\varepsilon,\beta(u_1^\varepsilon \oplus u_2^\varepsilon))+\varepsilon  \\
&=&J(t+\delta,\widetilde{X}_{t+\delta}^{t,x;u_1^\varepsilon,\beta_1(u_1^\varepsilon)};u^\varepsilon,\beta(u^\varepsilon))
+\varepsilon,\  \  \mbox{P-a.s.},
\end{array}\end{equation}
where $u^\varepsilon:=u_1^\varepsilon \oplus u_2^\varepsilon \in
\mathcal {U}_{t,T}.$ From (\ref{ee7.1}), (\ref{ee7.3}) and the
comparison theorem for fully coupled FBSDE with jumps (Theorem
3.2 in \cite{LiWei-Lp}), we have for $0<\delta<\delta_0$ sufficiently small
\begin{equation}\label{ee7.4}
\begin{array}{llll}
W_\delta(t,x)& \leq&
G_{t,t+\delta}^{t,x;u_1^\varepsilon,\beta_1(u_1^\varepsilon)}[J(t+\delta,\widetilde{X}_{t+\delta}^{t,x;u_1^\varepsilon,\beta_1(u_1^\varepsilon)};u^\varepsilon,\beta(u^\varepsilon))
+\varepsilon]+\varepsilon \\
&\leq
&G_{t,t+\delta}^{t,x;u_1^\varepsilon,\beta_1(u_1^\varepsilon)}[J(t+\delta,\widetilde{X}_{t+\delta}^{t,x;u_1^\varepsilon,\beta_1(u_1^\varepsilon)};u^\varepsilon,\beta(u^\varepsilon))
] +(C+1)\varepsilon \\
&=&G_{t,t+\delta}^{t,x;u^\varepsilon,\beta(u^\varepsilon)}[J(t+\delta,\widetilde{X}_{t+\delta}^{t,x;u_1^\varepsilon,\beta_1(u_1^\varepsilon)};u^\varepsilon,\beta(u^\varepsilon))
] +(C+1)\varepsilon \\
&=&J(t,x;u^\varepsilon,\beta(u^\varepsilon)+(C+1)\varepsilon =Y_{t}^{t,x;u^\varepsilon,\beta(u^\varepsilon)} +(C+1)\varepsilon\\
&\leq &\esssup \limits_{u\in \mathcal
{U}_{t,T}}Y_t^{t,x;u,\beta(u)}+(C+1)\varepsilon \  \  \mbox{P-a.s.,}
\end{array}\end{equation}
where we have used also Remark 3.5 in \cite{LiWei-Lp}. Indeed, from the
definition of our stochastic backward semigroup we have
$$G_{t,t+\delta}^{t,x;u_1^\varepsilon,\beta_1(u_1^\varepsilon)}[J(t+\delta,\widetilde{X}_{t+\delta}^{t,x;u_1^\varepsilon,\beta_1(u_1^\varepsilon)};u^\varepsilon,\beta(u^\varepsilon))
+\varepsilon]=\widehat{Y}_s^{t,x;u_1^\varepsilon,\beta_1(u_1^\varepsilon)},\
s\in[t,t+\delta],$$ where
$(\widehat{\Pi}_s^{t,x;u_1^\varepsilon,\beta_1(u_1^\varepsilon)},\widehat{K}_s^{t,x;u_1^\varepsilon,\beta_1(u_1^\varepsilon)}):=(\widehat{X}_s^{t,x;u_1^\varepsilon,\beta_1(u_1^\varepsilon)},\widehat{Y}_s^{t,x;u_1^\varepsilon,\beta_1(u_1^\varepsilon)},\widehat{Z}_s^{t,x;u_1^\varepsilon,\beta_1(u_1^\varepsilon)},\widehat{K}_s^{t,x;u_1^\varepsilon,\beta_1(u_1^\varepsilon)})$
is the solution of the following fully coupled FBSDEs with jumps:
\begin{equation}\label{ee7.5} \left \{
\begin{array}{llll}
d\widehat{X}_s^{t,x;u_1^\varepsilon,\beta_1(u_1^\varepsilon)} & = & b(s,\widehat{\Pi}_s^{t,x;u_1^\varepsilon,\beta_1(u_1^\varepsilon)},u^\varepsilon_1(s),\beta_1(u_1^\varepsilon)(s))ds +\sigma(s,\widehat{\Pi}_s^{t,x;u_1^\varepsilon,\beta_1(u_1^\varepsilon)},u^\varepsilon_1(s),\beta_1(u_1^\varepsilon)(s)) dB_s \\
&&
+h(s,\widehat{\Pi}_{s-}^{t,x;u_1^\varepsilon,\beta_1(u_1^\varepsilon)},u^\varepsilon_1(s),\beta_1(u_1^\varepsilon)(s))\tilde{\mu}(ds,de), \  \  \  \  \ s\in[t,t+\delta],\\
d\widehat{Y}_s^{t,x;u_1^\varepsilon,\beta_1(u_1^\varepsilon)} & = & -f(s,\widehat{\Pi}_s^{t,x;u_1^\varepsilon,\beta_1(u_1^\varepsilon)},\int_E\widehat{K}_s^{t,x;u_1^\varepsilon,\beta_1(u_1^\varepsilon)}l(e)\lambda(de),u^\varepsilon_1(s),\beta_1(u_1^\varepsilon)(s))ds+\widehat{Z}_s^{t,x;u_1^\varepsilon,\beta_1(u_1^\varepsilon)}dB_s\\
&&+ \int_E\widehat{K}_s^{t,x;u_1^\varepsilon,\beta_1(u_1^\varepsilon)}\tilde{\mu}(ds,de), \\
\widehat{X}_t^{t,x;u_1^\varepsilon,\beta_1(u_1^\varepsilon)}& = & x,\\
\widehat{Y}_T^{t,x;u_1^\varepsilon,\beta_1(u_1^\varepsilon)} & = &
J(t+\delta,\widetilde{X}_{t+\delta}^{t,x;u_1^\varepsilon,\beta_1(u_1^\varepsilon)};u^\varepsilon,\beta(u^\varepsilon))
+\varepsilon.
\end{array}
\right.
\end{equation}
Since $\beta \in \mathcal {B}_{t,T}$ has been arbitrarily chosen we
have (\ref{ee7.5}) for all $\beta \in \mathcal {B}_{t,T}.$
Therefore,
\begin{equation} W_\delta(t,x) \leq \essinf_{\beta \in \mathcal
{B}_{t,T}}\esssup_{u\in \mathcal
{U}_{t,T}}Y_t^{t,x;u,\beta(u)}+(C+1)\varepsilon=W(t,x)+(C+1)\varepsilon.
\end{equation} Finally, letting $\varepsilon \downarrow 0,$ we get
$W_\delta(t,x)\leq W(t,x).$\\ \end{proof}

\begin{lemma}
$W(t,x)\leq W_{\delta}(t,x).$
\end{lemma}

\begin{proof} We continue to use the notations introduced above. From the
definition of $W_\delta(t,x)$ we have \begin{equation}\label{ee7.6}
\begin{array}{llll} W_\delta(t,x)&=& \essinf \limits_{\beta_1\in \mathcal
{B}_{t,t+\delta}}\esssup \limits_{u_1\in \mathcal
{U}_{t,t+\delta}}G_{t,t+\delta}^{t,x;u_1,\beta_1(u_1)}[W(t+\delta,\widetilde{X}_{t+\delta}^{t,x;u_1,\beta_1(u_1)})]=\essinf \limits_{\beta_1\in \mathcal
{B}_{t,t+\delta}}I_\delta(t,x,\beta_1),\end{array}\end{equation} and
for some sequence $\{ \beta_i^1,i\geq1\} \subset \mathcal
{B}_{t,t+\delta}$,
$W_\delta(t,x)=\inf\limits_{i\geq1}I_\delta(t,x,\beta_i^1), \
\mbox{P-a.s.}$

 For any $\varepsilon>0,$ we put
$\widetilde{\Pi}_i:=\{I_\delta(t,x,\beta_i^1)-\varepsilon \leq
W_\delta(t,x)\} \in \mathcal {F}_t,\ i\geq1,\
\Lambda_1:=\widetilde{\Pi}_1$ and
$\Lambda_i:=\widetilde{\Pi}_i\backslash(\bigcup\limits_{l=1}^{i-1}\widetilde{\Pi}_l)\in
\mathcal {F}_t,\ i\geq 2.$ Then, $\{ \Lambda_i,\ i\geq 1\}$ is an
$(\Omega,\mathcal {F}_t)$-partition,
$\beta_1^\varepsilon:=\sum\limits_{i\geq
1}\textbf{1}_{\Lambda_i}\beta_i^1$ belongs to $\mathcal
{B}_{t,t+\delta}$, and from the uniqueness of the solution of our
fully coupled FBSDE with jumps, we conclude that
$I_\delta(t,x,u_1,\beta_1^\varepsilon(u_1))=\sum\limits_{i\geq
1}\textbf{1}_{\Lambda_i}I_\delta(t,x,u_1,\beta_i^1(u_1)),\
\mbox{P-a.s.},$ for all $u_1(\cdot) \in \mathcal {U}_{t,t+\delta}.$
Moreover,
\begin{equation}\label{ee7.7} \begin{array}{llll}W_\delta(t,x)&\geq  & \sum \limits_{i\geq
1}\textbf{1}_{\Lambda_i}I_\delta(t,x,\beta_i^1)-\varepsilon \geq  \sum \limits_{i\geq
1}\textbf{1}_{\Lambda_i}I_\delta(t,x,u_1,\beta_i^1(u_1))-\varepsilon
=I_\delta(t,x,u_1,\beta_1^\varepsilon(u_1))-\varepsilon \\
&=&G_{t,t+\delta}^{t,x,u_1,\beta_1^\varepsilon(u_1)}[W(t+\delta,\widetilde{X}_{t+\delta}^{t,x;u_1,\beta_1^\varepsilon(u_1)})]-\varepsilon,\
\mbox{P-a.s.},\  \mbox{for all }u_1\in \mathcal
{U}_{t,t+\delta}.\end{array}
\end{equation}
On the other hand, from the definition of $W(t+\delta,y)$, with the
same technique as before, we deduce that, for any $y\in \mathbb{R}$,
there exists $\beta_y^\varepsilon \in \mathcal {B}_{t+\delta,T}$
such that \begin{equation}\label{ee7.8} W(t+\delta,y)\geq
\esssup_{u_2\in \mathcal
{U}_{t+\delta,T}}J(t+\delta,y;u_2,\beta_y^\varepsilon(u_2))-\varepsilon,\
\mbox{P-a.s.}\end{equation} Let $\{O_i\}_{i\geq1}\subset \mathcal
{B}(\mathbb{R})$ be a decomposition of $\mathbb{R}$ such that $\sum
\limits_{i\geq1}O_i=\mathbb{R}$ and diam$(O_i)\leq \varepsilon,\
i\geq1.$ Let $y_i $ be an arbitrarily fixed element of $O_i,\
i\geq1.$ Defining
$[\widetilde{X}_{t+\delta}^{t,x;u_1,\beta_1^\varepsilon(u_1)}]:=\sum
\limits_{i\geq1}y_i\textbf{1}_{\{
\widetilde{X}_{t+\delta}^{t,x;u_1,\beta_1^\varepsilon(u_1)}\in
O_i\}},$ we have \begin{equation}\label{ee7.9}
|\widetilde{X}_{t+\delta}^{t,x;u_1,\beta_1^\varepsilon(u_1)}-[\widetilde{X}_{t+\delta}^{t,x;u_1,\beta_1^\varepsilon(u_1)}]|
\leq \varepsilon,\  \mbox{everywhere on }\Omega,\  \mbox{for all
}u_1\in \mathcal {U}_{t,t+\delta}.\end{equation} Moreover, for each
$y_i$, there exists some $\beta_{y_i}^\varepsilon \in \mathcal
{B}_{t+\delta,T}$ such that (\ref{ee7.8}) holds, and, clearly,\\
$\beta_{u_1}^\varepsilon:=\sum\limits_{i\geq1}\textbf{1}_{\{
\widetilde{X}_{t+\delta}^{t,x;u_1,\beta_1^\varepsilon(u_1)}\in
O_i\}}\beta_{y_i}^\varepsilon \in \mathcal {B}_{t+\delta,T}.$

 Now
we can define the new strategy
$\beta^\varepsilon(u):=\beta_1^\varepsilon(u_1)\oplus
\beta_1^\varepsilon(u_2), \ u\in \mathcal {U}_{t,T},$ where
$u_1=u|_{[t,t+\delta]},\ u_2=u|_{(t+\delta,T]}$ (restriction of $u$
to $[t,t+\delta]\times \Omega$ and $(t+\delta,T]\times \Omega,$
resp.). Obviously, $\beta^\varepsilon$ maps $\mathcal {U}_{t,T}$
into $\mathcal {V}_{t,T}.$ Moreover, $\beta^\varepsilon$ is
nonanticipating: Indeed, let $S$ $:\Omega \rightarrow[t,T]$ be an
$\mathbb{F}-$stopping time and $u,\ u'\  \in \mathcal {U}_{t,T}$ be
such that $u\equiv u'$ on $[[ t,S]]$. Decomposing $u,\ u'$ into
$u_1,\ u'_1\in \mathcal {U}_{t,t+\delta},\ u_2,\ u'_2\in \mathcal
{U}_{t+\delta,T}$ such that $u=u_1\oplus u'_1$ and $u=u_2\oplus
u'_2,$ we have $u_1\equiv u'_1$ on $[[t,S\wedge(t+\delta)]]$, and
hence, we get $\beta_1^\varepsilon(u_1)=\beta_1^\varepsilon(u'_1)$
on $[[t,S\wedge(t+\delta)]]$ (recall that $\beta_1^\varepsilon$ is
nonanticipating). On the other hand, $u_2\equiv u'_2$ on
$[[t+\delta, S\vee
(t+\delta)]](\subset(t+\delta,T]\times\{S>t+\delta\})$, and on
$\{S>t+\delta\}$ we have
$X_{t+\delta}^{t,x;u_1,\beta_1^\varepsilon(u_1)}=X_{t+\delta}^{t,x;u'_1,\beta_{1}^\varepsilon(u'_1)}.$
Consequently, from our definition,
$\beta_{u_1}^\varepsilon=\beta_{u'_1}^\varepsilon$ on $\{S>t+\delta
\}$ and
$\beta_{u_1}^\varepsilon(u_2)=\beta_{u'_1}^\varepsilon(u'_2)$ on
$]]t+\delta,S\vee(t+\delta)]]$. This yields
$\beta^\varepsilon(u)=\beta_1^\varepsilon(u_1)\oplus
\beta_{u_1}^\varepsilon(u_2)=\beta_1^\varepsilon(u'_1)\oplus
\beta_{u'_1}^\varepsilon(u'_2)$ on $[[t,S]]$, from which it follows
that $\beta^\varepsilon \in \mathcal {B}_{t,T}.$ Let now $u\in
\mathcal {U}_{t,T}$ be arbitrarily chosen and decomposed into
$u_1=u|_{(t,t+\delta]}\in \mathcal {U}_{t,t+\delta},\ and\
u_2=u|_{[t+\delta,T]}\in \mathcal {U}_{t+\delta,T}.$ Then from
(\ref{ee7.7}), (\ref{ee7.2})-(i), (\ref{ee7.9}) and the comparison
theorem, we obtain
\begin{equation}\label{ee7.10}\begin{array}{llll} W_\delta(t,x)
&\geq&
G_{t,t+\delta}^{t,x;u_1,\beta_1^\varepsilon(u_1)}[W(t+\delta,\widetilde{X}_{t+\delta}^{t,x;u_1,\beta_1^\varepsilon(u_1)})]-\varepsilon \\
& \geq & G_{t,t+\delta}^{t,x;u_1,\beta_1^\varepsilon(u_1)}[W(t+\delta,[\widetilde{X}_{t+\delta}^{t,x;u_1,\beta_1^\varepsilon(u_1)}])-C\varepsilon]-\varepsilon \\
& \geq & G_{t,t+\delta}^{t,x;u_1,\beta_1^\varepsilon(u_1)}[W(t+\delta,[\widetilde{X}_{t+\delta}^{t,x;u_1,\beta_1^\varepsilon(u_1)}])]-C\varepsilon \\
& = & G_{t,t+\delta}^{t,x;u_1,\beta_1^\varepsilon(u_1)}[\sum
\limits_{i\geq 1}\textbf{1}_{\{
\widetilde{X}_{t+\delta}^{t,x;u_1,\beta_1^\varepsilon(u_1)}\in
O_i\}}W(t+\delta,y_i)]-C\varepsilon,\  \  \mbox{P-a.s.}\end{array}
\end{equation} Furthermore, from (\ref{ee7.10}), (\ref{ee7.2})-(ii), (\ref{ee7.8}) and the
comparison theorem (Theorem 3.3 in \cite{LiWei-Lp}), we have
\begin{equation}\label{ee7.11}\begin{array}{llll} W_\delta(t,x)
&\geq&
G_{t,t+\delta}^{t,x;u_1,\beta_1^\varepsilon(u_1)}[\sum\limits_{i\geq
1}\textbf{1}_{\{
\widetilde{X}_{t+\delta}^{t,x;u_1,\beta_1^\varepsilon(u_1)}\in
O_i\}}J(t+\delta,y_i;u_2,\beta_{y_i}^\varepsilon(u_2))-\varepsilon]-C\varepsilon \\
& \geq &
G_{t,t+\delta}^{t,x;u_1,\beta_1^\varepsilon(u_1)}[\sum\limits_{i\geq
1}\textbf{1}_{\{
\widetilde{X}_{t+\delta}^{t,x;u_1,\beta_1^\varepsilon(u_1)}\in
O_i\}}J(t+\delta,y_i;u_2,\beta_{y_i}^\varepsilon(u_2))]-C\varepsilon \\
& = & G_{t,t+\delta}^{t,x;u_1,\beta_1^\varepsilon(u_1)}[J(t+\delta,[\widetilde{X}_{t+\delta}^{t,x;u_1,\beta_1^\varepsilon(u_1)}];u_2,\beta_{u_1}^\varepsilon(u_2))]-C\varepsilon \\
& \geq &
G_{t,t+\delta}^{t,x;u_1,\beta_1^\varepsilon(u_1)}[J(t+\delta,\widetilde{X}_{t+\delta}^{t,x;u_1,\beta_1^\varepsilon(u_1)};u_2,\beta_{u_1}^\varepsilon(u_2))-C\varepsilon]-C\varepsilon \\
& \geq &
G_{t,t+\delta}^{t,x;u_1,\beta_1^\varepsilon(u_1)}[J(t+\delta,X_{t+\delta}^{t,x;u_1,\beta_1^\varepsilon(u_1)};u_2,\beta_{u_1}^\varepsilon(u_2))]-C\varepsilon \\
& = &G_{t,t+\delta}^{t,x;u,\beta^\varepsilon(u)}[Y_{t+\delta}^{t,x;u,\beta^\varepsilon(u)}]-C\varepsilon = Y_t^{t,x;u,\beta^\varepsilon(u)}-C\varepsilon,\  \
\mbox{\mbox{P-a.s.},\  \ for any }u\in \mathcal {U}_{t,T}.
\end{array} \end{equation}
Consequently, \begin{equation}\label{ee7.12}\begin{array}{llll}
W_\delta(t,x) &\geq&
\esssup \limits_{u\in \mathcal {U}_{t,T}}J(t,x;u,\beta^\varepsilon(u))-C\varepsilon \geq \essinf \limits_{\beta \in \mathcal {B}_{t,T}}\esssup \limits_{u\in \mathcal {U}_{t,T}}J(t,x;u,\beta(u))-C\varepsilon \\
&=&W(t,x)-C\varepsilon,\  \  \mbox{P-a.s.} \end{array}
\end{equation} Finally, letting $\varepsilon \downarrow 0$ we get
$W_\delta(t,x)\geq W(t,x).$\\
\end{proof}

\begin{remark}
\label{re7.1} {\rm(a)} {\rm(i)} For every $\beta \in \mathcal{B}%
_{t,t+\delta}$, there exists some
$u^{\varepsilon}(\cdot)\in{\mathcal{U}}_{t, t+\delta}$ such that
\begin{equation}
\label{ee7.13}W(t,x)(=W_{\delta}(t, x))\leq G^{t,x; u^{\varepsilon}%
,\beta(u^{\varepsilon})}_{t,t+\delta} [W(t+\delta, X^{t,x;
u^{\varepsilon
},\beta(u^{\varepsilon})}_{t+\delta})]+\varepsilon,\quad \mbox{{\it
P}-a.s.}
\end{equation}
{\rm(ii)} There exists some $\beta^{\varepsilon}(\cdot) \in{\mathcal{B}%
}_{t, t+\delta}$ such that, for all $u(\cdot)\in{\mathcal{U}}_{t,
t+\delta}$
\begin{equation}
\label{ee7.14}W(t,x)(=W_{\delta}(t, x))\geq G^{t,x; u^{\varepsilon}%
}_{t,t+\delta} [W(t+\delta, X^{t,x;
u,\beta^{\varepsilon}(u)}_{t+\delta })]-\varepsilon,\  \mbox{{\it
P}-a.s.}
\end{equation}
{\rm(b)} From Proposition \ref{pro1}, we know that the lower value
function $W$ is deterministic. So, by choosing $\delta=T-t$ and
taking expectation on both sides of (\ref{ee7.13}), (\ref{ee7.14}),
we get $W(t,x)=\inf \limits_{\beta
\in \mathcal{B}_{t,T}}\mathop{\rm sup}\limits_{u\in \mathcal{U}_{t,T}%
}E[J(t,x;u,\beta(u))].$
\end{remark}

\end{document}